\documentclass[11pt,twoside]{article}
\usepackage{amsmath,amssymb,epsfig,subfigure,graphicx,epstopdf,color}
\usepackage{multirow}
\usepackage{algorithm}
\usepackage{algorithmic}
\usepackage{enumerate}
\usepackage{rotating,booktabs}
\usepackage[colorlinks=true]{hyperref}
\hypersetup{urlcolor=blue,citecolor=blue,linkcolor=blue}

\newtheorem{proposition}{Proposition}[section]
\newtheorem{theorem}[proposition]{Theorem}
\newtheorem{lemma}[proposition]{Lemma}

\newtheorem{remark}[proposition]{Remark}
\newtheorem{example}[proposition]{Example}

\newenvironment{proof}{{\noindent \em Proof.}}{\hfill $\fbox{}$ \vspace*{5mm}}
\numberwithin{equation}{section}
\numberwithin{algorithm}{section}
\numberwithin{table}{section}
\numberwithin{figure}{section}

\newcommand{\ba}{{\bf a}}
\newcommand{\bb}{{\bf b}}
\newcommand{\bfe}{{\bf e}}

\newcommand{\bfo}{{\bf 1}}
\newcommand{\bg}{{\bf g}}

\newcommand{\bu}{{\bf u}}

\newcommand{\bc}{{\bf c}}

\newcommand{\bv}{{\bf v}}
\newcommand{\la}{\lambda}
\newcommand{\La}{\Lambda}

\newcommand{\co}{\mathcal{O}}

\newcommand{\fl}{{\rm fl}}

\newcommand{\diag}{{\rm diag}}

\newcommand{\off}{\mathrm{off}}
\newcommand{\rank}{{\rm rank}}
\newcommand{\op}{{\rm op}}

\newcommand{\R}{{\mathbb R}}

\newcommand{\Rn}{{\mathbb R}^n}

\newcommand{\Rnn}{{\mathbb R}^{n\times n}}

\newcommand{\Rmn}{{\mathbb R}^{m\times n}}
\newcommand{\Rmm}{{\mathbb R}^{m\times m}}

\newcommand{\BE}{\begin{equation}}
\newcommand{\EE}{\end{equation}}

\begin{document}

\title{\bf A mixed precision preconditioned Jacobi method for the symmetric eigenvalue problem}
\author{Zhiyuan Zhang\thanks{School of Mathematical Sciences, Xiamen University, Xiamen 361005, People's Republic of China (zyzhang510zg@stu.xmu.edu.cn).}
\and Zheng-Jian Bai\thanks{Corresponding author. School of Mathematical Sciences,  Xiamen University, Xiamen 361005, People's Republic of China (zjbai@xmu.edu.cn). The research of this author was partially supported by the National Natural Science Foundation of China grant 12371382.}
}
\maketitle

\begin{abstract}
The eigenvalue problem is a fundamental problem in scientific computing. In this paper, we first give the error analysis for  a single step or sweep of Jacobi's method in floating point arithmetic. Then we propose a mixed precision preconditioned Jacobi method for the symmetric eigenvalue problem: We first compute the eigenvalue decomposition of a real symmetric matrix by an eigensolver at low precision and we obtain a low-precision matrix of eigenvectors; Then by using the high-precision modified Gram-Schmidt orthogonalization process, a high-precision orthogonal matrix is obtained,  which is used as an initial guess for Jacobi's method. The rounding error analysis of  the proposed  method is established under some conditions.  We also present a mixed precision preconditioned  one-sided Jacobi method for the singular value problem and the corresponding rounding error analysis is discussed. Numerical experiments on CPUs and GPUs are reported to illustrate the efficiency of the proposed method over the original  Jacobi method.

\vspace{3mm}
\end{abstract}

\vspace{3mm}

\section{Introduction} \label{Sec:Int}
The symmetric eigenvalue problem has widespread applications in scientific computing such as engineering computing \cite{Datta-1995,Yang-2005}, numerical partial differential equations \cite{Larsson.Thomee-2003}, and computing chemistry \cite{Davidson-1975}, etc.

The solution strategy of the symmetric eigenvalue problem  depends on the structure of the given symmetric matrix and the desirable eigenvalues with or without associated eigenvectors. For example, when  only a few extreme eigenvalues of a large, sparse, and symmetric  matrix are desired, Lanczos method and Jacobi-Davidson method \cite{Golub.VanLoan-2013, Sleijpen.VanderVorst-2000} are recommended. In general, a symmetric matrix can be reduced to tridiagonal form by finite Householder reflections or Givens rotations and there are some popular tridiagonalization based  strategies for the symmetric eigenvalue problem \cite{Golub.VanLoan-2013,Parlett-1998}, e.g., the symmetric QR algorithm \cite{Bowdler.Martin.Reinsch.Wilkinson-1968}, the divide-and-conquer method \cite{Gu.Eisenstat-1995}, the bisection method \cite{Barth.Martin.Wilkinson-1967}, Sturm sequence method \cite{Gupta-1972} and the method of multiple relatively robust representations (MR$^3$) \cite{Dhillon.Parlett-2004}.  Compared with these tridiagonalization based algorithm, there is another method directly applied to the original real symmetric matrix, i.e., Jacobi's method. The Jacobi method is a very old method for diagonalizing a real symmetric matrix \cite{Jacobi-1846}. Recently, the Jacobi method has received much attention due to its natural suitability for parallel computation \cite{Becka.Oksa-2020} and high accuracy in finite precision \cite{Demmel.Veselic-1992}. Another interesting aspect of Jacobi's method with proper procedure ordering shows sweep-quadratic convergence rate after sufficient iterations \cite{Ruhe-1968,vanKempen-1966}. However, the Jacobi method is slower than a method based on tridiagonalization since it is conjectured that $\co (n^3 \log n)$ operations are required for its standard implementation \cite{Slapnicar-2014}. For an overview of the symmetric eigenvalue problem, one may refer to \cite{Parlett-1998,Wilkinson-1965} and \cite[Chap.7, Chap,8, Chap.11]{Golub.VanLoan-2013}.

To improve the  efficiency of a numerical solver, in many engineering applications, one of the emerging strategies is to combine different precision arithmetics \cite{Baboulin.Buttari.Dongarra.Kurzak.Langou.Langou.Luszczek.Tomov-2009}. In the past decades, the most common IEEE 754 floating-point arithmetic in scientific computing has mainly been carried out in double precision (64 bit) and single precision (32 bit) \cite{IEEE-2008}. Theoretically, single precision runs twice as fast as double precision in both communication and computation cost. And these two formats are supported by most of hardware architectures \cite{Abdelfattah.Anzt.Boman-2021}. Recently, half precision (16-bit) floating point arithmetic has  gradually been  popular in the machine learning community.
Half precision arithmetic is already available in some hardware (e.g., the NVIDIA V100 GPU), which runs faster in machine learning applications and also reduces memory storage and energy consumption. For higher precisions, there exists quadruple precision (128 bit) in some softwares \cite{Higham.Mary-2021} such as Advanpix Multiprecision Computing Toolbox for {\tt MATLAB} \cite{software-MCT}.

Recently, based on different floating point precisions, many mixed precision algorithms have been proposed \cite{Carson.Higham-2018, Yamazaki.Tomov.Dongarra-2015, Yang.Fox.Sanders-2021}. For mixed precision algorithms in numerical linear algebra, there are two survey papers \cite{Abdelfattah.Anzt.Boman-2021,Higham.Mary-2021}.
For the general eigenvalue problem, an earlier work given by Dongarra et al. \cite{Dongarra-1982,Dongarra.Moler.Wilkinson-1983} was the mixed precision iterative refinement based on Newton's method for computing eigenpairs of a matrix, which was extended to solving the symmetric eigenvalue problem by using the Sherman-Morrison formula  \cite{Tsai-2020}.
For the symmetric eigenvalue problem,  Petschow et al. \cite{Petschow.Quintana-Orti.Bientinesi-2014} proposed a mixed precision MR$^3$-based eigensolver with improved accuracy and negligible performance penalty. Ogita and Aishima \cite{Ogita.Aishima-2018,Ogita.Aishima-2019} developed another novel iterative refinement for the symmetric eigenvalue decomposition. Very most recently, Gao et al. \cite{Gao.Ma.Shao-2022} proposed an elaborate mixed precision Jacobi singular value decomposition (SVD) algorithm which can achieve about 2x speedup comparable to LAPACK in x86-64 architecture.

The rounding error analysis has undisputed importance in numerical analysis, especially with the rise of mixed-precision computations in scientific computing. There exist some theoretical results on the error analysis of Jacobi's method.  In \cite[p.279]{Wilkinson-1965},   it was showed that the computed diagonal entries of the updated matrix after some sweeps are closed to the eigenvalues of the original symmetric matrix $A$ with an error bound proportional to the product of machine precision and the norm of $A$. In  \cite{Barlow.Demmel-1990, Demmel.Veselic-1992},   it was established that  the  Jacobi method can compute the eigenvalues of  a real symmetric positive definite diagonal scaling matrix with a uniformly relative  accuracy bound (see \cite{Mathias-1995} for extended error analysis results). In  \cite{Davies.Higham.Tisseur-2001}, a backward error analysis was provided for the Cholesky–Jacobi method for the  symmetric definite generalized eigenproblem. In \cite{Dopico.Molera.Moro-2003}, an high relative  accuracy bound  was provided for an orthogonal algorithm for the symmetric eigenproblem. In \cite{Dopico.Koev.Molera-2009},  it was showed that the eigenvalues of a symmetric matrix $A$ via  the  implicit Jacobi algorithm are  computed with an error bound proportional to the product of machine precision and the spectral condition number of the eigenvector matrix of $A$. 

In this paper,  we first give the error analysis for  a single step or sweep of the Jacobi method in floating point arithmetic.  We derive the error bounds of the iterative matrix and its off-diagonal entries updated after one Jacobi rotation, and the computed diagonal entries of the updated matrix  are closed to the eigenvalues of the original symmetric matrix with an error bound proportional to the product of machine precision and the norm of the original matrix. The  error bounds of the off-diagonal entries of the iterative matrix updated after one sweep are established for the  general and row cyclic order with distinct eigenvalues and the row-cyclic order with one multiple eigenvalue. Then we propose a mixed precision preconditioned Jacobi method  for the symmetric eigenvalue problem. That is, by using an eigensolver to computing the eigenvalue decomposition of a real symmetric matrix at low precision, we can obtain a low-precision matrix of eigenvectors; Then, by using the high-precision modified Gram-Schmidt (MGS) orthogonalization process, a high-precision orthogonal matrix is obtained, which is employed as an initial guess for the Jacobi method. We give the rounding error analysis of the proposed mixed precision preconditioned  Jacobi method with the cyclic ordering under some conditions.  We also present a mixed precision preconditioned  one-sided Jacobi method for the singular value problem and the corresponding rounding  error analysis is studied.  Finally, we report some numerical experiments to illustrate the efficiency of the proposed method over the original Jacobi method.

Throughout this paper, we use the following notation. Let $\Rmn$ be the set of all $m$-by-$n$ real matrices and $\Rn=\R^{n\times 1}$. $I_n$ is an identity matrix of order $n$ and $\bfe_s$ is the $s$th column of $I_n$. $\bfo_n$ is an $n$-vector of all ones. Let $|\cdot|$ be the absolute value of a real number.  Let $\|\cdot \|$ and $\| \cdot\|_F$ be the Euclidean vector norm or its induced matrix norm and the Frobenius matrix norm, respectively. The symbol ``$\otimes$" means the Kronecker product. The superscript ``$\cdot^T$'' stands for the transpose of a matrix or vector.
For a symmetric matrix $G\in\Rnn$, we denote by $\la_1(G) \geq  \la_2(G) \geq\cdots\geq\la_{n}(G)$ its eigenvalues, arranged in decreasing order. For a matrix $G\in\Rmn$, we denote by $\sigma_1(G) \geq  \sigma_2(G) \geq\cdots\geq\sigma_{\min\{m,n\}}(G)\equiv \sigma_{\min}(G)\geq  0$ its singular values, arranged in decreasing order. For a matrix $G=(g_{ij})\in\Rnn$, let $\off(G):= G - \diag(g_{11},\ldots,g_{nn})$ and $\bg_j$ be the $j$th column vector of $G$ for $j=1,\ldots,n$.

The rest of the paper is organized as follows. In Section \ref{Sec:Pre} we review some error analysis results on classical numerical methods for the eigenvalue problem and the singular value problem. In Section \ref{Sec:Jacobi} we give the error analysis for  a single step or sweep of the Jacobi method in floating point arithmetic. In Section \ref{Sec:Alg} we propose a mixed precision preconditioned Jacobi method for the symmetric eigenvalue problem. The rounding  error analysis is also discussed. In Section \ref{Sec:SVD} we present a  mixed precision preconditioned one-sided Jacobi method for the singular value problem. In Section \ref{Sec:Exp} we present some numerical tests to demonstrate the efficiency of the proposed methods. Some concluding remarks are given in Section \ref{Sec:Con}.

\section{Preliminaries}	\label{Sec:Pre}
In this section, we review some error analysis results on some numerical methods for symmetric eigenvalue problems and singular value problems. We first recall the following error bounds for the MGS method \cite[Theorem 19.13]{Higham-2002}.
\begin{lemma}  \label{lea:MGS}
Let $A \in \R^{m\times n}$ with $\rank(A)=n$. Suppose  the MGS method computes the approximate QR factorization $A\approx \hat{Q}\hat{R}$ in  precision $\upsilon$, where $\hat{R}\in\Rnn$ is upper triangular and $\hat{Q}\in\Rmn$. Then there exist constants $\eta_i\equiv \eta_i(m,n)$ {\rm(}$i=1,2,3${\rm)} such that
$
\| A - \hat{Q}\hat{R} \| \le  \eta_1 \|A\| \upsilon$, $\|\hat{Q}^T \hat{Q} -  I_n \| \le \eta_2 \kappa(A)\upsilon,
$
and $\hat{Q} + \delta Q$ is orthogonal with $\|\delta Q \| \le \eta_3\kappa(A)\upsilon$,
where $\kappa(A)=\sigma_1(A)/\sigma_{\min}(A)$ is the condition number of $A$.
\end{lemma}

On the error analysis for symmetric eigenvalue problems and  singular value problems, we have the following results  (see \cite[pp.104--105] {Anderson.Bai.Bischof-1999}  and \cite[pp.112--113]{Anderson.Bai.Bischof-1999}).
\begin{lemma} \label{lea:symeig}
Let  $A \in \R^{n\times n}$ be a symmetric matrix. The computed symmetric eigenvalue decomposition $A\approx\hat{P}\hat{\La}\hat{P}^T$  with $\hat{P}\in\Rnn$ and $\hat{\La}=\diag(\hat{\la}_1,\ldots,\hat{\la}_n)$ $\in\Rnn$ via any eigensolver in LAPACK or EISPACK  in precision  $\upsilon$ is nearly the exact symmetric Schur decomposition of $A+E$, i.e.,
$
A+E = (\hat{P} + \delta P)\hat{\La} (\hat{P} + \delta P)^T,
$
where $\| E \| \le p(n) \| A \| \upsilon$ and $\hat{P} + \delta P$ is orthogonal with $\| \delta P \| \le p(n)\upsilon$. Here, $p(n)$ is a modestly growing function of $n$.
\end{lemma}
\begin{lemma} \label{lea:svd}
Let  $A \in \R^{m \times n}$ be a real matrix $(m \ge n)$. The computed SVD $A\approx\hat{U}\hat{\Sigma}\hat{V}^T$ with $\hat{U}\in\Rmm$, $\hat{V}\in\Rnn$, and $\hat{\Sigma}=\diag(\hat{\sigma}_1,\ldots,\hat{\sigma}_n)\in\Rmn$ via any SVD solver in LAPACK, LINPACK or  EISPACK  in precision  $\upsilon$  is nearly the exact SVD of $A+E$, i.e.,
$
A + E = (\hat{U} + \delta U) \hat{\Sigma} (\hat{V} + \delta V)^T,
$
where $\| E\| \le p(m,n) \| A \| \upsilon$ and $\hat{U} + \delta U$ and $\hat{V} + \delta V$ are both orthogonal with $\| \delta U \| \le p(m,n) \upsilon$ and $\| \delta V \| \le p(m,n) \upsilon$. Here, $p(m,n)$ is a modestly growing function of $m$ and $n$.
\end{lemma}

Finally, we recall the perturbation bounds for eigenvalues and singular values   \cite[p.442 and p.487]{Golub.VanLoan-2013}.
\begin{lemma}\label{lea:perturbation-eig}
If $G$ and $G+E$ are $n\times n$ real symmetric matrices, then
$|\la_j(G+E)-\la_j(G)| \le  \|E\|$ for $j=1,\ldots,n$ and
$\sum_{j=1}^n(\lambda_j(G+E)-\lambda_j(G))^2\le \|E\|_F^2$.
\end{lemma}
\begin{lemma}\label{lea:perturbation-svd}
If $G$ and $G+E$ are $m\times n$  real matrices with $m\ge n$, then
$|\sigma_j(G+E)-\sigma_j(G)| \le \|E\|$ for $j=1,\ldots,n$.
\end{lemma}


\section{Jacobi's method in floating point arithmetic} \label{Sec:Jacobi}
In this section, we first review Jacobi's method for the symmetric eigenvalue problem. Then we rework the
error analysis for one step/weep of Jacobi's method in floating point arithmetic.

\subsection{Jacobi's method}
Let $A$ be an $n\times n$ real symmetric matrix. The Jacobi method aims to construct a sequence of orthogonal updates $A^{(k+1)} = J_k^TA^{(k)}J_k$ such that the off-diagonal entries of $A^{(k+1)}$ are closer to zeros than $A^{(k)}$, where  $A^{(0)}=A$ and $J_k$ is a Jacobi rotation. When $\off(A^{(k)})$ is close to the zero matrix sufficiently, a computed eigenvalue decomposition of the original matrix $A$ is available.

Define a Jacobi rotation $J(i,j;c,s)$ by
\BE \label{eq:JacMat}
J(i,j;c,s)
= I_n + [\bfe_i, \bfe_j]
\begin{bmatrix} c -1 & s \\ - s & c-1 \end{bmatrix}
\begin{bmatrix} \bfe_i^T \\ \bfe_j^T \end{bmatrix},
\EE
where $c,s\in\R$ is such that $c^2+s^2=1$.
Then we have the following result \cite[\S 8.5]{Golub.VanLoan-2013}.

\begin{lemma}\label{lem:offa-jr}
Let $A\in\Rnn$ be a symmetric matrix. Then, for any index pair $(i,j)$ with $1\le i<j\le n$, there exists a Jacobi rotation  $J=J(i,j;c,s)$
defined by \eqref{eq:JacMat}
such that, for the updated matrix $B=J^TAJ$,
\[
b_{ij}=b_{ji}=0,\quad b_{ii}^2+b_{jj}^2=a_{ii}^2+a_{jj}^2+2a_{ij}^2,
\]
where  $c =(1+t^2)^{-1/2}$ and $s=tc$ with
$t =1/(\mu + \sqrt{1+\mu^2})$ if $\mu\ge 0$ and $t=1/(\mu - \sqrt{1+\mu^2})$ if $\mu< 0$ for $\mu =(a_{jj} - a_{ii})/(2a_{ij})$. If $a_{ij}=0$, then we set $(c,s)=(1,0)$.
\end{lemma}

From Lemma \ref{lem:offa-jr}, we observe that the updated matrix $B=J^TAJ$ agrees with $A$ except in rows and columns $i$ and $j$ and
$\|\off(B) \|_F^2 = \|B \|_F^2-\sum_{i=1}^{n}b_{ii}^2=\|A \|_F^2 -\sum_{i=1}^{n}a_{ii}^2
+ (a_{ii}^2 + a_{jj}^2 - b_{ii}^2 -b_{jj}^2) = \|\off(A) \|_F^2-2a_{ij}^2$.

To minimize $\|\off(B) \|_F$, a classical strategy is to choose the index $(i,j)$ such that the off-diagonal element $a_{ij}$ has the largest absolute value, i.e., $|a_{ij}|=\max_{p\neq q}|a_{pq}|$. This leads to the classical  Jacobi algorithm, which is stated as Algorithm {\rm \ref{alg:Jacobi}}.
\begin{algorithm}[!htb]
	\caption{Classical Jacobi's method for the symmetric eigenvalue problem.} \label{alg:Jacobi}
    \begin{algorithmic}[1]
    \REQUIRE  A symmetric matrix $A\in\Rnn$ and a  tolerance $\epsilon>0$. Let $P=I_n$.
    \WHILE{$\| \off(A) \|_F > \epsilon \| A\|_F$} 
     \STATE  Choose $(i,j)$ such that $|a_{ij}|=\max_{p\neq q}|a_{pq}|$.
	\STATE  Compute a cosine-sine group $(c,s)$ as in Lemma  \ref{lem:offa-jr}.
	\STATE Set $A=J(i,j;c,s)^TAJ(i,j;c,s)$ and $P=PJ(i,j;c,s)$.
	\ENDWHILE
    \end{algorithmic}
\end{algorithm}

Let $A_k$ be the matrix $A_0=A$ after $k$ Jacobi updates. Then Algorithm {\rm \ref{alg:Jacobi}} converges linearly in the sense that
$
\| \off(A_k) \|_F^2 \le \left(1 - 1/N \right)^k \| \off(A_0)\|_F^2
$
 \cite[\S 8.5]{Golub.VanLoan-2013},
where $N = n(n-1)/2$. Here, we refer to $N$ Jacobi updates as a {\em sweep}. The quadratic convergence of Algorithm \ref{alg:Jacobi} was established in \cite{Schonhage-1964} in the sense that for some constant $\alpha>0$,
$
\| \off(A_{k+N}) \|_F \le \alpha \| \off(A_k) \|_F^2
$
for $k$ sufficiently large.

We note that it is expensive to find the optimal index  $(i,j)$ in each Jacobi update.  A feasible alternative   is to update $A$  by rows or columns. This is the so-called {\em cyclic Jacobi} algorithm, which is described   as Algorithm {\rm \ref{alg:rcJacobi}} \cite{Forsythe.Henrici-1960}.
\begin{algorithm}[!htb]
	\caption{Cyclic Jacobi's method for the symmetric eigenvalue problem.} \label{alg:rcJacobi}
    \begin{algorithmic}[1]
    \REQUIRE  A symmetric matrix $A\in\Rnn$ and a  tolerance $\epsilon>0$. Let $P=I_n$.
	\WHILE{$\| \off(A) \|_F > \epsilon \| A\|_F$}
    \STATE  Choose $(i,j)$ in a general cyclic order or in the row cyclic order. \hspace*{0.02in}
	\STATE Compute a cosine-sine group $(c,s)$ as in Lemma  \ref{lem:offa-jr}.
    \STATE Set $A=J(i,j; c,s)^TAJ(i,j;c,s)$ and $P=PJ(i,j; c,s)$.
    \ENDWHILE
	\end{algorithmic}
\end{algorithm}

On the quadratic convergene of Algorithm  \ref{alg:rcJacobi}, one may refer to \cite{Henrici-1958,vanKempen-1966,Wilkinson-1962}.

\subsection{Error analysis for one step of Jacobi's method in floating point arithmetic}
In this subsection, we consider the error analysis for one step of Jacobi's method in floating point arithmetic.  We use the standard model for floating point arithmetic \cite[pp.40]{Higham-2002}
\begin{eqnarray*}
&&\fl(x\,\op \,y) = (x\,\op \,y)(1+\delta_1)  = (x\,\op\,y)/(1+\delta_2),\quad|\delta_1|,|\delta_2|\le u,\quad \op= +,-,*,/, \\
&&\fl(\sqrt{x}) = \sqrt{x}(1+\delta),\; |\delta| \le u,
\end{eqnarray*}
where $u$ is the unit roundoff. Here, $\fl(x)$ means floating-point operation of a real  number $x$ at  precision $u$.

We also recall the following lemma (see for instance \cite[pp.63]{Higham-2002}).

\begin{lemma}
If $|\delta_i|\le u$ and $\xi_i = \pm 1$ for $i=1,2,\ldots,n$, then
$
\prod_{i=1}^n (1+\delta_i)^{\xi_i} = 1 + \theta_n,
$
where $|\theta_n| \le \gamma_n := (1-u)^{-n}-1$.
\end{lemma}

Let
$
\tilde{\gamma}_j := (1-u)^{-wj}-1
$
for a small integer constant $w>0$ whose exact value is unimportant. In what follows, we denote  by $\tilde{\theta}_j$ a quantity with $|\tilde{\theta}_j| \le \tilde{\gamma}_j$.

For  the rounding error analysis for the computed Jacobi rotation $\hat{J}(i,j;\hat{c},\hat{s})$, we have the following result  \cite[Lemma 3.2]{Davies.Higham.Tisseur-2001}.  Here, the computed value in floating point arithmetic is denoted by $\hat{\cdot}$.
\begin{lemma} \label{lea:JacobiAngle}
Let $A\in\Rnn$ be symmetric. Suppose the exact  Jacobi rotation  $J=J(i,j;c,s)$ constructed by Lemma {\rm\ref{lem:offa-jr}} is such that $J^TAJ$ has zeros in the $(i,j)$ and $(j,i)$ positions.
Then the computed  Jacobi rotation $\hat{J}=\hat{J}(i,j;\hat{c},\hat{s})$ is such that
$\hat{c} = c(1+\tilde{\theta}_1)$, $\hat{s} = s(1+\tilde{\theta}_1^\prime)$, $\hat{t} = t(1+\tilde{\theta}_1^{\prime\prime})$, where $c,s$ and $t$ are defined by Lemma {\rm\ref{lem:offa-jr}}.
\end{lemma}

We have the following result after one step of Jacobi's method.

\begin{lemma}\label{lem:pqrk}
If one step of Jacobi's method is performed in the $(p,q)$ plane on the matrix $A_k = (\hat{a}_{ij}^{(k)})$ with the computed Jacobi rotation $\hat{J}_k = J(p,q;\hat{c}_k,\hat{s}_k)$
and the exact Jacobi rotation $J_k = J(p,q; c_k,s_k)$, then the computed  $A_{k+1} = (\hat{a}_{ij}^{(k+1)})$ satisfies
{\rm(i)} element invariance:
\BE \label{eq:uni-rot}
\hat{a}_{ij}^{(k+1)}=\hat{a}_{ij}^{(k)}\quad \mbox{$\forall i, j\neq p,q$},
\EE
{\rm(ii)} proximity to zero:
\BE \label{eq:bound-rot}
|\hat{a}_{pq}^{(k+1)}| \le (|\hat{a}_{pq}^{(k)}|+|s_kc_k| (|\hat{a}_{pp}^{(k)}|+|\hat{a}_{qq}^{(k)}|)) \tilde{\gamma}_4\le \sqrt{2|\hat{a}_{pq}^{(k)}|^2+|\hat{a}_{pp}^{(k)}|^2+|\hat{a}_{qq}^{(k)}|^2} \cdot \tilde{\gamma}_4,
\EE
and {\rm(iii)} sum of squares controllability:
\BE \label{eq:dec-rot}
|\hat{a}_{pj}^{(k+1)}|^2 + |\hat{a}_{qj}^{(k+1)}|^2 \le \big(|\hat{a}_{pj}^{(k)}|^2 + |\hat{a}_{qj}^{(k)}|^2\big)(1+2\tilde{\gamma}_4) \; \forall j \ne p,q.
\EE
\end{lemma}

\begin{proof}
We first show \eqref{eq:uni-rot}.
Observe that $A_{k+1}$ agrees with $A_k$ except in rows and columns $p$ and $q$. This implies that
\eqref{eq:uni-rot} holds.

Next, we show the inequality \eqref{eq:bound-rot}.
It follows from Lemma \ref{lea:JacobiAngle} that
\begin{eqnarray*}
\hat{a}_{pq}^{(k+1)} &=& \fl\left(
\begin{bmatrix} \hat{c}_k\\ - \hat{s}_k \end{bmatrix}^T\begin{bmatrix}\hat{a}_{pp}^{(k)} &  \hat{a}_{pq}^{(k)} \\ \hat{a}_{pq}^{(k)} & \hat{a}_{qq}^{(k)} \end{bmatrix}\begin{bmatrix} \hat{s}_k \\ \hat{c}_k \end{bmatrix}\right) \nonumber\\
&=& \fl(\hat{s}_k(\hat{c}_k\hat{a}_{pp}^{(k)}-\hat{s}_k\hat{a}_{pq}^{(k)})+\hat{c}_k(\hat{c}_k\hat{a}_{pq}^{(k)} -\hat{s}_k\hat{a}_{qq}^{(k)})) \nonumber\\
&=&\fl(\hat{s}_k(\hat{c}_k\hat{a}_{pp}^{(k)}-\hat{s}_k\hat{a}_{pq}^{(k)}))(1+\theta_1)+ \fl(\hat{c}_k(\hat{c}_k\hat{a}_{pq}^{(k)} -\hat{s}_k\hat{a}_{qq}^{(k)}))(1+\theta_1)\nonumber\\
&=& \hat{s}_k\cdot \fl(\hat{c}_k\hat{a}_{pp}^{(k)}-\hat{s}_k\hat{a}_{pq}^{(k)})(1+\theta_2)+ \hat{c}_k\cdot \fl(\hat{c}_k\hat{a}_{pq}^{(k)} -\hat{s}_k\hat{a}_{qq}^{(k)})(1+\theta_2')\nonumber\\
&=& \hat{s}_k (\hat{c}_k\hat{a}_{pp}^{(k)}(1+\theta_2'')-\hat{s}_k\hat{a}_{pq}^{(k)}(1+\theta_2'''))(1+\theta_2) \nonumber\\
&& + \hat{c}_k(\hat{c}_k\hat{a}_{pq}^{(k)}(1+\theta_2''')-\hat{s}_k\hat{a}_{qq}^{(k)}(1+\theta_2''''))(1+\theta_2') \nonumber\\
&=& \hat{s}_k \hat{c}_k\hat{a}_{pp}^{(k)}(1+\theta_4)-\hat{s}_k^2\hat{a}_{pq}^{(k)}(1+\theta_4')
+ \hat{c}_k^2\hat{a}_{pq}^{(k)}(1+\theta_4'')-\hat{s}_k\hat{c}_k\hat{a}_{qq}^{(k)}(1+\theta_4''') \nonumber\\
&=& s_kc_k\hat{a}_{pp}^{(k)}(1+\tilde{\theta}_4)-s_k^2\hat{a}_{pq}^{(k)}(1+\tilde{\theta}_4')+ c_k^2\hat{a}_{pq}^{(k)}(1+\tilde{\theta}_4'')-s_kc_k\hat{a}_{qq}^{(k)}(1+\tilde{\theta}_4''') \nonumber\\
&=& s_kc_k\hat{a}_{pp}^{(k)}\tilde{\theta}_4 -s_k^2\hat{a}_{pq}^{(k)}\tilde{\theta}_4'+ c_k^2\hat{a}_{pq}^{(k)}\tilde{\theta}_4''-s_kc_k\hat{a}_{qq}^{(k)}\tilde{\theta}_4''' ,
\end{eqnarray*}
where the last equality uses the fact that $ \hat{a}_{pq}^{(k)}(c_k^2-s_k^2)+(\hat{a}_{pp}^{(k)}-\hat{a}_{qq}^{(k)})c_ks_k=0$. We note that $c_k^2+s_k^2=1$ and $|s_kc_k|\le1/2(c_k^2+s_k^2)=1/2$.
Then the inequality \eqref{eq:bound-rot} follows from the Cauchy-Schwarz inequality.

Finally, we show the inequality \eqref{eq:dec-rot}.
We have by 
Lemma \ref{lea:JacobiAngle}, for any $j \ne p,q$,
\begin{eqnarray*}
&& (\hat{a}_{pj}^{(k+1)})^2 + (\hat{a}_{qj}^{(k+1)})^2 = \big(\fl(\hat{c}_k \hat{a}_{pj}^{(k)} - \hat{s}_k \hat{a}_{qj}^{(k)})\big)^2 + \big(\fl(\hat{s}_k\hat{a}_{pj}^{(k)} + \hat{c}_k \hat{a}_{qj}^{(k)})\big)^2 \\
	& &\quad = \big(\hat{c}_k \hat{a}_{pj}^{(k)} (1+\theta_2) - \hat{s}_k \hat{a}_{qj}^{(k)}(1+\theta_2')\big)^2
+ \big(\hat{s}_k \hat{a}_{pj}^{(k)} (1+\theta_2'') + \hat{c}_k \hat{a}_{qj}^{(k)}(1+\theta_2''')\big)^2 \\
& &\quad = (c_k\hat{a}_{pj}^{(k)} (1+\tilde{\theta}_2) -s_k \hat{a}_{qj}^{(k)}(1+\tilde{\theta}_2'))^2
+ (s_k \hat{a}_{pj}^{(k)} (1+\tilde{\theta}_2'') + c_k \hat{a}_{qj}^{(k)}(1+\tilde{\theta}_2'''))^2 \\
& &\quad = c_k^2 (\hat{a}_{pj}^{(k)})^2 (1+\tilde{\theta}_4) + s_k^2 (\hat{a}_{qj}^{(k)})^2 (1+\tilde{\theta}_4')
- 2c_ks_k \hat{a}_{pj}^{(k)}\hat{a}_{qj}^{(k)}(1+\tilde{\theta}_4'')  \\
	& &\qquad +s_k^2 (\hat{a}_{pj}^{(k)})^2 (1+\tilde{\theta}_4''') + c_k^2 (\hat{a}_{qj}^{(k)})^2 (1+\tilde{\theta}_4'''')
+ 2c_ks_k \hat{a}_{pj}^{(k)}\hat{a}_{qj}^{(k)}(1+\tilde{\theta}_4''''') \\
	& &\quad  \le  \big((\hat{a}_{pj}^{(k)})^2 + (\hat{a}_{qj}^{(k)})^2\big)(1+\tilde{\gamma}_4) + 4 |s_kc_k| |\hat{a}_{pj}^{(k)}\hat{a}_{qj}^{(k)}|\tilde{\gamma}_4\\
	& &\quad  \le  \big((\hat{a}_{pj}^{(k)})^2 + (\hat{a}_{qj}^{(k)})^2\big) + \big((\hat{a}_{pj}^{(k)})^2 + (\hat{a}_{qj}^{(k)})^2 + 2 |\hat{a}_{pj}^{(k)}\hat{a}_{qk}^{(k)}|\big)\tilde{\gamma}_4\\
	& &\quad  \le \big((\hat{a}_{pj}^{(k)})^2 + (\hat{a}_{qj}^{(k)})^2\big) +2\big((\hat{a}_{pj}^{(k)})^2 + (\hat{a}_{qj}^{(k)})^2\big)\tilde{\gamma}_4 = \big((\hat{a}_{pj}^{(k)})^2 + (\hat{a}_{qj}^{(k)})^2\big) (1+2\tilde{\gamma}_4).
\end{eqnarray*}
This completes the proof.
\end{proof}

On the error  bound after one Jacobi rotation, we have the following result.
\begin{lemma}\label{lem:pqrk-err}
If one step of Jacobi's method is performed in the $(p,q)$ plane on the matrix $A_k = (\hat{a}_{ij}^{(k)})$ with the computed Jacobi rotation $\hat{J}_k = J(p,q;\hat{c}_k,\hat{s}_k)$ and the exact Jacobi rotation $J_k = J(p,q; c_k,s_k)$, then the computed  $A_{k+1}$ satisfies
\[
A_{k+1}= J_k^T A_k J_k+Y_k,\quad
\|Y_k\|_F \le \sqrt{(6+4\sqrt{2})\|G_k\|_F^2 + 2\|H_k\|_F^2}\cdot  \tilde{\gamma}_4,
\]
where  the symmetric matrix $G_k\in\Rnn$ has zero entries except that the entries at the intersections of rows and columns $p$ and $q$ are the same as  those of $A_k$, and  the symmetric matrix $H_k\in\Rnn$ has zero  entries except that  $(H_k)_{pj}=\hat{a}_{pj}^{(k)}$ and $(H_k)_{qj}=\hat{a}_{qj}^{(k)}$ for all $j\neq p,q$ and $(H_k)_{ip}=\hat{a}_{ip}^{(k)}$ and $(H_k)_{iq}=\hat{a}_{iq}^{(k)}$ for all $i\neq p,q$.
\end{lemma}
\begin{proof}
Let $J_k^T A_k J_k=(\tilde{a}_{ij}^{(k+1)})$. Then we have
\begin{eqnarray*}\label{eq:hatapp}
&& \hat{a}_{pp}^{(k+1)} = \fl\left(
\begin{bmatrix}\hat{c}_k \\ - \hat{s}_k \end{bmatrix}^T \begin{bmatrix}\hat{a}_{pp}^{(k)} &  \hat{a}_{pq}^{(k)} \\ \hat{a}_{pq}^{(k)} & \hat{a}_{qq}^{(k)} \end{bmatrix}\begin{bmatrix} \hat{c}_k \\ -\hat{s}_k \end{bmatrix}\right) \\
&=& \fl(\hat{c}_k(\hat{c}_k\hat{a}_{pp}^{(k)}-\hat{s}_k\hat{a}_{pq}^{(k)})-\hat{s}_k(\hat{c}_k\hat{a}_{pq}^{(k)} -\hat{s}_k\hat{a}_{qq}^{(k)})) \\
&=&\fl(\hat{c}_k(\hat{c}_k\hat{a}_{pp}^{(k)}-\hat{s}_k\hat{a}_{pq}^{(k)}))(1+\theta_1) - \fl(\hat{s}_k(\hat{c}_k\hat{a}_{pq}^{(k)} -\hat{s}_k\hat{a}_{qq}^{(k)}))(1+\theta_1)\nonumber\\
&=& \hat{c}_k\cdot \fl(\hat{c}_k\hat{a}_{pp}^{(k)}-\hat{s}_k\hat{a}_{pq}^{(k)})(1+\theta_2) - \hat{s}_k\cdot \fl(\hat{c}_k\hat{a}_{pq}^{(k)} -\hat{s}_k\hat{a}_{qq}^{(k)})(1+\theta_2')\nonumber\\
&=& \hat{c}_k (\hat{c}_k\hat{a}_{pp}^{(k)}(1+\theta_2'')-\hat{s}_k\hat{a}_{pq}^{(k)}(1+\theta_2'''))(1+\theta_2) \nonumber\\
&& - \hat{s}_k(\hat{c}_k\hat{a}_{pq}^{(k)}(1+\theta_2'''')-\hat{s}_k\hat{a}_{qq}^{(k)}(1+\theta_2'''''))(1+\theta_2') \nonumber\\
&=& \hat{c}_k^2 \hat{a}_{pp}^{(k)}(1+\theta_4)-\hat{s}_k\hat{c}_k\hat{a}_{pq}^{(k)}(1+\theta_4')
- \hat{s}_k\hat{c}_k\hat{a}_{pq}^{(k)}(1+\theta_4'') + \hat{s}_k^2\hat{a}_{qq}^{(k)}(1+\theta_4''') \nonumber\\
&=& c_k^2\hat{a}_{pp}^{(k)}(1+\tilde{\theta}_4)-s_kc_k\hat{a}_{pq}^{(k)}(1+\tilde{\theta}_4') - s_kc_k\hat{a}_{pq}^{(k)}(1+\tilde{\theta}_4'') + s_k^2\hat{a}_{qq}^{(k)}(1+\tilde{\theta}_4''') \nonumber\\
&=& c_k^2\hat{a}_{pp}^{(k)} - 2 s_k c_k \hat{a}_{pq}^{(k)} + s_k^2 \hat{a}_{qq}^{(k)} + \big( c_k^2\hat{a}_{pp}^{(k)}\tilde{\theta}_4-s_kc_k\hat{a}_{pq}^{(k)}\tilde{\theta}_4'- s_kc_k\hat{a}_{pq}^{(k)}\tilde{\theta}_4'' + s_k^2\hat{a}_{qq}^{(k)}\tilde{\theta}_4'''\big) \nonumber \\
&=& \tilde{a}_{pp}^{(k+1)} + \big( c_k^2\hat{a}_{pp}^{(k)}\tilde{\theta}_4-s_kc_k\hat{a}_{pq}^{(k)}\tilde{\theta}_4'- s_kc_k\hat{a}_{pq}^{(k)}\tilde{\theta}_4'' + s_k^2\hat{a}_{qq}^{(k)}\tilde{\theta}_4'''\big).
\end{eqnarray*}
Thus,
$
|\hat{a}_{pp}^{(k+1)} - \tilde{a}_{pp}^{(k+1)}|
 \le  (c_k^2|\hat{a}_{pp}^{(k)}|+s_k^2|\hat{a}_{qq}^{(k)}|+2|s_kc_k||\hat{a}_{pq}^{(k)}|)\tilde{\gamma}_4.
$

Analogously, we can  show that
$
|\hat{a}_{qq}^{(k+1)} - \tilde{a}_{qq}^{(k+1)}| \le (s_k^2|\hat{a}_{pp}^{(k)}|+c_k^2|\hat{a}_{qq}^{(k)}|+2|s_kc_k||\hat{a}_{pq}^{(k)}|)\tilde{\gamma}_4.
$

For any $j \ne p,q$, we have
\begin{eqnarray*}
\hat{a}_{pj}^{(k+1)} &=&  \fl\left(
\begin{bmatrix}\hat{c}_k \\ - \hat{s}_k \end{bmatrix}^T \begin{bmatrix}\hat{a}_{pj}^{(k)} \\ \hat{a}_{qj}^{(k)}\end{bmatrix}\right)
 =  \hat{c}_k\hat{a}_{pj}^{(k)} (1+\theta_2) - \hat{s}_k \hat{a}_{qj}^{(k)}(1+\theta_2')\\
&=& c_k\hat{a}_{pj}^{(k)} (1+\tilde{\theta}_2) -s_k \hat{a}_{qj}^{(k)}(1+\tilde{\theta}_2')
 =c_k\hat{a}_{pj}^{(k)} - s_k \hat{a}_{qj}^{(k)} + (c_k\hat{a}_{pj}^{(k)}\tilde{\theta}_2-s_k \hat{a}_{qj}^{(k)}\tilde{\theta}_2') \\
 &=&  \tilde{a}_{pj}^{(k+1)} + (c_k\hat{a}_{pj}^{(k)}\tilde{\theta}_2-s_k \hat{a}_{qj}^{(k)}\tilde{\theta}_2')
\end{eqnarray*}
and thus
$
|\hat{a}_{pj}^{(k+1)} - \tilde{a}_{pj}^{(k+1)}|
\le  (|c_k||\hat{a}_{pj}^{(k)}| + |s_k||\hat{a}_{qj}^{(k)}|)\tilde{\gamma}_2.
$

In a similar way, we have
$
|\hat{a}_{qj}^{(k+1)} - \tilde{a}_{qj}^{(k+1)}| \le (|s_k||\hat{a}_{pj}^{(k)}| + |c_k||\hat{a}_{qj}^{(k)}|)\tilde{\gamma}_2.
$
Using \eqref{eq:bound-rot}  and  the fact $\tilde{a}_{pq}^{(k+1)} = 0$, we obtain
$
|\hat{a}_{pq}^{(k+1)} - \tilde{a}_{pq}^{(k+1)}| \le (|\hat{a}_{pq}^{(k)}|+|s_kc_k|(|\hat{a}_{pp}^{(k)}|+|\hat{a}_{qq}^{(k)}|)) \tilde{\gamma}_4.
$

Therefore,
\begin{eqnarray*}
&&  \|Y_k\|_F^2 = \|A_{k+1}-J_k^T A_k J_k\|_F^2 \\
&&\quad =(\hat{a}_{pp}^{(k+1)} - \tilde{a}_{pp}^{(k+1)})^2+(\hat{a}_{qq}^{(k+1)} - \tilde{a}_{qq}^{(k+1)})^2 +2(\hat{a}_{pq}^{(k+1)} - \tilde{a}_{pq}^{(k+1)})^2 \\
&&\qquad + 2\sum_{j\neq p,q}\big((\hat{a}_{pj}^{(k+1)} -
\tilde{a}_{pj}^{(k+1)})^2+(\hat{a}_{qj}^{(k+1)} - \tilde{a}_{qj}^{(k+1)})^2\big) \\
&&\quad \le  (c_k^2 |\hat{a}_{pp}^{(k)}|+ s_k^2 |\hat{a}_{qq}^{(k)}|+2|s_kc_k||\hat{a}_{pq}^{(k)}|)^2\tilde{\gamma}_4^2 + (s_k^2 |\hat{a}_{pp}^{(k)}|+ c_k^2 |\hat{a}_{qq}^{(k)}|+2|s_kc_k||\hat{a}_{pq}^{(k)}|)^2 \tilde{\gamma}_4^2 \nonumber \\
& &\qquad + 2(|\hat{a}_{pq}^{(k)}|+|s_kc_k|(|\hat{a}_{pp}^{(k)}|+|\hat{a}_{qq}^{(k)}|))^2 \tilde{\gamma}_4^2\\
& &\qquad + 2\sum_{j\ne p,q}\big( (|c_k||\hat{a}_{pj}^{(k)}| + |s_k||\hat{a}_{qj}^{(k)}|)^2 +(|s_k||\hat{a}_{pj}^{(k)}| + |c_k||\hat{a}_{qj}^{(k)}|)^2\big) \tilde{\gamma}_4^2 \\
&&\quad \le  \big((1+\sqrt{2}|s_kc_k|)|\hat{a}_{pp}^{(k)}| + (1+\sqrt{2}|s_kc_k|)|\hat{a}_{qq}^{(k)}| + (\sqrt{2} +4|s_kc_k|)|\hat{a}_{pq}^{(k)}| \big)^2 \tilde{\gamma}_4^2 \\
&&\qquad + 4 \sum_{j\ne p,q} (|\hat{a}_{pj}^{(k)}|^2 + |\hat{a}_{qj}^{(k)}|^2) \tilde{\gamma}_4^2 \nonumber \\
&&\quad \le  \big(2(1+\sqrt{2}|s_kc_k|)^2+(1+2\sqrt{2}|s_kc_k|)^2\big) \big((\hat{a}_{pp}^{(k)})^2+(\hat{a}_{qq}^{(k)})^2 +2(\hat{a}_{pq}^{(k)})^2\big)\tilde{\gamma}_4^2 \\
&&\qquad + 4 \sum_{j\ne p,q} (|\hat{a}_{pj}^{(k)}|^2 + |\hat{a}_{qj}^{(k)}|^2) \tilde{\gamma}_4^2 \nonumber \\
&&\quad \le  2(1+\sqrt{2})^2 \big((\hat{a}_{pp}^{(k)})^2+(\hat{a}_{qq}^{(k)})^2+2(\hat{a}_{pq}^{(k)})^2\big)\tilde{\gamma}_4^2 + 4 \sum_{j\ne p,q} (|\hat{a}_{pj}^{(k)}|^2  + |\hat{a}_{qj}^{(k)}|^2) \tilde{\gamma}_4^2 \nonumber \\
&&\quad \equiv 2(3+2\sqrt{2})\|G_k\|_F^2\tilde{\gamma}_4^2 + 2\|H_k\|_F^2 \tilde{\gamma}_4^2,
\end{eqnarray*}
where the third inequality uses the fact that $a_1^2+a_2^2+a_3^2\le(a_1+a_2+a_3)^2$ for all $a_1,a_2,a_3\ge 0$ and the Cauchy-Schwarz inequality, the fourth inequality uses the
Cauchy-Schwarz inequality, and the fifth inequality uses the fact that $|s_kc_k|\le 1/2$.
The lemma follows by taking the square root of the above inequality.
\end{proof}


On the off-diagonal entries  of the updated matrix after one step of Jacobi's method, we have the following result.
\begin{lemma}\label{lem:saij}
If one step of Jacobi's method is performed in the $(p,q)$ plane on the matrix $A_k = (\hat{a}_{ij}^{(k)})$ with the computed Jacobi rotation $\hat{J}_k = J(p,q;\hat{c}_k,\hat{s}_k)$, then the computed $A_{k+1}$ satisfies
\begin{eqnarray}\label{diff:offam1m}
 \|\off(A_{k+1})\|_F^2 -\|\off(A_k)\|_F^2
\le   - 2|\hat{a}_{pq}^{(k)}|^2 + 2\|H_k\|_F^2\tilde{\gamma}_4 + 2\|G_k\|_F^2\tilde{\gamma}_4^2,
\end{eqnarray}
where $G_k$ and $H_k$ are defined as in Lemma {\rm \ref{lem:pqrk-err}}.
Moreover, for arbitrarily  chosen index pairs $\{(p_k,q_k)\}_{k=0}^{N-1}$, we have
\BE \label{ieq:sumz}
2\sum_{k=0}^{N-1} |\hat{a}_{p_k q_k}^{(k)}|^2 \le \|\off(A_0)\|_F^2 + 2\sum_{k=0}^{N-1}\|H_k\|_F^2\tilde{\gamma}_4 + 2\sum_{k=0}^{N-1}\|G_k\|_F^2\tilde{\gamma}_4^2.
\EE
\end{lemma}

\begin{proof}
It directly follows from Lemma \ref{lem:pqrk} that, for the chosen index pair $(p,q)$,
\begin{eqnarray*}
& & \|\off(A_{k+1})\|_F^2 - \|\off(A_k)\|_F^2 =  \sum_{i\ne j} \big(|\hat{a}_{ij}^{(k+1)}|^2 - |\hat{a}_{ij}^{(k)}|^2\big) \\
&=&   2\big(|\hat{a}_{pq}^{(k+1)}|^2-|\hat{a}_{pq}^{(k)}|^2\big) + 2\sum_{j\ne p,q} \big(|\hat{a}_{pj}^{(k+1)}|^2 + |\hat{a}_{qj}^{(k+1)}|^2 -|\hat{a}_{pj}^{(k)}|^2-|\hat{a}_{qj}^{(k)}|^2\big) \\
& \le & 2\big(|\hat{a}_{pp}^{(k)}|^2+|\hat{a}_{qq}^{(k)}|^2+2|\hat{a}_{pq}^{(k)}|^2\big) \tilde{\gamma}_4^2 - 2|\hat{a}_{pq}^{(k)}|^2 + 4\sum_{j\ne p,q} \big(|\hat{a}_{pj}^{(k)}|^2 + |\hat{a}_{qj}^{(k)}|^2\big)\tilde{\gamma}_4 \\
& = & 2\|G_k\|_F^2\tilde{\gamma}_4^2   - 2|\hat{a}_{pq}^{(k)}|^2 + 2\|H_k\|_F^2\tilde{\gamma}_4,
\end{eqnarray*}
which is \eqref{diff:offam1m}.
For  arbitrarily chosen index pairs $\{(p_k,q_k)\}_{k=0}^{N-1}$, we have  by \eqref{diff:offam1m},
\begin{eqnarray*}
& & \|\off(A_{N})\|_F^2 - \|\off(A_0)\|_F^2
 \le  2\sum_{k=0}^{N-1} \|G_k\|_F^2\tilde{\gamma}_4^2   - 2\sum_{k=0}^{N-1} |\hat{a}_{p_kq_k}^{(k)}|^2 + 2\sum_{k=0}^{N-1} \|H_k\|_F^2\tilde{\gamma}_4,
\end{eqnarray*}
which yields  \eqref{ieq:sumz}.
\end{proof}

\begin{remark}\label{rem:dc}
We observe from Lemma {\rm\ref{lem:saij}} that $A_k=(\hat{a}_{ij}^{(k)})$ moves closer to diagonal form with each Jacobi step if  $\tilde{\gamma}_4$ is such that
$
\|H_k\|_F^2\tilde{\gamma}_4 + \|G_k\|_F^2 \tilde{\gamma}_4^2 < |\hat{a}_{pq}^{(k)}|^2.
$
If $|\hat{a}_{pq}^{(k)}|\le u\min\{|\hat{a}_{pp}^{(k)}|, |\hat{a}_{qq}^{(k)}|\}$ or $|\hat{a}_{pq}^{(k)}|\le u(|\hat{a}_{pp}^{(k)}\hat{a}_{qq}^{(k)}|)^{1/2}$, then one may set $\hat{a}_{pq}^{(k)} =\hat{a}_{qp}^{(k)} = 0$ {\rm\cite{Demmel.Veselic-1992, Davies.Higham.Tisseur-2001}}.
\end{remark}

The following theorem states the error bound for one step of Jacobi's method.
\begin{theorem}\label{cov:cjm}
Let $A_k$ be the matrix $A_0=A$ after $k$ Jacobi updates with the computed Jacobi rotations $\{\hat{J}_j = J(p_j,q_j;\hat{c}_j,\hat{s}_j)\}_{j=0}^{k-1}$ in Algorithm {\rm \ref{alg:Jacobi}}.  If $2N\tilde{\gamma}_4 <1$, then we have, for any $k\ge 0$,
\BE\label{diff:offam1m-classical}
\|\off(A_{k+1})\|_F^2 \le c\|\off(A_k)\|_F^2 + 2\|G_k\|_F^2\tilde{\gamma}_4^2.
\EE
where $c = 1 - 1/N + 2\tilde{\gamma}_4 \in (0,1)$.
Moreover,  for any $k\ge 1$, there exists some ordering $\{\la_{\pi^{(0)}(i)}(A_0)\}_{i=1}^n$ of $\{\la_i(A_0)\}_{i=1}^n$ such that
\BE\label{ineq:akdiags}
\Big(\sum_{i=1}^n \big(\hat{a}_{ii}^{(k)} - \la_{\pi^{(0)}(i)}(A_0)\big)^2\Big)^{1/2} \le \|\off(A_k)\|_F +  \tilde{\gamma}_4\sum_{j=0}^{k-1}\varphi_j,
\EE
where $ \varphi_j= ((6+4\sqrt{2})\|G_j\|_F^2 + 2\|H_j\|_F^2)^{1/2}$ with $G_j$ and $H_j$ being defined as in Lemma {\rm \ref{lem:pqrk-err}}  for $j=0,1,\ldots,k-1$.
\end{theorem}
\begin{proof}
From the procedure  of Algorithm \ref{alg:Jacobi}, it follows that, in the $k$th iteration,  the index pair $(p_k,q_k)$ is chosen such that $|\hat{a}^{(k)}_{p_kq_k}|=\max_{i\neq j}|\hat{a}^{(k)}_{ij}|$. Then,
by \eqref{diff:offam1m} and using the definition of $G_k$ and the fact that $\|H_k\|_F \le \|\off(A_k)\|_F$ and $\|\off(A_k)\|_F^2\le 2N |a_{p_kq_k}|^2$, we obtain \eqref{diff:offam1m-classical}, where $c = 1 - 1/N + 2\tilde{\gamma}_4 \in (0,1)$ due to $2N\tilde{\gamma}_4 <1$.

Let  $J_j = J(p_j,q_j;c_j,s_j)$ be the exact  Jacobi rotation corresponding to $\hat{J}_j$ for $j=0,1,\ldots,k-1$.
Note that $A_j$ and $J_j^T A_j J_j$  have the same eigenvalues  for $j=0,1,\ldots,k-1$. Then, for any $k\ge 1$, there exist  some ordering $\{\la_{\pi^{(j)}(i)}(A_j)\}_{i=1}^n$ of $\{\la_i(A_j)\}_{i=1}^n$ with
$D(A_j;\pi^{(j)})=\diag(\lambda_{\pi^{(j)}(1)}(A_j),\ldots, \lambda_{\pi^{(j)}(n)}(A_j) )$ for $j=0,1,\ldots,k-1$ such that
\begin{eqnarray*}
& & \Big(\sum_{i=1}^n (\hat{a}_{ii}^{(k)} - \lambda_{\pi^{(0)}(i)}(A_0))^2 \Big)^{1/2} = \| A_k - \off(A_k)-D(A_0;\pi^{(0)})\|_F\\
& \le & \|A_k - \off(A_k) - D(J_{k-1}^T A_{k-1} J_{k-1};\pi^{(k-1)})\|_F + \|D(A_{k-1};\pi^{(k-1)}) - D(A_0;\pi^{(0)})\|_F \\
& \le & \|A_k - \off(A_k) - J_{k-1}^T A_{k-1} J_{k-1}\|_F + \|D(A_{k-1};\pi^{(k-1)}) - D(A_0;\pi^{(0)})\|_F
\\
& \le & \|\off(A_k)\|_F + \|Y_{k-1}\|_F + \|D(A_{k-1};\pi^{(k-1)}) - D(A_0;\pi^{(0)})\|_F \\
& \le & \|\off(A_k)\|_F + \|Y_{k-1}\|_F + \|D(A_{k-1};\pi^{(k-1)}) - D(J_{k-2}^T A_{k-2} J_{k-2};\pi^{(k-2)})\|_F \\
&& + \|D(A_{k-2};\pi^{(k-2)} ) - D(A_0;;\pi^{(0)})\|_F \\
& \le & \|\off(A_k)\|_F + \|Y_{k-1}\|_F + \|Y_{k-2}\|_F + \|D( A_{k-2};\pi^{(k-2)} ) - D(A_0;\pi^{0j)})\|_F \\
& \le & \cdots \le \|\off(A_k)\|_F + \sum_{j=0}^{k-1}\|Y_j\|_F \le \|\off(A_k)\|_F +  \tilde{\gamma}_4\sum_{j=0}^{k-1}\varphi_j,
\end{eqnarray*}
where  the second, third, fifth  inequalities and so on use  Lemma \ref{lea:perturbation-eig} and/or $Y_j= A_{j+1} -J_j^T A_j J_j$ for $j=k-1,k-2,\ldots,0$ as defined in Lemma \ref{lem:pqrk-err}, and the last  inequality uses Lemma \ref{lem:pqrk-err}.
\end{proof}

\begin{remark}
The error bound \eqref{ineq:akdiags} also holds for $A_k$ being  the matrix $A_0=A$ after $k$ Jacobi updates by Algorithm {\rm \ref{alg:rcJacobi}}.
\end{remark}
\begin{remark}
We observe from Theorem {\rm \ref{cov:cjm}} that if $ \tilde{\gamma}_4=0$, then \eqref{diff:offam1m-classical} is reduced to $
\|\off(A_{k+1})\|_F^2 \le (1-1/N)\|\off(A_k)\|_F^2,
$
which implies the  linear rate of the classical Jacobi method by induction  {\rm (}see {\rm \cite[p.479]{Golub.VanLoan-2013}}{\rm )}.  Furthermore, it follows from \eqref{diff:offam1m-classical} that
\[
\begin{array}{lcl}
&& \|\off(A_{k+1})\|_F \le \sqrt{c} \|\off(A_k)\|_F   + \sqrt{2} \|G_k\|_F\tilde{\gamma}_4  \\
& \le&  \sqrt{c}\|\off(A_k)\|_F + \sqrt{2}(\|A_0\| + \sum_{j=0}^{k-1}\|Y_j\|) \tilde{\gamma}_4  \\
&\le& \sqrt{c}\|\off(A_k)\|_F + \sqrt{2} (\|A_0\| +\sum_{j=0}^{k-1} \varphi_j\tilde{\gamma}_4) \tilde{\gamma}_4,
\end{array}
\]
by using  the fact that  $a_1^2+a_2^2\le(a_1+a_2)^2$ for all $a_1,a_2\ge 0$ and $\|G_k\|_F=((a_{p_kp_k}^{(k)})^2+(a_{q_kq_k}^{(k)})^2+2(a_{p_kq_k}^{(k)})^2)^{1/2}\le \|A_k\|_F \le \|A_{k-1}\|_F+\|Y_{k-1}\|_F \le \cdots \le \|A_0\|_F + \sum_{j=0}^{k-1}\|Y_j\|_F$, where $Y_k$'s are defined as in Lemma {\rm\ref{lem:pqrk-err}}. In addition,  if Algorithm {\rm\ref{alg:Jacobi}}  terminates at the $k_0$th iteration with the stopping criterion being satisfied, then \eqref{ineq:akdiags} becomes
$
(\sum_{i=1}^n \big(\hat{a}_{ii}^{(k)} - \la_{\pi^{(0)}(i)}(A_0)\big)^2 )^{1/2} \le \epsilon \|A_0\|_F + \tilde{\gamma}_4 \sum_{j=0}^{k_0-1} \varphi_j.
$
For the similar  error analysis on cyclic Jacobi's method, one may refer to \cite[p.279]{Wilkinson-1965} {\rm(}see also \cite[pp.480--481]{Golub.VanLoan-2013}{\rm)}.
\end{remark}

\subsection{Error analysis for one sweep of the cyclic Jacobi method in floating point arithmetic}
In this subsection, we consider the error analysis for one sweep of the cyclic Jacobi method  in floating point arithmetic.
\subsubsection{The general cyclic order with distinct eigenvalues} \label{sec321}
We consider the error analysis for one sweep of the general cyclic Jacobi method in floating point arithmetic for a symmetric matrix $A\in\Rnn$ with $n$ distinct eigenvalues.

We first discuss the relationship between  the Frobenius matrix norm of the off-diagonal entries of $A_k$  and the  minimal gap between the eigenvalues of  $A_k$ in the general cyclic Jacobi method.

\begin{lemma} \label{lea:Jacobi-dist}
If one step of Jacobi's method is performed in the $(p,q)$ plane on the matrix $A_k = (\hat{a}_{ij}^{(k)})$ with the computed Jacobi rotation $\hat{J}_k = J(p,q;\hat{c}_k,\hat{s}_k)$ and $A_k$ has $n$  distinct eigenvalues with
$
d(A_k) := \min_{i\neq j} | \la_i(A_k) - \la_j(A_k) |>0,
$
then the computed  $A_{k+1}$ also has $n$ distinct eigenvalues with
$
d(A_{k+1}) \ge  d(A_k) - 2\varphi_k \tilde{\gamma}_4>0,
$
provided that $2\varphi_k \tilde{\gamma}_4<d_k$, where $\varphi_k = ((6+4\sqrt{2})\|G_k\|_F^2 + 2\|H_k\|_F^2)^{1/2}$ with $G_k$ and $H_k$ being defined as in Lemma {\rm \ref{lem:pqrk-err}}.
\end{lemma}
\begin{proof}
Let $J_k = J(p,q;c_k,s_k)$ be the exact Jacobi rotation. Then, by Lemma \ref{lea:perturbation-eig} we have, for any $i \ne j$,
\begin{eqnarray*}
& & |\la_i(A_{k+1}) - \la_j(A_{k+1}) | \\
&=& | \la_i(A_{k+1}) - \la_i(A_k) + \la_i(A_{k}) - \la_j(A_{k}) + \la_j(A_k)- \la_j(A_{k+1}) | \\
		& \ge & | \la_i(A_k) - \la_j(A_k) | - |\la_i(A_{k+1}) - \la_i(A_k)| - |\la_j(A_{k+1}) - \la_j(A_k)| \\
		& = & | \la_i(A_k) - \la_j(A_k) | - |\la_i(A_{k+1}) - \la_i(J_k^T A_k J_k)| - |\la_j(A_{k+1}) - \la_j(J_k^T A_k J_k)| \\
		& \ge & d(A_k) - 2 \|A_{k+1}-J_k^T A_k J_k\|,
\end{eqnarray*}
completing the proof of the lemma by invoking Lemma \ref{lem:pqrk-err}.
\end{proof}
\begin{lemma} \label{lem:cond-ineq}
Suppose one step of Jacobi's method is performed in the $(p,q)$ plane on the matrix $A_k = (\hat{a}_{ij}^{(k)})$ with the computed Jacobi rotation $\hat{J}_k = J(p,q;\hat{c}_k,\hat{s}_k)$. Let $ \varphi_k= ((6+4\sqrt{2})\|G_k\|_F^2 + 2\|H_k\|_F^2)^{1/2}$ with $G_k$ and $H_k$ being defined as in Lemma {\rm \ref{lem:pqrk-err}}. If $A_k$ has $n$ distinct eigenvalues with $4\|\off(A_k)\|_F <d(A_k)$ and
\BE \label{eq:sdescent}
\frac{1}{8}\varphi_k^2 \tilde{\gamma}_4^2 + \frac{1}{2}\|\off(A_k)\|_F \varphi_k\tilde{\gamma}_4 + \|H_k\|_F^2\tilde{\gamma}_4 + \|G_k\|_F^2\tilde{\gamma}_4^2 < |\hat{a}_{pq}^{(k)}|^2,
\EE
then we have
$
4 \|\off(A_{k+1})\|_F < d(A_{k+1}).
$
\end{lemma}
\begin{proof}
Let $b_k = \|H_k\|_F^2\tilde{\gamma}_4 + \|G_k\|_F^2\tilde{\gamma}_4^2$.
From  \eqref{eq:sdescent} we have $b_k < |\hat{a}_{pq}^{(k)}|^2$. By Lemma {\rm\ref{lem:saij}}, we have  $\|\off(A_{k+1})\|_F<\|\off(A_k)\|_F$ and 
$
\|\off(A_{k+1})\|_F^2 + 2(\hat{a}_{pq}^{(k)})^2 - 2b_k>0.
$
This, together with  \eqref{eq:sdescent} again, yields
\begin{eqnarray}\label{eq:Qeq-r}
&&(\|\off(A_{k+1})\|_F  + \frac{1}{2}\varphi_k\tilde{\gamma}_4 )^2 -(\|\off(A_{k+1})\|_F^2 + 2|\hat{a}_{pq}^{(k)}|^2 - 2b_k) \nonumber\\
&\le& \frac{1}{4}\varphi_k^2 \tilde{\gamma}_4^2  + \|\off(A_{k+1})\|_F \varphi_k\tilde{\gamma}_4  + 2b_k - 2|\hat{a}_{pq}^{(k)}|^2  \nonumber\\
&\le& \frac{1}{4}\varphi_k^2 \tilde{\gamma}_4^2  + \|\off(A_{k})\|_F \varphi_k\tilde{\gamma}_4  + 2b_k - 2|\hat{a}_{pq}^{(k)}|^2<0.
\end{eqnarray}

By hypothesis, we have $4\|\off(A_k)\|_F < d(A_k)$. 
Then, by following the arguments similar to the proof of  Lemma \ref{lea:Jacobi-dist} we have
\begin{eqnarray*}
d(A_{k+1}) & \ge & d(A_{k}) - 2\varphi_k \tilde{\gamma}_4  > 4 \|\off(A_k)\|_F -  2\varphi_k \tilde{\gamma}_4.  \\
& \ge & 4\big(\|\off(A_{k+1})\|_F^2 + 2(\hat{a}_{pq}^{(k)})^2 - 2b_k  \big)^{1/2} -  2\varphi_k\tilde{\gamma}_4  
>  4 \|\off(A_{k+1})\|_F,
\end{eqnarray*}
where the third inequality uses  \eqref{diff:offam1m} and the  last inequality uses \eqref{eq:Qeq-r}.
\end{proof}

Next, we  recall the relationship between the gap of  diagonal entries of a real symmetric matrix $A$  and the minimal gap between the eigenvalues of $A$ and  the relationship between the rotation angle and the off-diagonal Frobenius norm \cite{Wilkinson-1962}.
\begin{lemma} \label{lea:diagdist}
Let $G = (g_{ij})$ be an $n\times n$ real  symmetric matrix with $n$ distinct eigenvalues. If
$
d(G) = \min_{i\ne j} | \la_i(G) - \la_j(G) |>0
$
and $\|\off(G)\|_F\le d(G)/4$, then we have, for some ordering of $\{\la_i(G)\}_{i=1}^n$,
$|g_{ii} -g_{jj}| \ge d(G)/2$ for all $i\ne j$, and for $(p,q)$ plane on $G$, the angle  $\theta_{pq}$ of the exact Jacobi rotation  satisfies
$
|\sin \theta_{pq}| \le 2| g_{pq}|/d(G).
$
\end{lemma}

In the following, we study the error analysis for one sweep of the general cyclic Jacobi method in floating
point arithmetic. Suppose that  $A\in\Rnn$ is symmetric with $n$ distinct eigenvalues. In the general cyclic Jacobi method, all off-diagonal entries are annihilated (in the sense of floating point arithmetic) successively in some order. For convenience, we assume that $\|\off(A_0) \|_F<d(A_0)/4$. We use $(p_k, q_k)$ as the chosen annihilation position at the $k$th iteration of the whole general cyclic Jacobi method. During a fixed cycle of $N$ consecutive rotations of a general cyclic ordering starting from $k=0$, the entries before annihilation are denoted by $\{z_k\}_{k=0}^{N-1}$ and the computed and exact Jacobi rotations by $\{\hat{J}(p_k,q_k;\hat{c}_k,\hat{s}_k)\}_{k=0}^{N-1}$ and $\{J(p_k,q_k;c_k,s_k)\}_{k=0}^{N-1}$.

We now assume that, after the annihilation of the entry $z_k$ at the the $k$th iteration, its value will be affected only by a subset of the later Jacobi rotations with subscripts $k_{1}, k_{2}, \ldots, k_{r}$ with $r\le 2(n-2)$ being a function of $k$ \cite{Henrici-1958}. Let $\hat{z}_{k,j}$ be the computed value of $z_k$ after the  rotation $\hat{J}(p_{k_j},q_{k_j};\hat{c}_{k_j},\hat{s}_{k_j})$. Then we have by Lemma \ref{lem:pqrk},
\begin{eqnarray*}
&\hat{z}_{k,1} = \fl(\hat{z}_{k,0} \hat{c}_{k_1} \pm \hat{a}_{k_1} \hat{s}_{k_1} )
		    =  \hat{z}_{k,0}  \hat{c}_{k_1}(1+\delta_{k_1})(1+\delta_{k_1}') \pm \hat{a}_{k_1} \hat{s}_{k_1} (1 + \delta_{k_1}')(1+\delta_{k_1}''), &\\
&\hat{z}_{k,2} = \fl(\hat{z}_{i,1}\hat{c}_{k_2} \pm \hat{a}_{k_2} \hat{s}_{k_2}) =  \hat{z}_{i,1}\hat{c}_{k_2}(1+\delta_{k_2})(1+\delta_{k_2}')
\pm \hat{a}_{k_2} \hat{s}_{k_2} (1+\delta_{k_2}')(1+\delta_{k_2}''), &\\
&  \cdots \cdots&\\
&\hat{z}_{k,r} = \fl(\hat{z}_{i,r-1} \hat{c}_{k_r} \pm \hat{a}_{k_r} \hat{s}_{k_r})
			 =  \hat{z}_{i,r-1} \hat{c}_{k_r}(1+\delta_{k_r})(1+\delta_{k_r}')
\pm \hat{a}_{k_r} \hat{s}_{k_r}(1+\delta_{k_r}')(1+\delta_{k_r}''),&
\end{eqnarray*}
where  $\hat{a}_{k_j} $ stands for an entry of $A_{k_j}$ at either the $p_{k_j}$th or $q_{k_j}$th row or the $p_{k_j}$th or $q_{k_j}$th column except  the intersections of rows and columns $p_{k_j}$ and $q_{k_j}$.

By Lemma \ref{lea:JacobiAngle}, we have $|\hat{c}|\le 1+\tilde{\gamma}_1$ and $|\hat{s}|\le |s|(1+\tilde{\gamma}_1)$. Thus,
\BE \label{ieq:zir-base}
|\hat{z}_{k,j}| \le (|\hat{z}_{k,j-1}| + |\hat{a}_{k_j}| |s_{k_j}|)(1+\tilde{\gamma}_1) (1+\gamma_2) \le (|\hat{z}_{k,j-1}| + |\hat{a}_{k_j}| |s_{k_j}|)(1+\tilde{\gamma}_2),
\EE
for all $j=0,1,\ldots,r$. This implies that
\begin{eqnarray*}
|\hat{z}_{k,r}| & \le & (|\hat{z}_{k,r-1}| + |\hat{a}_{k_r}| |s_{k_r}|)(1+\tilde{\gamma}_2) \\
& \le & \big( (|\hat{z}_{k,r-2}| + |\hat{a}_{k_{r-1}}| |s_{k_{r-1}}|)(1+\tilde{\gamma}_2)+ |\hat{a}_{k_r}| |s_{k_r}|\big) (1+\tilde{\gamma}_2) \\
& \le & \cdots \\
& \le & |\hat{z}_{k,0}| (1+\tilde{\gamma}_2)^r + |\hat{a}_{k_1}| |s_{k_1}| (1+\tilde{\gamma}_2)^r + |\hat{a}_{k_2}| |s_{k_2}| (1+\tilde{\gamma}_2)^{r-1} + \cdots \\
				  & & + |\hat{a}_{k_{r-1}}| |s_{k_{r-1}}|(1+\tilde{\gamma}_2)^2 + |\hat{a}_{k_r}| |s_{k_r}| (1+\tilde{\gamma}_2) \\
&\le& (|\hat{z}_{k,0}| + |\hat{a}_{k_1}| |s_{k_1}|  + |\hat{a}_{k_2}| |s_{k_2}| + \cdots + |\hat{a}_{k_r}| |s_{k_r}|) (1+\tilde{\gamma}_2)^r.
\end{eqnarray*}

Without loss of generality, we assume that row and column effects are divided into two segments, $[1,r_1]$ and $[r_1+1,r_1+r_2]$ with $r = r_1+r_2$ (see Figure \ref{fig:diagram}), i.e.,
\begin{figure}[!h]
\centering
\includegraphics[width=0.35\textwidth,height=0.20\textheight]{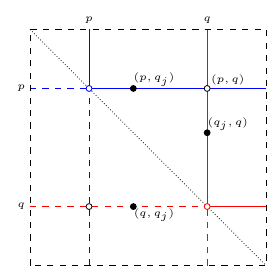}
\caption{The diagram of the row-effected indices (in blue) and the column-effected indices (in red).}
\label{fig:diagram}
\end{figure}
$
\hat{a}_{k_j} = e_{p_k}^T A_{k_j} e_{q_{k_j}} = e_{p_k}^T H_{k_j} e_{q_{k_j}}$ for $j = 1,2,\ldots,r_1$,
and $\hat{a}_{k_j} = e_{p_{k_j}}^T A_{k_j} e_{q_k} = e_{p_{k_j}}^T H_{k_j} e_{q_k}$ for $j =r_1+1,\ldots,r_1+r_2$,
where $H_{k_j}$'s are  defined as in Lemma \ref{lem:pqrk-err}.
Hence,
\BE \label{ieq:zir}
|\hat{z}_{k,r}| \le \Big(|\hat{z}_{k,0}| + \sum_{j=1}^{r_1}e_{p_k}^T|H_{k_j}|e_{q_{k_j}} |s_{k_j}| + \sum_{j=r_1+1}^{r_1+r_2}e_{p_{k_j}}^T|H_{k_j}|e_{q_k} |s_{k_j}| \Big) (1+\tilde{\gamma}_2)^r.
\EE
Note that, at the end of the complete set of $N$ rotations, the off-diagonal  entries of $\off(A_{N})$ is composed of $\hat{z}_{k,r}$'s. By \eqref{ieq:zir} we have
\begin{eqnarray*}
|\off(A_{N})| =  \sum_{k=0}^{N-1} |\hat{z}_{k,r}| \cdot E_k
\le \Big( \sum_{k=0}^{N-1} |\hat{z}_{k,0}| \cdot  E_k + \sum_{k=0}^{N-1}  |s_k| |H_{k}|\Big)(1+\tilde{\gamma}_2)^{2n-4},
\end{eqnarray*}
where $E_k\in\Rnn$ has zero entries except  $E_k(p_k,q_k)=1=E_k(q_k,p_k)$. Therefore,
\begin{eqnarray} \label{eq:offA-general-bound}
&& \|\off(A_N)\|_F =  \| |\off(A_N)| \|_F \nonumber\\
&\le& \Big\|\sum_{k=0}^{N-1} |\hat{z}_{k,0}| \cdot E_k\Big\|_F (1+\tilde{\gamma}_2)^{2n-4}
+ \Big\|\sum_{k=0}^{N-1}  |s_k| |H_{k}|\Big\|_F (1+\tilde{\gamma}_2)^{2n-4}.
\end{eqnarray}

In the following, we estimate the two items on the right-hand side of \eqref{eq:offA-general-bound}.
We note that $\hat{z}_{k,0}$ denotes the value of $z_k$ after its annihilation. By Lemma \ref{lem:pqrk} we have
\BE \label{eq:offA-general-bound1}
\Big\|\sum_{k=0}^{N-1} |\hat{z}_{k,0}|\cdot E_k\Big\|_F \le \Big(2\sum_{k=0}^{N-1} \|G_{k}\|_F^2\Big)^{1/2}\tilde{\gamma}_4.
\EE
Let $\tilde{\gamma}_4$ be such that
\BE \label{eq:con-sdescent}
\frac{1}{8}\varphi_k^2 \tilde{\gamma}_4^2 + \frac{1}{2}\|\off(A_k)\|_F \varphi_k\tilde{\gamma}_4 + \|H_k\|_F^2\tilde{\gamma}_4 + \|G_k\|_F^2\tilde{\gamma}_4^2 < |\hat{a}_{p_k q_k}^{(k)}|^2,
\EE
for  $k = 0,1,\ldots,N-1$, and
\BE \label{eq:con-accuerr}
d(A_0)\ge \underline{d} + \sum_{k=0}^{N-1} 2\varphi_k \tilde{\gamma}_4
\EE
for some $\underline{d}>0$. Then by  Lemmas \ref{lem:cond-ineq}--\ref{lea:diagdist} we have
\BE \label{eq:prerequisite-cond}
\|\off(A_{k})\|_F < d(A_{k}) / 4\quad\mbox{and} \quad |s_k|\le 2 |z_k|/d(A_k),
\EE
for $k = 0,1,\ldots,N-1$. Hence,
\begin{eqnarray} \label{eq:offA-general-bound2}
& &\Big\|\sum_{k=0}^{N-1}  |s_k| |H_{k}|\Big\|_F (1+\tilde{\gamma}_2)^{2n-4}
\le \sum_{k=0}^{N-1}  |s_k| \|H_{k}\|_F(1+\tilde{\gamma}_2)^{2n-4} \nonumber\\
&\le& (1+\tilde{\gamma}_2)^{2n-4} \max_{0\le k\le N-1} \|H_{k}\|_F \sum_{k=0}^{N-1}  |s_k|    \nonumber \\
& \le &  (1+\tilde{\gamma}_2)^{2n-4} \max_{0\le k\le N-1} \|H_{k}\|_F \Big(N \sum_{k=0}^{N-1} |s_k|^2\Big)^{1/2}  \nonumber \\
& \le &  (1+\tilde{\gamma}_2)^{2n-4} \max_{0\le k\le N-1} \|H_{k}\|_F  \Big(N\sum_{k=0}^{N-1} \frac{|2z_k|^2}{(d(A_k))^2}\Big)^{1/2} \nonumber \\
& \le &  (1+\tilde{\gamma}_2)^{2n-4}\sqrt{2N} \cdot \frac{\max_{0\le k\le N-1} \|H_{k}\|_F}{\min_{0 \le k \le N-1} d(A_{k})} \Big(\sum_{k=0}^{N-1} 2|z_k|^2\Big)^{1/2} \nonumber\\
& \le & \alpha_1 \|\off(A_{0})\|_F \Big(\|\off(A_{0})\|_F^2 + 2\sum_{k=0}^{N-1}\|H_{k}\|_F^2\tilde{\gamma}_4 + 2\sum_{k=0}^{N-1}\|G_{k}\|_F^2\tilde{\gamma}_4^2\Big)^{1/2} \nonumber \\
& \le & \alpha_1 \|\off(A_{0})\|_F^2 + \alpha_1 \sum_{k=0}^{N-1}\|H_{k}\|_F^2\tilde{\gamma}_4 + \alpha_1 \sum_{k=0}^{N-1}\|G_{k}\|_F^2\tilde{\gamma}_4^2,
\end{eqnarray}
where the third inequality uses the Cauchy-Schwarz inequality, the fourth inequality uses \eqref{eq:prerequisite-cond},  the sixth  inequality uses \eqref{ieq:sumz},
the last  inequality uses the fact that $a_1\sqrt{a_1^2+a_2}\le a_1^2 + \frac{1}{2}a_2$ for all $a_1,a_2\ge 0$, and
\BE \label{eq:rate-alp1}
\alpha_1 = (1+\tilde{\gamma}_2)^{2n-4} \sqrt{2N} \frac{\max_{0\le k\le N-1} \|H_{k}\|_F}{\min_{0 \le k \le N-1} d(A_{k})\|\off(A_0)\|_F}
\EE
Substituting  \eqref{eq:offA-general-bound1} and \eqref{eq:offA-general-bound2}  into \eqref{eq:offA-general-bound} yields
\begin{eqnarray}
 \|\off(A_N)\|_F &\le&  \alpha_1 \|\off(A_{0})\|_F^2 + \alpha_1 \sum_{k=0}^{N-1}\|H_{k}\|_F^2\tilde{\gamma}_4 + \alpha_1 \sum_{k=0}^{N-1}\|G_{k}\|_F^2\tilde{\gamma}_4^2 \nonumber \\
 && +\Big(2\sum_{k=0}^{N-1} \|G_{k}\|_F^2\Big)^{1/2}\tilde{\gamma}_4 (1+\tilde{\gamma}_2)^{2n-4}
 \le  \alpha_1 \|\off(A_0)\|_F^2 + \beta_1\tilde{\gamma}_4, \nonumber
\end{eqnarray}
where
\BE \label{eq:rate-beta1}
\beta_1 = \alpha_1\Big(\sum_{k=0}^{N-1}\|H_{k}\|_F^2 +\sum_{k=0}^{N-1}\|G_{k}\|_F^2\tilde{\gamma}_4\Big) + (1+\tilde{\gamma}_2)^{2n-4}\Big(2\sum_{k=0}^{N-1} \|G_{k}\|_F^2\Big)^{1/2}.
\EE

Finally, we show that  $\alpha_1<\infty$. It follows from  \eqref{eq:con-accuerr} and Lemma \ref{lea:Jacobi-dist}   that
\[
d(A_{k+1}) \ge d(A_k) - 2\varphi_k \tilde{\gamma}_4\ge\cdots\ge d(A_0)- \sum_{j=0}^{k}2\varphi_j \tilde{\gamma}_4
 \ge   d(A_0)- \sum_{j=0}^{N-1}2\varphi_j \tilde{\gamma}_4 \ge  \underline{d} >0,
\]
for  $k = 0,1,\ldots,N-1$. This implies that
$\min_{0 \le k \le N-1} d(A_{k})\ge \underline{d} >0$.
It follows from \eqref{eq:con-sdescent} and  Lemma \ref{lem:saij} that
$\|\off(A_{N})\|_F\le \|\off(A_{N-1})\|_F\le\cdots\le\|\off(A_{0})\|_F$ and thus
$\max_{0\le k\le N-1} \|\off(A_{k})\|_F\le\|\off(A_{0})\|_F$. Thus,
 $\alpha_1 \le (1+\tilde{\gamma}_2)^{2n-4} \sqrt{2N} /\underline{d} < \infty.
$

Based on the above analysis, we have the following result on the error bound for one sweep of the general cyclic Jacobi method in floating point arithmetic.
\begin{theorem} \label{thm:qc-de}
Let $A_k$ be the matrix $A_0=A$ after $k$ Jacobi updates, which is generated  by Algorithm {\rm\ref{alg:rcJacobi}} in floating point arithmetic, where the index pairs $\{(p_k,q_k)\}$ are chosen in a general cyclic order.  Let $\varphi_k= ((6+4\sqrt{2})\|G_k\|_F^2 + 2\|H_k\|_F^2)^{1/2}$ with $G_k$ and $H_k$ being defined as in Lemma {\rm \ref{lem:pqrk-err}} with $(p,q)=(p_k,q_k)$ for all $k\ge 0$. Suppose $A_0$ has $n$ distinct eigenvalues and $\|\off(A_0) \|_F < d(A_0)/4$. If  $\tilde{\gamma}_4$ satisfies  the conditions in \eqref{eq:con-sdescent}--\eqref{eq:con-accuerr},  then we have
$
 \|\off(A_N)\|_F \le \alpha_1 \|\off(A_0)\|_F^2 + \beta_1\tilde{\gamma}_4,
$
where the  constants $\alpha_1>0$ and $\beta_1>0$ are defined by \eqref{eq:rate-alp1} and \eqref{eq:rate-beta1}, respectively.

\end{theorem}
\begin{remark}
In Theorem {\rm \ref{thm:qc-de}}, it is easy to see that the condition \eqref{eq:con-sdescent} indicates the descent condition in Remark \ref{rem:dc}. The condition in \eqref{eq:con-accuerr}  is sufficient to guarantee that $2\varphi_k \tilde{\gamma}_4<d(A_k)$ {\rm(}which is not easy to check numerically{\rm)} for $k = 0,1,\ldots,N-1$ and $\alpha_1$ is a finite constant. We also assume that $\|\off(A_0) \|_F < d(A_0)/4$ for convenience. In fact, one may assume that  $\|\off(A_t) \|_F < d(A_t)/4$ for some $t>0$ and the corresponding error bound can be established similarly by replacing  the condition in \eqref{eq:con-accuerr} by
$
d(A_0)\ge \underline{d} + \sum_{k=0}^{t+N-1} 2\varphi_k \tilde{\gamma}_4.
$
This is reasonable   since the condition \eqref{eq:con-sdescent} is such that $\|\off(A_t) \|_F$ moves closer to zero and it follows from \eqref{eq:con-accuerr} and Lemma {\rm\ref{lea:Jacobi-dist}} that $d(A_t) \ge d(A_0)-\sum_{k=0}^{t-1} 2\varphi_k \tilde{\gamma}_4\ge \underline{d} >0$.
\end{remark}
\begin{remark}
We see from Theorem {\rm \ref{thm:qc-de}} that if $\tilde{\gamma}_4=0$, then $A_k$ has the same eigenvalues as $A_0$, $d(A_k)=d(A_0)$, and $\|H_k\|_F\le\|A_k\|_F\le\|A_0\|_F$ for all $k\ge 0$. Thus
$\|\off(A_N)\|_F \le \sqrt{2N} \|\off(A_0)\|_F^2/d(A_0)$.
This coincides  with Wilkinson's result  {\rm\cite{Wilkinson-1962}}.
\end{remark}

\subsubsection{The row-cyclic order with distinct eigenvalues} \label{sec322}
We study  the error analysis for one sweep of Algorithm \ref{alg:rcJacobi} in floating point arithmetic for a symmetric matrix $A\in\Rnn$ with $n$ distinct eigenvalues, where all off-diagonal entries are annihilated successively in  the row-cyclic order. For simplicity, as in section \ref{sec321}, we also assume that $\|\off(A_0) \|_F < d(A_0)/4$ and a fixed cycle of $N$ consecutive rotations of a row-cyclic ordering.
We still use $(p_k, q_k)$ as the chosen annihilation position at the $k$th iteration of the whole row-cyclic Jacobi method. During a fixed cycle of $N$ consecutive rotations of a row-cyclic ordering starting from $k=0$, the entries before annihilation are still denoted by $\{z_k\}_{k=0}^{N-1}$ and the computed and exact Jacobi rotations by $\{\hat{J}(p_k,q_k;\hat{c}_k,\hat{s}_k)\}_{k=0}^{N-1}$ and $\{J(p_k,q_k;c_k,s_k)\}_{k=0}^{N-1}$.

We first discuss, after the annihilation of the strictly upper diagonal entries in the first row, how their values are affected by the later Jacobi rotations. To illustrate the effects of the annihilation, we list the effected indices in Table \ref{tab:effect}.
\begin{table}[!ht]\renewcommand{\arraystretch}{0.7} \addtolength{\tabcolsep}{0.1pt}
\caption{Subsequent effects of annihilation on entries in positions $\{(1,q)\}_{q=2}^n$.} \label{tab:effect}
\begin{center}{\footnotesize
\begin{tabular}{ccc}
\toprule
position & row-indexed annihilation &  column-indexed annihilation \\ \midrule
$(1,2)$ & $(1,3), \ldots, (1,n)$ & $(2,3),\ldots,(2,n)$ \\
$(1,3)$ & $(1,4), \ldots, (1,n)$ & $(2,3),(3,4),\ldots,(3,n)$ \\
$\vdots $ & $\vdots$ & $\vdots$ \\
$(1,p)$ & $(1,p+1), \ldots, (1,n)$ & $(2,p),\ldots,(p-1,p),(p,p+1),\ldots,(p,n)$ \\
$\vdots $ & $\vdots$ & $\vdots$ \\
$(1,n-2)$ & $(1,n-1),(1,n)$  & $(2,n-2),\ldots,(n-3,n-2),(n-2,n-1),(n-2,n)$ \\
$(1,n-1)$ & $(1,n)$  & $(2,n-1),\ldots,(n-2,n-1),(n-1,n)$ \\
$(1,n)$ & /  & $(2,n),\ldots,(n-1,n)$ \\ \bottomrule
\end{tabular}}
\end{center}
\end{table}

We see from Table \ref{tab:effect} that the entries in the positions $\{(1,q)\}_{q=2}^n$ are affected exactly $n-2$ times after completing their annihilation.

As \eqref{ieq:zir-base}, for any  $2 \le q \le n-1$, we have
$|\hat{a}_{1q}^{(k+1)}| \le (|\hat{a}_{1q}^{(k)}| + |\hat{a}_{q,k+2}^{(k)}||s_k|)(1+\tilde{\gamma}_2)$, for $k = q-1,\ldots,n-2$. This indicates that
\begin{eqnarray*}
|\hat{a}_{1q}^{(n-1)}| & \le & (|\hat{a}_{1q}^{(n-2)}| + |\hat{a}_{q,n}^{(n-2)}||s_{n-2}|)(1+\tilde{\gamma}_2) \\
&\le & (|\hat{a}_{1q}^{(n-3)}| + |\hat{a}_{q,n-1}^{(n-3)}||s_{n-3}|)(1+\tilde{\gamma}_2)^2 + |\hat{a}_{q,n}^{(n-2)}||s_{n-2}|(1+\tilde{\gamma}_2) \\
& \le & \cdots \le  |\hat{a}_{1q}^{(q-1)}|(1+\tilde{\gamma}_2)^{n-q} + \sum_{k=q-1}^{n-2} |s_k| |\hat{a}_{q,k+2}^{(k)}| (1+\tilde{\gamma}_2)^{n-1-k}.
\end{eqnarray*}
Then, using the inequality $(|a_1|+|a_2|)^2 \le 2(a_1^2+a_2^2)$ for all $a_1,a_2\in\R$ and the Cauchy-Schwarz inequality, we have for any  $2 \le q \le n-1$,
\begin{eqnarray} \label{eq:squarebound}
&& |\hat{a}_{1q}^{(n-1)}|^2 \le  2|\hat{a}_{1q}^{(q-1)}|^2(1+\tilde{\gamma}_2)^{2(n-q)} + 2\Big(\sum_{k=q-1}^{n-2} |s_k| |\hat{a}_{q,k+2}^{(k)}| (1+\tilde{\gamma}_2)^{n-1-k}\Big)^2 \nonumber\\
&&\qquad\quad \le 2|\hat{a}_{1q}^{(q-1)}|^2(1+\tilde{\gamma}_2)^{2(n-q)} + 2\Big(\sum_{k=q-1}^{n-2} |s_k|^2 \Big) \Big(\sum_{k=q-1}^{n-2} |\hat{a}_{q,k+2}^{(k)}|^2 (1+\tilde{\gamma}_2)^{2(n-1-k)}\Big).
\end{eqnarray}
Thus,
\begin{eqnarray} \label{eq:ssrow}
&& \sum_{q=2}^{n} |\hat{a}_{1,q}^{(n-1)}|^2 \le 2\sum_{q=2}^{n} |\hat{a}_{1,q}^{(q-1)}|^2 (1+\tilde{\gamma}_2)^{2(n-q)} \nonumber\\
&& + 2 \Big( \sum_{k=0}^{n-2}|s_k|^2 \Big)\cdot \Big(\sum_{q=2}^{n-1}\sum_{k=q-1}^{n-2} |\hat{a}_{q,k+2}^{(k)}|^2 (1+\tilde{\gamma}_2)^{2(n-1-k)} \Big) \nonumber\\
& \le &  2\sum_{q=2}^n |\hat{a}_{1,q}^{(q-1)}|^2 (1+\tilde{\gamma}_2)^{2(n-q)} \nonumber\\
&&  + 2 \Big(\sum_{k=0}^{n-2}|s_k|^2\Big) \cdot \Big( \sum_{q=3}^n\sum_{p=2}^{q-1} |\hat{a}_{p,q}^{(q-2)}|^2 (1+\tilde{\gamma}_2)^{2(n+1-q)} \Big).
\end{eqnarray}

By Lemma \ref{lem:pqrk}, we have 
\[
\begin{cases} |\hat{a}_{q-1,q}^{(q-2)}|^2 + |\hat{a}_{1,q}^{(q-2)}|^2 \le (\hat{a}_{q-1,q}^{(q-3)}|^2 + |\hat{a}_{1,q}^{(q-3)}|^2)(1+2\tilde{\gamma}_4), \\[1mm]
\hat{a}_{pq}^{(q-2)} = \hat{a}_{pq}^{(q-3)} = \cdots = \hat{a}_{pq}^{(0)}, \; q \le p \le n
\end{cases}, \; q \ge 3
\]
and $\hat{a}_{ij}^{(q-2)} = \hat{a}_{ij}^{(q-3)}, 2 \le i \le q-2, q\le j \le n,\; q \ge 4$.
Then for any $q\ge 4$, we have
\begin{eqnarray*}
& & \sum_{p=1}^{q-1} |\hat{a}_{p,q}^{(q-2)}|^2
 =  \sum_{p=2}^{q-2} |\hat{a}_{p,q}^{(q-2)}|^2+ (|\hat{a}_{q-1,q}^{(q-2)}|^2 + |\hat{a}_{1,q}^{(q-2)}|^2)\\
& \le & \sum_{p=2}^{q-2} |\hat{a}_{p,q}^{(q-3)}|^2+ (|\hat{a}_{q-1,q}^{(q-3)}|^2 + |\hat{a}_{1,q}^{(q-3)}|^2)(1+2\tilde{\gamma}_4)\\
& = & \sum_{p=2}^{q-3} |\hat{a}_{p,q}^{(q-4)}|^2 + (|\hat{a}_{q-2,q}^{(q-3)}|^2+ |\hat{a}_{1,q}^{(q-3)}|^2) + |\hat{a}_{q-1,q}^{(q-3)}|^2 (1+2\tilde{\gamma}_4) + 2|\hat{a}_{1,q}^{(q-3)}|^2\tilde{\gamma}_4 \\
& \le & \sum_{p=2}^{q-4} |\hat{a}_{p,q}^{(q-5)}|^2 + (|\hat{a}_{q-3,q}^{(q-4)}|^2 + |\hat{a}_{1,q}^{(q-4)}|^2) + \sum_{p=q-2}^{q-1} |\hat{a}_{p,q}^{(p-2)}|^2 (1+2\tilde{\gamma}_4) + 2\sum_{k=q-4}^{q-3} |\hat{a}_{1,q}^{(k)}|^2\tilde{\gamma}_4\\
& \le & \cdots\\
& \le & (|\hat{a}_{2,q}^{(1)}|^2 + |\hat{a}_{1,q}^{(1)}|^2) + \sum_{p=3}^{q-1} |\hat{a}_{p,q}^{(p-2)}|^2 (1+2\tilde{\gamma}_4) + 2\sum_{k=1}^{q-3} |\hat{a}_{1,q}^{(k)}|^2\tilde{\gamma}_4\\
& \le & \sum_{p=1}^{q-1} |\hat{a}_{p,q}^{(0)}|^2 (1+2\tilde{\gamma}_4) + 2 \sum_{k=0}^{q-3} |\hat{a}_{1,q}^{(k)}|^2 \tilde{\gamma}_4.
\end{eqnarray*}
In addition, for $q = 3$, we have
$
\sum_{p=1}^{q-1} |\hat{a}_{p,q}^{(q-2)}|^2  =  |\hat{a}_{1,3}^{(1)}|^2 + |\hat{a}_{2,3}^{(1)}|^2 \le (|\hat{a}_{1,3}^{(0)}|^2 + |\hat{a}_{2,3}^{(0)}|^2)(1+2\tilde{\gamma}_4).
$
Therefore,
\begin{eqnarray} \label{eq:bound-row-square-sum}
&& \sum_{q=3}^n\sum_{p=2}^{q-1} |\hat{a}_{p,q}^{(q-2)}|^2(1+\tilde{\gamma}_2)^{2(n+1-q)} \nonumber\\
&&\quad \le \sum_{q=3}^n \Big(\sum_{p=1}^{q-1} |\hat{a}_{p,q}^{(0)}|^2 (1+2\tilde{\gamma}_4) + 2 \sum_{k=0}^{q-3} |\hat{a}_{1,q}^{(k)}|^2 \tilde{\gamma}_4\Big) (1+\tilde{\gamma}_2)^{2(n+1-q)}.
\end{eqnarray}

We now study, after the annihilation of the entries in the positions $\{(1,q)\}_{q=2}^n$, how their values are affected by the later rotations. We first consider how their values are affected after the annihilation of the entries in the positions $\{(2,q)\}_{q=3}^n$. We can easily check that
\begin{eqnarray*}
& &\sum_{q=2}^n |\hat{a}_{1,q}^{(2n-3)}|^2 \\
& \le & \big(|\hat{a}_{12}^{(2n-4)}|^2+|\hat{a}_{1n}^{(2n-4)}|^2\big)(1+2\tilde{\gamma}_4) + \sum_{q=3}^{n-1} |\hat{a}_{1,q}^{(2n-3)}|^2 \\
& = & \big(|\hat{a}_{12}^{(2n-4)}|^2+|\hat{a}_{1n}^{(2n-4)}|^2\big)(1+2\tilde{\gamma}_4) + \sum_{q=3}^{n-1} |\hat{a}_{1,q}^{(2n-4)}|^2 \\
& = & \big(|\hat{a}_{12}^{(2n-4)}|^2+|\hat{a}_{1n}^{(n-1)}|^2\big)(1+2\tilde{\gamma}_4) + \sum_{q=3}^{n-1} |\hat{a}_{1,q}^{(2n-4)}|^2 \\
& = & (|\hat{a}_{12}^{(2n-4)}|^2 +  |\hat{a}_{1,n-1}^{(2n-4)}|^2) + 2\sum_{q=n-2}^{n-2}|\hat{a}_{12}^{(n-2+q)}|^2\tilde{\gamma}_4 \\
&& + \sum_{q=n-2}^{n-2}|\hat{a}_{1,q+2}^{(n-1)}|^2(1+2\tilde{\gamma}_4) + \sum_{q=3}^{n-2} |\hat{a}_{1,q}^{(2n-4)}|^2 \\
& \le & (|\hat{a}_{12}^{(2n-5)}|^2 +  |\hat{a}_{1,n-1}^{(2n-5)}|^2)(1+2\tilde{\gamma}_4) + 2\sum_{q=n-2}^{n-2}|\hat{a}_{12}^{(n-2+q)}|^2\tilde{\gamma}_4 \\
&& + \sum_{q=n-2}^{n-2}|\hat{a}_{1,q+2}^{(n-1)}|^2(1+2\tilde{\gamma}_4) + \sum_{q=3}^{n-2} |\hat{a}_{1,q}^{(2n-5)}|^2\\
& = & (|\hat{a}_{12}^{(2n-5)}|^2 +  |\hat{a}_{1,n-2}^{(2n-5)}|^2) + 2\sum_{q=n-3}^{n-2}|\hat{a}_{12}^{(n-2+q)}|^2\tilde{\gamma}_4 \\
&& + \sum_{q=n-3}^{n-2}|\hat{a}_{1,q+2}^{(n-1)}|^2(1+2\tilde{\gamma}_4) + \sum_{q=3}^{n-3} |\hat{a}_{1,q}^{(2n-5)}|^2\\
& \le & \cdots\\
& \le & (|\hat{a}_{12}^{(n-1)}|^2 +  |\hat{a}_{13}^{(n-1)}|^2)(1+2\tilde{\gamma}_4) + 2\sum_{q=2}^{n-2}|\hat{a}_{12}^{(n-2+q)}|^2\tilde{\gamma}_4 + \sum_{q=2}^{n-2}|\hat{a}_{1,q+2}^{(n-1)}|^2(1+2\tilde{\gamma}_4)
\end{eqnarray*}
\begin{eqnarray*}
& = & |\hat{a}_{12}^{(n-1)}|^2 + 2\sum_{q=1}^{n-2}|\hat{a}_{12}^{(n-2+q)}|^2\tilde{\gamma}_4 + \sum_{q=1}^{n-2}|\hat{a}_{1,q+2}^{(n-1)}|^2(1+2\tilde{\gamma}_4) \\
& \le & 2\sum_{q=1}^{n-2}|\hat{a}_{12}^{(n-2+q)}|^2\tilde{\gamma}_4 + \sum_{q=0}^{n-2}|\hat{a}_{1,q+2}^{(n-1)}|^2(1+2\tilde{\gamma}_4) \\
& = & 2\sum_{q=1}^{n-2}|\hat{a}_{12}^{(n-2+q)}|^2\tilde{\gamma}_4 + \sum_{q=2}^{n}|\hat{a}_{1,q}^{(n-1)}|^2(1+2\tilde{\gamma}_4).
\end{eqnarray*}

To summarize, we have, for $r=2,\ldots,n-1$,
\begin{eqnarray*}
\sum_{q=r}^n |\hat{a}_{1,q}^{(nr-r(r+1)/2)}|^2
 \le  2\sum_{q=1}^{n-r}|\hat{a}_{1,r}^{(n(r-1)-(r-1)r/2+q-1)}|^2\tilde{\gamma}_4 + \sum_{q=r}^{n}|\hat{a}_{1,q}^{(n(r-1)-(r-1)r/2)}|^2(1+2\tilde{\gamma}_4).
\end{eqnarray*}
We note that
\[
\begin{cases}
\hat{a}_{12}^{(3n-6)}=\hat{a}_{12}^{(2n-3)}, \\[3pt]
\hat{a}_{12}^{(4n-10)}=\hat{a}_{12}^{(3n-6)}, \hat{a}_{13}^{(4n-10)}=\hat{a}_{13}^{(3n-6)}, \\[3pt]
\cdots \\
\hat{a}_{1j}^{(nr-r(r+1)/2)}=\hat{a}_{1j}^{(n(r-1)-(r-1)r/2)}, 2\le j \le r-1, \\[3pt]
\cdots \\
\hat{a}_{1j}^{(n-1)n/2}=\hat{a}_{1j}^{(n-1)n/2-1}, 2\le j \le n-2.
\end{cases}
\]
Let $\iota(r) \equiv nr - r(r+1)/2$, $r=1,2,\ldots,n-1$. Then
\begin{eqnarray*}
\sum_{q=r}^n |\hat{a}_{1,q}^{(\iota(r))}|^2  \le  2\sum_{q=1}^{n-r}|\hat{a}_{1,r}^{(\iota(r-1)+q-1)}|^2\tilde{\gamma}_4 + \sum_{q=r}^{n}|\hat{a}_{1,q}^{(\iota(r-1))}|^2(1+2\tilde{\gamma}_4)
\end{eqnarray*}
and thus
\begin{eqnarray*}
& & \sum_{q=2}^n |\hat{a}_{1,q}^{(\iota(r))}|^2  \le  2(\sum_{q=1}^{n-r}|\hat{a}_{1,r}^{(\iota(r-1)+q-1)}|^2 - \sum_{q=2}^{r-1} |\hat{a}_{1,q}^{(\iota(r-1))}|^2)\tilde{\gamma}_4 + \sum_{q=2}^{n}|\hat{a}_{1,q}^{(\iota(r-1))}|^2(1+2\tilde{\gamma}_4),
\end{eqnarray*}
which reads as follows:
\[
\begin{cases}
\sum_{q=2}^n |\hat{a}_{1,q}^{(2n-3)}|^2  \le \sum_{q=2}^{n}|\hat{a}_{1,q}^{(n-1)}|^2(1+2\tilde{\gamma}_4) + 2\sum_{q=1}^{n-2}|\hat{a}_{12}^{(n-2+q)}|^2\tilde{\gamma}_4,\\[3pt]
\sum_{q=2}^n |\hat{a}_{1,q}^{(3n-6)}|^2  \le  \sum_{q=2}^{n}|\hat{a}_{1,q}^{(2n-3)}|^2(1+2\tilde{\gamma}_4) + 2(\sum_{q=1}^{n-3}|\hat{a}_{13}^{(2n-3+q-1)}|^2 -|\hat{a}_{12}^{(2n-3)}|^2) \tilde{\gamma}_4,\\[3pt]
\sum_{q=2}^n |\hat{a}_{1,q}^{(4n-10)}|^2  \le  \sum_{q=2}^{n}|\hat{a}_{1,q}^{(3n-6)}|^2(1+2\tilde{\gamma}_4) + 2(\sum_{q=1}^{n-4}|\hat{a}_{14}^{(3n-6+q-1)}|^2 - \sum_{q=2}^3 |\hat{a}_{1q}^{(3n-6)}|^2) \tilde{\gamma}_4,\\[3pt]
\cdots  \cdots \\[3pt]
\end{cases}
\]
and
\[
\begin{cases}
\sum_{q=2}^n |\hat{a}_{1,q}^{(\iota(r))}|^2  \le \sum_{q=2}^{n}|\hat{a}_{1,q}^{(\iota(r-1))}|^2(1+2\tilde{\gamma}_4) + 2(\sum_{q=1}^{n-r}|\hat{a}_{1,r}^{(\iota(r-1)+q-1)}|^2 - \sum_{q=2}^{r-1} |\hat{a}_{1,q}^{(\iota(r-1))}|^2)\tilde{\gamma}_4, \\[3pt]
\cdots  \cdots \\[3pt]
\sum_{q=2}^n |\hat{a}_{1,q}^{((n-1)n/2)}|^2  \le \sum_{q=2}^{n}|\hat{a}_{1,q}^{((n(n-1)/2-1)}|^2(1+2\tilde{\gamma}_4) + 2(|\hat{a}_{1,n-1}^{(n(n-1)/2-1)}|^2 \\[3pt]
\qquad \qquad\qquad\qquad \qquad - \sum_{q=2}^{n-2} |\hat{a}_{1,q}^{(n(n-1)/2-1)}|^2 )\tilde{\gamma}_4.
\end{cases}
\]
This implies that
\begin{eqnarray} \label{eq:bound-row-square-sum2}
\sum_{q=2}^n |\hat{a}_{1,q}^{(N)}|^2
& \le & \sum_{q=2}^{n}|\hat{a}_{1,q}^{(n-1)}|^2(1+2\tilde{\gamma}_4)^{n-2} + 2\sum_{q=1}^{n-2}|\hat{a}_{12}^{(n-2+q)}|^2(1+2\tilde{\gamma}_4)^{n-3} \cdot \tilde{\gamma}_4 + \nonumber \\
& & 2\sum_{r=3}^{n-1}\Big(\sum_{q=1}^{n-r}|\hat{a}_{1,r}^{(\iota(r-1)+q-1)}|^2 - \sum_{q=2}^{r-1} |\hat{a}_{1,q}^{(\iota(r-1))}|^2 \Big)(1+2\tilde{\gamma}_4)^{n-1-r} \cdot \tilde{\gamma}_4 \nonumber \\
&\le & \sum_{q=2}^{n}|\hat{a}_{1,q}^{(n-1)}|^2(1+2\tilde{\gamma}_4)^{n-2}(1+2\zeta_1\tilde{\gamma}_4),
\end{eqnarray}
where $\zeta_1 = |\zeta_{11}+\zeta_{12}|/\mbox{$\sum_{q=2}^{n}|\hat{a}_{1,q}^{(n-1)}|^2(1+2\tilde{\gamma}_4)^{n-2}$}$
with $\zeta_{11} = \sum_{q=1}^{n-2}|\hat{a}_{12}^{(n-2+q)}|^2(1+2\tilde{\gamma}_4)^{n-3}$ and
\[
\zeta_{12}=\sum_{r=3}^{n-1}\Big(\sum_{q=1}^{n-r}|\hat{a}_{1,r}^{(\iota(r-1)+q-1)}|^2 - \sum_{q=2}^{r-1} |\hat{a}_{1,q}^{(\iota(r-1))}|^2\Big)(1+2\tilde{\gamma}_4)^{n-1-r}.
\]

Using \eqref{eq:ssrow}, \eqref{eq:bound-row-square-sum}, and \eqref{eq:bound-row-square-sum2} we have

\begin{eqnarray*}
 & & \sum_{q=2}^{n} |\hat{a}_{1,q}^{(N)}|^2 \le \sum_{q=2}^{n} |\hat{a}_{1,q}^{(n-1)}|^2 (1+2\tilde{\gamma}_4)^{n-2}(1+2\zeta_1\tilde{\gamma}_4) \nonumber \\
 & \le &2\sum_{q=2}^n |\hat{a}_{1,q}^{(q-1)}|^2 (1+\tilde{\gamma}_2)^{2(n-q)}(1+2\tilde{\gamma}_4)^{n-2}(1+2\zeta_1\tilde{\gamma}_4) + 2(1+2\tilde{\gamma}_4)^{n-2} (1+2\zeta_1\tilde{\gamma}_4)\Big(\sum_{k=0}^{n-2}|s_k|^2\Big) \cdot \nonumber \\
 & & \Big(\sum_{q=3}^n \sum_{p=1}^{q-1} |\hat{a}_{p,q}^{(0)}|^2 (1+2\tilde{\gamma}_4)(1+\tilde{\gamma}_2)^{2(n+1-q)} + 2 \sum_{q=3}^n \sum_{k=0}^{q-3} |\hat{a}_{1,q}^{(k)}|^2 \tilde{\gamma}_4 (1+\tilde{\gamma}_2)^{2(n+1-q)}\Big) \\
 & \le & 2\sum_{q=2}^n \|G_{q-2}\|_F^2\tilde{\gamma}_4^2 (1+\tilde{\gamma}_2)^{2(n-q)}(1+2\tilde{\gamma}_4)^{n-2}(1+2\zeta_1\tilde{\gamma}_4) + (1+2\tilde{\gamma}_4)^{n-2}(1+2\zeta_1\tilde{\gamma}_4)\Big(\sum_{k=0}^{n-2}|s_k|^2\Big) \cdot \\
 & & \Big(\|\off(A_0)\|_F^2 (1+2\tilde{\gamma}_4)(1+\tilde{\gamma}_2)^{2(n-2)} + 4 \sum_{q=3}^{n} \sum_{k=0}^{q-3} |\hat{a}_{1,q}^{(k)}|^2 \tilde{\gamma}_4 (1+\tilde{\gamma}_2)^{2(n+1-q)} \Big),
 \end{eqnarray*}
where the third inequality uses Lemma \ref{lem:pqrk} with  $G_k$ being defined as in Lemma \ref{lem:pqrk-err}.

To summarize,  we have, for $r=1,2,\ldots,n-1$,
\begin{eqnarray*}\label{inq:arqn}
& & \sum_{q=r+1}^{n} |\hat{a}_{r,q}^{(N)}|^2 = \sum_{q=2}^{n-r+1} |\hat{a}_{r-1+1,r-1+q}^{(\iota(r-1)+(n-r+1)(n-r)/2)}|^2 \nonumber 
\end{eqnarray*}
\begin{eqnarray}
& \le & 2 \sum_{q=2}^{n-r} \|G_{\iota(r-1)+q-2}\|_F^2\tilde{\gamma}_4^2 (1+\tilde{\gamma}_2)^{2(n-r+1-q)}(1+2\tilde{\gamma}_4)^{n-r-1}(1+2\zeta_{r}\tilde{\gamma}_4) \nonumber \\
& & + \sum_{k=0}^{n-r-1}|s_{\iota(r-1)+k}|^2 \cdot \Big(\|\off(A^{\iota(r-1)}_0)\|_F^2 (1+2\tilde{\gamma}_4)(1+\tilde{\gamma}_2)^{2(n-r-1)}  \nonumber \\
& & + 4 \sum_{q=3}^{n-r+1} \sum_{k=0}^{q-3} |\hat{a}_{r,r-1+q}^{(\iota(r-1)+k)}|^2 \tilde{\gamma}_4 (1+\tilde{\gamma}_2)^{2(n-r+2-q)}\Big)(1+2\tilde{\gamma}_4)^{n-r-1}(1+2\zeta_{r} \tilde{\gamma}_4),
\end{eqnarray}
where $A^{\iota(r-1)}_{k} = A_{\iota(r-1)+k}(r:n,r:n)$ and
\[
\zeta_r = |\zeta_{r1} + \zeta_{r2}| / \mbox{ $\sum_{q=2}^{n-r+1}|\hat{a}_{r,r-1+q}^{(\iota(r-1)+n-r)}|^2(1+2\tilde{\gamma}_4)^{n-r-1}$}
\]
with $\zeta_{r1}  =\sum_{q=1}^{n-r-1}|\hat{a}_{r,r+1}^{(\iota(r-1)+n-r-1+q)}|^2(1+2\tilde{\gamma}_4)^{n-r-2} $ and
\begin{eqnarray*}
\zeta_{r2} &=& \sum_{j=3}^{n-r}\sum_{q=1}^{n-r+1-j}|\hat{a}_{r,r-1+j}^{(\iota(r-1)+(n-r+1)(j-1)-(j-1)j/2+q-1)}|^2(1+2\tilde{\gamma}_4)^{n-r-j} \\
& & - \sum_{j=3}^{n-r}\sum_{q=2}^{j-1} |\hat{a}_{r,r-1+q}^{(\iota(r-1)+(n-r+1)(j-1)-(j-1)j/2)}|^2(1+2\tilde{\gamma}_4)^{n-r-j}.
\end{eqnarray*}
Thus,
\begin{eqnarray*}
& & \sum_{r=1}^{n-1}\sum_{q=r+1}^{n} |\hat{a}_{r,q}^{(N)}|^2 \nonumber\\
& \le & 2\sum_{r=1}^{n-1} \sum_{q=2}^{n-r} \|G_{\iota(r-1)+q-2}\|_F^2\tilde{\gamma}_4^2 (1+\tilde{\gamma}_2)^{2(n-r+1-q)}(1+2\tilde{\gamma}_4)^{n-r-1}(1+2\zeta_{r}\tilde{\gamma}_4) \nonumber \\
& & + \sum_{r=1}^{n-1}\sum_{k=0}^{n-r-1}|s_{\iota(r-1)+k}|^2 \cdot \Big(\|\off(A^{\iota(r-1)}_0)\|_F^2 (1+2\tilde{\gamma}_4)(1+\tilde{\gamma}_2)^{2(n-r-1)} \nonumber \\
& & + 4 \sum_{q=3}^{n-r+1} \sum_{k=0}^{q-3} |\hat{a}_{r,r-1+q}^{(\iota(r-1)+k)}|^2 \tilde{\gamma}_4 (1+\tilde{\gamma}_2)^{2(n-r+2-q)}\Big)(1+2\tilde{\gamma}_4)^{n-r-1}(1+2\zeta_{r} \tilde{\gamma}_4) \\
& \le &  2(1+2\zeta\tilde{\gamma}_4)(1+\tilde{\gamma}_2)^{2(n-2)}(1+2\tilde{\gamma}_4)^{n-2}\sum_{k=0}^{N-1} \|G_k\|_F^2 \tilde{\gamma}_4^2  \nonumber \\
& &  + (1+2\zeta\tilde{\gamma}_4) (1+\tilde{\gamma}_2)^{2(n-2)}(1+2\tilde{\gamma}_4)^{n-1}\xi^2 \sum_{k=0}^{N-1} |s_k|^2,\nonumber
\end{eqnarray*}
where  $\zeta,\xi\ge 0$ are defined by
\BE \label{eq:zetaxi}
\zeta = \max_{1\le r \le n-1}\zeta_r,\quad
\xi^2 = \max_{1\le r\le n-1} \|\off(A^{\iota(r-1)}_0)\|_F^2 + 4\sum_{q=3}^{n-r+1} \sum_{k=0}^{q-3} |\hat{a}_{r,r-1+q}^{(\iota(r-1)+k)}|^2 \tilde{\gamma}_4.
\EE

Suppose that the conditions \eqref{eq:con-sdescent} and \eqref{eq:con-accuerr} are satisfied. This implies that \eqref{eq:prerequisite-cond} holds.
Hence,
\begin{eqnarray*}
& & \|\off(A_{N})\|_F^2
 =  2\sum_{r=1}^{n-1}\sum_{q=r+1}^{n} |\hat{a}_{r,q}^{(N)}|^2 \\
& \le & 4\sum_{k=0}^{N-1} \|G_k\|_F^2 \tilde{\gamma}_4^2(1+\tilde{\gamma}_6)^{2(n-2)}(1+2\zeta\tilde{\gamma}_4) + 2\sum_{k=0}^{N-1} |s_k|^2 (1+\tilde{\gamma}_6)^{2(n-1)}(1+2\zeta\tilde{\gamma}_4) \xi^2\\
& \le & 4\sum_{k=0}^{N-1} \|G_k\|_F^2 \tilde{\gamma}_4^2(1+\tilde{\gamma}_6)^{2(n-2)}(1+2\zeta\tilde{\gamma}_4) + 2\sum_{k=0}^{N-1} \frac{4|z_k|^2}{(d(A_{k}))^2} (1+\tilde{\gamma}_6)^{2(n-1)}(1+2\zeta\tilde{\gamma}_4) \xi^2 \\
& \le & 4\sum_{k=0}^{N-1} \|G_k\|_F^2 \tilde{\gamma}_4^2(1+\tilde{\gamma}_6)^{2(n-2)}(1+2\zeta\tilde{\gamma}_4) + 4(1+\tilde{\gamma}_6)^{2(n-1)}(1+2\zeta\tilde{\gamma}_4)  \|\off(A_0)\|_F^2  \cdot\\
& &   \xi^2 \big(\min_{0 \le k \le N-1} (d(A_k))^2\|\off(A_0)\|_F^2\big)^{-1}\big( \|\off(A_0)\|_F^2 + 2\sum_{k=0}^{N-1}\|H_k\|_F^2\tilde{\gamma}_4 + 2\sum_{k=0}^{N-1}\|G_k\|_F^2\tilde{\gamma}_4^2 \big)\\
& \le & 4\sum_{k=0}^{N-1} \|G_k\|_F^2 \tilde{\gamma}_4^2(1+\tilde{\gamma}_6)^{2(n-2)}(1+2\zeta\tilde{\gamma}_4) + 4(1+\tilde{\gamma}_6)^{2(n-1)}(1+2\zeta\tilde{\gamma}_4) \cdot\\
& & \xi^2\big(\min_{0 \le k \le N-1} (d(A_k))^2\|\off(A_0)\|_F^2\big)^{-1} \Big( \|\off(A_0)\|_F^2 + \sum_{k=0}^{N-1}\|H_k\|_F^2\tilde{\gamma}_4 + \sum_{k=0}^{N-1}\|G_k\|_F^2\tilde{\gamma}_4^2 \Big)^2 \\
&\le & \Big(\alpha_2 \|\off(A_0)\|_F^2 + \beta_2\tilde{\gamma}_4\Big)^2,
\end{eqnarray*}
where the first inequality follows from $(1+2\tilde{\gamma}_4) \le (1+\tilde{\gamma}_4)^2$ and $(1+\tilde{\gamma}_2)(1+\tilde{\gamma}_4)\le (1+\tilde{\gamma}_6)$, the second inequality uses \eqref{eq:prerequisite-cond}, the third  inequality uses Lemma \ref{lem:saij}, and the fourth inequality  is obtained by applying $a_1^2(a_1^2+a_2) \le (a_1^2+a_2/2)^2$ for all $a_1,a_2\in\R$, and
\BE \label{eq:rate-alp2}
\alpha_2 = 2(1+\tilde{\gamma}_6)^{n-1}(1+2\zeta\tilde{\gamma}_4)^{1/2} \xi\big(\min_{0 \le k \le N-1} d(A_k) \|\off(A_0)\|_F\big)^{-1}
\EE
and
\begin{eqnarray} \label{eq:rate-beta2}
\beta_2 &=& \alpha_2\Big(\sum_{k=0}^{N-1}\|H_k\|_F^2 + \sum_{k=0}^{N-1}\|G_k\|_F^2\tilde{\gamma}_4\Big) \nonumber\\
&& + 2\Big(\sum_{k=0}^{N-1} \|G_k\|_F^2\Big)^{1/2}(1+\tilde{\gamma}_6)^{n-2}(1+2\zeta\tilde{\gamma}_4)^{1/2}.
\end{eqnarray}

Finally, we show that $\xi(\min_{0 \le k \le N-1} d(A_k) \|\off(A_0)\|_F)^{-1} <\infty$. Under the conditions \eqref{eq:con-sdescent} and \eqref{eq:con-accuerr},  it follows from  Lemmas \ref{lea:Jacobi-dist}  and   \ref{lem:saij} that $d(A_{k})\ge \underline{d} >0$ for all $0 \le k \le N-1$ and $\|\off(A_k^{\iota(r-1)})\|_F \le \|\off(A_{\iota(r-1)+k})\|_F\le\|\off(A_{0})\|_F$ for all $1\le r\le n-1$ and $k\ge 0$. Interchanging the sum
of $
\sum_{q=3}^{n-r+1} \sum_{k=0}^{q-3} |\hat{a}_{r,r-1+q}^{(\iota(r-1)+k)}|^2
$
yields
\begin{eqnarray*}
\xi^2 & = & \max_{1\le r\le n-1} \Big(\|\off(A^{\iota(r-1)}_0)\|_F^2 + 4\sum_{k=0}^{n-r-2} \sum_{q=3+k}^{n-r+1} |\hat{a}_{r,r-1+q}^{(\iota(r-1)+k)}|^2 \tilde{\gamma}_4 \Big) \\
& \le &  \max_{1\le r\le n-1} \Big( \|\off(A^{\iota(r-1)}_0)\|_F^2 + 4\sum_{k=0}^{n-r-2} \|\off(A_k^{\iota(r-1)})\|_F^2 \tilde{\gamma}_4 \Big)
\end{eqnarray*}
\begin{eqnarray*}
& \le & (1+4(n-2)\tilde{\gamma}_4)\|\off(A_0)\|_F^2.
\end{eqnarray*}
Therefore,
$
\xi\big(\min_{0 \le k \le N-1} d(A_k) \|\off(A_0)\|_F\big)^{-1} \le \sqrt{1+4n\tilde{\gamma}_4}/\underline{d} \le (1+2n\tilde{\gamma}_4)/\underline{d}<\infty.
$

Based on the above analysis, we have the following error bound for one sweep of  the row  cyclic Jacobi method in floating point arithmetic.
\begin{theorem} \label{thm:qc-rc}
Let $A_k$ be the matrix $A_0=A$ after $k$ Jacobi updates, which is generated  by Algorithm {\rm\ref{alg:rcJacobi}} in floating point arithmetic, where the index pairs $\{(p_k,q_k)\}$ are chosen in the row cyclic order.  Let $ \varphi_k= ((6+4\sqrt{2})\|G_k\|_F^2 + 2\|H_k\|_F^2)^{1/2}$ with $G_k$ and $H_k$ being defined as in Lemma {\rm \ref{lem:pqrk-err}}  with $(p,q)=(p_k,q_k)$ for all $k\ge 0$. Suppose $A_0$ has $n$ distinct eigenvalues and $\|\off(A_0) \|_F < d(A_0)/4$. If  $\tilde{\gamma}_4$ satisfies  the conditions in \eqref{eq:con-sdescent}--\eqref{eq:con-accuerr},  then we have
$
 \|\off(A_N)\|_F \le \alpha_2 \|\off(A_0)\|_F^2 + \beta_2\tilde{\gamma}_4,
$
where the  constants $\alpha_2>0$ and $\beta_2>0$ are defined by \eqref{eq:rate-alp2} and \eqref{eq:rate-beta2}, respectively.
\end{theorem}

\begin{remark}
In Theorem {\rm \ref{thm:qc-rc}}, we note that
$
\alpha_2<2(1+\tilde{\gamma}_6)^{n-1}(1+2\zeta\tilde{\gamma}_4)^{1/2}(1+2n\tilde{\gamma}_4)/\underline{d}
$
and
$
\alpha_1 \le (1+\tilde{\gamma}_2)^{2n-4} \sqrt{2N} /\underline{d}.
$
One may expect that  $\alpha_2$ is significantly smaller than $\alpha_1$ defined by \eqref{eq:rate-alp1} if $1+2\zeta\tilde{\gamma}_4$ is not too large for a sufficient small $\tilde{\gamma}_4$. This is observed from the later numerical results {\rm(}see Section {\rm\ref{sec63}}{\rm)}.
\end{remark}

\subsubsection{The row-cyclic order with one multiple eigenvalue}
We study  the error analysis for one sweep of Algorithm \ref{alg:rcJacobi} in floating point arithmetic for a symmetric matrix $A\in\Rnn$ with only one multiple eigenvalue, where all off-diagonal entries are annihilated successively in  the row-cyclic order.
Under the same assumptions as in section \ref{sec322}, without loss of generality, we assume that only the eigenvalue $\la_1(A)$ is multiple with multiplicity $m$ and that the diagonal elements $\hat{a}_{11}^{(k)},\ldots,\hat{a}_{mm}^{(k)}$ are sufficiently close to $\la_1(A)$   for all $k\ge 1$.

For any $k\ge 1$, we  define the intra cluster distance
\[
d_a(A_k) := \max_{1\le i,j \le m, i \ne j} |\la_i(A_k) - \la_j(A_k)|
\]
and the inter cluster distance
\[
d_r(A_k) := \min_{1 \le i \le n, m+1\le j \le n, i \ne j} |\la_i(A_k) - \la_j(A_k)|.
\]

We first extend Lemma \ref{lea:Jacobi-dist} to the case of  partially  distinct eigenvalues.
\begin{lemma} \label{lea:Jacobi-dist-revised}
If one step of Jacobi's method is performed in the $(p,q)$ plane on the matrix $A_k = (\hat{a}_{ij}^{(k)})$ with the computed Jacobi rotation $\hat{J}_k = J(p,q;\hat{c}_k,\hat{s}_k)$ and $A_k$ has $n-m$ distinct eigenvalues $\{\la_j(A_k)\}_{j=m+1}^n$ with
$
d_r(A_k)>0,
$
then we have
$
d_r(A_{k+1}) \ge  d_r(A_k) - 2\varphi_k \tilde{\gamma}_4>0,
$
provided that $2\varphi_k \tilde{\gamma}_4<d_r(A_k)$, where $\varphi_k = ((6+4\sqrt{2})\|G_k\|_F^2 + 2\|H_k\|_F^2)^{1/2}$ with $G_k$ and $H_k$ being defined as in Lemma {\rm \ref{lem:pqrk-err}}.
\end{lemma}

\begin{proof}
Let $J_k = J(p,q;c_k,s_k)$ be the exact Jacobi rotation. Then, by Lemma \ref{lea:perturbation-eig} we have, for any $1 \le i \le n$ and $m+1 \le j\le n$ with $i\neq j$,
\begin{eqnarray*}
& & |\la_i(A_{k+1}) - \la_j(A_{k+1}) | \\
& \ge & | \la_i(A_k) - \la_j(A_k) | - |\la_i(A_{k+1}) - \la_i(A_k)| - |\la_j(A_{k+1}) - \la_j(A_k)| \\
& = & | \la_i(A_k) - \la_j(A_k) | - |\la_i(A_{k+1}) - \la_i(J_k^T A_k J_k)| - |\la_j(A_{k+1}) - \la_j(J_k^T A_k J_k)| \\
		& \ge & d_r(A_k) - 2 \|A_{k+1}-J_k^T A_k J_k\|,
\end{eqnarray*}
completing the proof of the lemma by invoking Lemma \ref{lem:pqrk-err}.
\end{proof}
\begin{remark}\label{rem:akdm}
By Lemma {\rm\ref{lea:Jacobi-dist-revised}}, we have, for any $k\ge 0$,
\BE \label{ieq:bounddr}
d_r(A_{k+1}) \ge d_r(A_k) - 2\varphi_k \tilde{\gamma}_4 \ge d_r(A_0) - 2 \sum_{t=0}^{k}\varphi_t \tilde{\gamma}_4.
\EE
This show that if $2 \sum_{t=0}^{k}\varphi_t \tilde{\gamma}_4<d_r(A_0)$, then $d_r(A_{k+1})>0$, i.e., $A_{k+1}$ has $n-m$ distinct eigenvalues $\{\la_j(A_k)\}_{j=m+1}^n$ for all $k\ge 0$.
Moreover, by Lemmas {\rm\ref{lea:perturbation-eig}} and {\rm\ref{lem:pqrk-err}} with the fact that $d_a(A_0) = 0$ we have
\begin{eqnarray} \label{ieq:boundda}
d_a(A_{k+1}) &=& \max_{1\le i,j\le m, i\neq j} |\lambda_i(A_{k+1}) - \lambda_j(A_{k+1})| \nonumber \\
		 &\le& \max_{1\le i,j \le m, i\neq j} |\lambda_i(A_{k}) - \lambda_j(A_{k})|  \nonumber \\
&& + \max_{1\le i,j \le m, i\neq j} \big\{|\lambda_i(A_{k+1}) - \lambda_i(A_{k})| +  |\lambda_j(A_{k+1}) - \lambda_j(A_k)| \big\}\nonumber \\
		   &\le &  d_a(A_k) + 2\|Y_k\| \le \cdots \le  2 \sum_{t=0}^k \|Y_t\|\le 2\sum_{t=0}^k\varphi_t \tilde{\gamma}_4.
\end{eqnarray}
This shows that the eigenvalues $\{\la_j(A_k)\}_{j=1}^m$ of $A_{k+1}$ are close to each other in floating point arithmetic.
\end{remark}

We now extend Lemma \ref{lea:diagdist} to the case of partially  distinct eigenvalues.
\begin{lemma} \label{lem:diagdist-revised}
Let $A_k=(\hat{a}_{ij}^{(k)})$ be the matrix $A_0=A$ after $k$ Jacobi updates in the row-cyclic order.  Suppose $2 \sum_{t=0}^{k}\varphi_t \tilde{\gamma}_4<d_r(A_0)$ and the diagonal entries $\{\hat{a}_{tt}^{(k)}\}_{t=1}^m$ are sufficiently close to $\la_1(A)$ and the diagonal  entries $\{\hat{a}_{tt}^{(k)}\}_{t=m+1}^n$ are sufficiently close to the $n-m$ distinct eigenvalues $\{\la_t(A)\}_{t=m+1}^n$, respectively.  If $\|\off(A_k)\|_F\le d_r(A_k)/4$, then we have, for some ordering of $\{\la_t(A_k)\}_{t=1}^m$ and some ordering of $\{\la_t(A_k)\}_{t=m+1}^n$,
$|\hat{a}_{ii}^{(k)}-\hat{a}_{jj}^{(k)}| \ge d_r(A_k)/2$ for all $1 \le i \le n, m+1 \le j \le n, i\ne j$, and for $(p_k,q_k)$ plane on $A_k$ with $1 \le p_k \le n$ and $\max\{m+1,p_k+1\} \le q_k \le n$, the angle  $\theta_{p_kq_k}$ of the corresponding exact  Jacobi rotation  $J_k = J(p_k,q_k; c_k,s_k)$ satisfies
$
|\sin \theta_{p_kq_k}| \le 2| \hat{a}_{p_k,q_k}^{(k)}|/d_r(A_k).
$
\end{lemma}
\begin{proof}
It follows from Remark \ref{rem:akdm} that the $n-m$  eigenvalues $\{\la_j(A_k)\}_{j=m+1}^n$ are distinct and the eigenvalues $\{\la_j(A_k)\}_{j=1}^m$ are close to each other in floating point arithmetic. By hypothesis, $\{\hat{a}_{tt}^{(k)}\}_{t=1}^m$ are sufficiently close to $\la_1(A)$ and  $\{\hat{a}_{tt}^{(k)}\}_{t=m+1}^n$ are sufficiently close to the $n-m$ distinct eigenvalues $\{\la_t(A)\}_{t=m+1}^n$, respectively. Hence, $\{\hat{a}_{tt}^{(k)}\}_{t=1}^m$ are sufficiently close to each other and $\{\hat{a}_{tt}^{(k)}\}_{t=m+1}^n$ are all distinct.
Therefore, from Lemma \ref{lea:perturbation-eig} we have, for some ordering of $\{\la_t(A_k)\}_{t=1}^m$ and some ordering of $\{\la_t(A_k)\}_{t=m+1}^n$,
\begin{eqnarray*}
|\hat{a}_{ii}^{(k)}-\hat{a}_{jj}^{(k)}|
 & \ge & |\la_{i}(A_k) - \la_{j}(A_k)| - |\hat{a}_{ii}^{(k)} - \la_{i}(A_k)| - |\hat{a}_{jj}^{(k)} - \la_{j}(A_k)| \\
& \ge & d_r(A_k) - 2\|\off(A_k)\|_F \ge   d_r(A_k)/2,
\end{eqnarray*}
for all $1 \le i \le n, m+1 \le j \le n, i\ne j$. Then the second conclusion of this lemma follows from Lemma \ref{lea:diagdist}.
\end{proof}

By using the arguments similar to those of Lemmas \ref{lem:cond-ineq} and \ref{lea:Jacobi-dist-revised}, we have the following result.
\begin{lemma} \label{lem:cond-ineq-revised}
Suppose one step of Jacobi's method is performed in the $(p,q)$ plane on the matrix $A_k = (\hat{a}_{ij}^{(k)})$ with the computed Jacobi rotation $\hat{J}_k = J(p,q;\hat{c}_k,\hat{s}_k)$. Let $ \varphi_k= ((6+4\sqrt{2})\|G_k\|_F^2 + 2\|H_k\|_F^2)^{1/2}$ with $G_k$ and $H_k$ being defined as in Lemma {\rm \ref{lem:pqrk-err}}. If $A_k$ has  $n-m$ distinct eigenvalues $\{\la_j(A_k)\}_{j=m+1}^n$  with $4\|\off(A_k)\|_F <d_r(A_k)$ and
\BE \label{eq:sdescent-revised}
\frac{1}{8}\varphi_k^2 \tilde{\gamma}_4^2 + \frac{1}{2}\|\off(A_k)\|_F \varphi_k\tilde{\gamma}_4 + \|H_k\|_F^2\tilde{\gamma}_4 + \|G_k\|_F^2\tilde{\gamma}_4^2 < |\hat{a}_{pq}^{(k)}|^2,
\EE
then we have
$
4 \|\off(A_{k+1})\|_F < d_r(A_{k+1}).
$
\end{lemma}
\begin{proof}
Let $b_k = \|H_k\|_F^2\tilde{\gamma}_4 + \|G_k\|_F^2\tilde{\gamma}_4^2$.
From  \eqref{eq:sdescent-revised} we have $b_k< |\hat{a}_{pq}^{(k)}|^2$. By Lemma {\rm\ref{lem:saij}}, we have  
$\|\off(A_{k+1})\|_F<\|\off(A_k)\|_F$ and
$
\|\off(A_{k+1})\|_F^2 + 2(\hat{a}_{pq}^{(k)})^2 - 2b_k>0.
$
This, together with  \eqref{eq:sdescent-revised}  again, yields
\begin{eqnarray} \label{eq:Qeq-r-revised}
&&(\|\off(A_{k+1})\|_F  + \frac{1}{2}\varphi_k\tilde{\gamma}_4 )^2 -(\|\off(A_{k+1})\|_F^2 + 2|\hat{a}_{pq}^{(k)}|^2 - 2b_k) \nonumber\\
&\le& \frac{1}{4}\varphi_k^2 \tilde{\gamma}_4^2  + \|\off(A_{k+1})\|_F \varphi_k\tilde{\gamma}_4  + 2b_k - 2|\hat{a}_{pq}^{(k)}|^2  \nonumber\\
&\le& \frac{1}{4}\varphi_k^2 \tilde{\gamma}_4^2  + \|\off(A_{k})\|_F \varphi_k\tilde{\gamma}_4  + 2b_k - 2|\hat{a}_{pq}^{(k)}|^2<0.
\end{eqnarray}

By hypothesis, we have $4\|\off(A_k)\|_F < d_r(A_k)$. 
Then, by following the arguments similar to the proof of  Lemma \ref{lea:Jacobi-dist-revised} we have
\begin{eqnarray*}
d_r(A_{k+1}) & \ge & d_r(A_{k}) - 2\varphi_k \tilde{\gamma}_4  > 4 \|\off(A_k)\|_F -  2\varphi_k \tilde{\gamma}_4.  \\
& \ge & 4\big(\|\off(A_{k+1})\|_F^2 + 2(\hat{a}_{pq}^{(k)})^2 - 2b_k \big)^{1/2} - 2\varphi_k\tilde{\gamma}_4  
>  4 \|\off(A_{k+1})\|_F,
\end{eqnarray*}
where the third inequality uses  \eqref{diff:offam1m} and the last inequality uses  \eqref{eq:Qeq-r-revised}.
\end{proof}

Let $1\le k\le N$ be fixed.
As in \cite{vanKempen-1966}, we partition $A_k \in \Rnn$ into the following form
\BE \label{Ak:par}
A_k = \begin{bmatrix} B_1^{(k)} & B_2^{(k)} \\ (B_2^{(k)})^T & B_3^{(k)} \end{bmatrix},
\quad B_1^{(k)}\in\R^{m\times m},
\EE
where the diagonal entries of $B_1^{(k)}$  are sufficiently close to $\la_1(A)$.
We  estimate the quantity $\|\off(B_1^{(k)})\|_F$. The matrix $A_k$ admit the following spectral decomposition:
\[
A_k=Q_k \diag(\la_1(A_k),\ldots,\la_n(A_k))Q_k^T,
\]
where $Q_k\in\Rnn$ is orthogonal matrix. Then we have
\begin{eqnarray*}
&& \begin{bmatrix} B_1^{(k)} - \la_1(A_k) I_m & B_2^{(k)} \\ (B_2^{(k)})^T & B_3^{(k)} - \la_1(A_k) I_{n-m} \end{bmatrix}:=A_k - \la_1(A_k) I_n\\
&&\quad =Q_k\diag(\la_1(A_k) - \la_1(A_k), \ldots, \la_m(A_k) - \la_1(A_k), \ldots, \la_n(A_k) - \la_1(A_k))Q_k^T.
\end{eqnarray*}
Let
\begin{eqnarray*}
&&Q_k\diag(\la_1(A_k) - \la_1(A_k) , \ldots, \la_m(A_k)  - \la_1(A_k), 0, \ldots, 0)Q_k^T\\
&& :=\begin{bmatrix} W_1^{(k)} & W_2^{(k)} \\ (W_2^{(k)})^T & W_3 \end{bmatrix}=W_k,
\; W_1^{(k)}\in\R^{m\times m}.
\end{eqnarray*}
Then
\begin{eqnarray*}
&& A_k - \la_1(A_k) I_n-W_k= \begin{bmatrix} B_1^{(k)} - \la_1(A_k) I_m - W_1^{(k)} & B_2^{(k)} - W_2^{(k)} \\ (B_2^{(k)})^T - (W_2^{(k)})^T & B_3^{(k)} - \la_1(A_k) I_{n-m} - W_3^{(k)} \end{bmatrix} \\
&&\quad = Q_k\diag(0, \ldots, 0, \la_{m+1}(A_k) - \la_1(A_k), \ldots, \la_n(A_k) - \la_1(A_k))Q_k^T.
\end{eqnarray*}
We note that, for any $1\le i\le 3$,
$
\|W_i^{(k)}\| \le \|W_i^{(k)}\|_F \le \sqrt{m}\cdot d_a(A_k).
$
If
\BE \label{con:det}
\|\off(A_k)\|_F + \sqrt{m} d_a(A_k) < d_r(A_k),
\EE
then we have
\begin{eqnarray*}
&&|\la_j(B_3^{(k)}) - \la_1(A_k)| \ge |\la_{j+m}(A_k) - \la_1(A_k)| - |\la_{j+m}(A_k) - \la_j(B_3^{(k)})|\\
&&\quad \ge d_r(A_k) - \|\off(A_k)\|_F > \sqrt{m} d_a(A_k)
\end{eqnarray*}
for all $1\le j\le n-m$. This means
$\sigma_{\min}(B_3^{(k)} - \la_1(A_k) I_{n-m} ) - \|W_3^{(k)}\| > 0$ by Lemma \ref{lea:perturbation-svd} and thus $B_3^{(k)} - \la_1(A_k) I_{n-m} - W_3^{(k)}$ is nonsingular. By using congruent transformation we obtain
\[
C_k^T(A_k - \la_1(A_k) I_n-W_k) C_k = \begin{bmatrix} \Delta A^{(k)} & 0\\ 0 & B_3^{(k)} - \la_1(A_k) I_{n-m} - W_3^{(k)} \end{bmatrix},
\]
where
\begin{eqnarray*}
\Delta A^{(k)} &=& B_1^{(k)} - \la_1(A_k) I_m - W_1^{(k)}  \\
&& - (B_2^{(k)} - W_2^{(k)})
( B_3^{(k)} - \la_1(A_k) I_{n-m} - W_3^{(k)})^{-1}(B_2^{(k)} - W_2^{(k)})^T, \\
C_k &=& \begin{bmatrix} I_m & 0 \\ -(B_3^{(k)} - \la_1(A_k) I_{n-m} - W_3^{(k)})^{-1}(B_2^{(k)}-W_2^{(k)})^T  & I_{n-m} \end{bmatrix}.
\end{eqnarray*}
We note that, under the condition \eqref{con:det}, we have $d_r(A_k)>0$ and thus the rank of $A_k - \la_1(A_k) I_n-W_k$ is $n-m$.  Also, $C_k$ is nonsingular. Hence, $\Delta A^{(k)}=0$, i.e.,
\[
B_1^{(k)} - \la_1(A_k) I_m  =  W_1^{(k)} + (B_2^{(k)} - W_2^{(k)})
( B_3^{(k)} - \la_1(A_k) I_{n-m} - W_3^{(k)})^{-1}(B_2^{(k)} - W_2^{(k)})^T.
\]
This implies that
\begin{eqnarray} \label{ieq:offBi}
\|\off(B_1^{(k)})\|_F^2 & \le & 2\|W_1^{(k)}\|_F^2 + \frac{2\big(2\|B_2^{(k)}\|_F^2+2\|W_2^{(k)}\|_F^2\big)^2}{(\min_{1\le j\le n-m} |\la_j(B_3^{(k)}) - \la_1(A_k)| - \|W_3^{(k)}\|)^2} \nonumber \\
 & \le & 2m d_a^2(A_k) + \frac{2\big(2\|B_2^{(k)}\|_F^2+2m d_a^2(A_k)\big)^2}{\big(d_r(A_k) - \|\off(A_k)\|_F-\sqrt{m} d_a(A_k)\big)^2}.
\end{eqnarray}
Furthermore, if
\BE \label{eq:con-multipleerr}
\frac{4}{4\sqrt{m}+3}\underline{a} + 2\sum_{j=0}^{N-1} \varphi_j \tilde{\gamma}_4 < \frac{3}{4\sqrt{m}+3}d_r(A_0),
\EE
for some $\underline{a}>0$,
then using \eqref{ieq:bounddr} and \eqref{ieq:boundda}, we have
\begin{eqnarray*}
&& 4\sqrt{m} d_a(A_k)  \le 4\sqrt{m} \Big(2\sum_{j=0}^{k-1} \varphi_j \tilde{\gamma}_4\Big) \le 4\sqrt{m} \Big(2\sum_{j=0}^{N-1} \varphi_j \tilde{\gamma}_4\Big) \\
&&\quad < 3 \Big(d_r(A_0) - 2\sum_{j=0}^{N-1} \varphi_j \tilde{\gamma}_4\Big) \le 3 \Big(d_r(A_0) - 2\sum_{j=0}^{k-1} \varphi_j \tilde{\gamma}_4\Big)\le 3 d_r(A_k).
\end{eqnarray*}
Thus  \eqref{con:det} holds if $4 \|\off(A_k)\|_F < d_r(A_k)$.

Under conditions \eqref{eq:con-sdescent}, \eqref{eq:con-multipleerr}, and
\BE \label{eq:con-accuerr-r}
d_r(A_0)\ge \underline{d} + \sum_{k=0}^{N-1} 2\varphi_k \tilde{\gamma}_4
\EE
for some $\underline{d}>0$, it follows from  Lemma \ref{lem:cond-ineq-revised} that
\BE \label{eq:offbound}
\|\off(A_{k})\|_F \le d_r(A_{k})/4 \quad\mbox{and}\quad \|\off(A_k)\|_F \le \|\off(A_0)\|_F.
\EE
Using \eqref{eq:squarebound} we have
\begin{eqnarray} \label{eq:ssrow-v2}
\sum_{q=m+1}^{n} |\hat{a}_{1,q}^{(n-1)}|^2 &\le& 2\sum_{q=m+1}^{n} |\hat{a}_{1,q}^{(q-1)}|^2 (1+\tilde{\gamma}_2)^{2(n-q)} \nonumber\\
&& + 2 \Big( \sum_{k=m}^{n-2}|s_k|^2 \Big)\cdot \Big(\sum_{q=m+1}^{n-1}\sum_{k=q-1}^{n-2} |\hat{a}_{q,k+2}^{(k)}|^2 (1+\tilde{\gamma}_2)^{2(n-1-k)} \Big) \nonumber\\
& \le &  2\sum_{q=m+1}^n |\hat{a}_{1,q}^{(q-1)}|^2 (1+\tilde{\gamma}_2)^{2(n-q)} \nonumber \\
&&+ 2 \Big(\sum_{k=m}^{n-2}|s_k|^2\Big) \cdot \Big( \sum_{q=m+2}^n\sum_{p=m+1}^{q-1} |\hat{a}_{p,q}^{(q-2)}|^2 (1+\tilde{\gamma}_2)^{2(n+1-q)} \Big) \nonumber \\
& \le &  2\sum_{q=2}^n |\hat{a}_{1,q}^{(q-1)}|^2 (1+\tilde{\gamma}_2)^{2(n-q)} \nonumber\\
&&+ 2 \Big(\sum_{k=m}^{n-2}|s_k|^2\Big) \cdot \Big( \sum_{q=3}^n\sum_{p=2}^{q-1} |\hat{a}_{p,q}^{(q-2)}|^2 (1+\tilde{\gamma}_2)^{2(n+1-q)} \Big).
\end{eqnarray}
Using \eqref{eq:ssrow-v2}, \eqref{eq:bound-row-square-sum}, and \eqref{eq:bound-row-square-sum2} we have
\begin{eqnarray} \label{ineq:a1qn}
\sum_{q=2}^{n} |\hat{a}_{1,q}^{(N)}|^2 &\le& \sum_{q=2}^{n} |\hat{a}_{1,q}^{(n-1)}|^2 (1+2\tilde{\gamma}_4)^{n-2}(1+2\zeta_1\tilde{\gamma}_4) \nonumber\\
&\le & \sum_{q=2}^{m}|\hat{a}_{1,q}^{(n-1)}|^2(1+2\tilde{\gamma}_4)^{n-2}(1+2\zeta_1\tilde{\gamma}_4)  \nonumber\\
&& + \sum_{q=m+1}^{n}|\hat{a}_{1,q}^{(n-1)}|^2(1+2\tilde{\gamma}_4)^{n-2}(1+2\zeta_1\tilde{\gamma}_4) \nonumber\\
&\le & \sum_{q=2}^{m}|\hat{a}_{1,q}^{(n-1)}|^2(1+2\tilde{\gamma}_4)^{n-2} (1+2\zeta_1\tilde{\gamma}_4)\nonumber\\
& &   + 2\sum_{q=2}^n |\hat{a}_{1,q}^{(q-1)}|^2 (1+\tilde{\gamma}_2)^{2(n-q)}(1+2\tilde{\gamma}_4)^{n-2}(1+2\zeta_1\tilde{\gamma}_4)\nonumber \\
& &  + 2(1+2\tilde{\gamma}_4)^{n-2} (1+2\zeta_1\tilde{\gamma}_4)\Big(\sum_{k=m}^{n-2}|s_k|^2\Big) \cdot \nonumber\\
 & & \Big(\sum_{q=3}^n \big(\sum_{p=1}^{q-1} |\hat{a}_{p,q}^{(0)}|^2 (1+2\tilde{\gamma}_4) + 2 \sum_{k=0}^{q-3} |\hat{a}_{1,q}^{(k)}|^2 \tilde{\gamma}_4\big) (1+\tilde{\gamma}_2)^{2(n+1-q)}\Big) \nonumber
  \end{eqnarray}
 \begin{eqnarray}
 & \le & \sum_{q=2}^{m}|\hat{a}_{1,q}^{(n-1)}|^2(1+2\tilde{\gamma}_4)^{n-2} (1+2\zeta_1\tilde{\gamma}_4) \nonumber \\
 & & + 2\sum_{q=2}^n |\hat{a}_{1,q}^{(q-1)}|^2 (1+\tilde{\gamma}_2)^{2(n-q)}(1+2\tilde{\gamma}_4)^{n-2}(1+2\zeta_1\tilde{\gamma}_4) \nonumber \\
 & & + (1+2\tilde{\gamma}_4)^{n-1} (1+2\zeta_1\tilde{\gamma}_4)\Big(\sum_{k=m}^{n-2}|s_k|^2\Big)(1+\tilde{\gamma}_2)^{2(n-2)} \cdot \nonumber \\
 & & \big(\|\off(A_0)\|_F^2 + 4 \sum_{q=3}^{n}\sum_{k=0}^{q-3} |\hat{a}_{1,q}^{(k)}|^2 \tilde{\gamma}_4\big) \nonumber \\
 & \le & \sum_{q=2}^{m}|\hat{a}_{1,q}^{(n-1)}|^2(1+2\tilde{\gamma}_4)^{n-2} (1+2\zeta\tilde{\gamma}_4) \nonumber\\
  && + 2\sum_{q=2}^n \|G_{q-2}\|_F^2\tilde{\gamma}_4^2  (1+\tilde{\gamma}_6)^{2(n-2)}(1+2\zeta\tilde{\gamma}_4) \nonumber \\
 & & + (1+\tilde{\gamma}_6)^{2(n-1)} (1+2\zeta\tilde{\gamma}_4)\Big(\sum_{k=m}^{n-2}|s_k|^2\Big)\xi^2,
 \end{eqnarray}
where the last inequality uses the definitions of $\zeta$ and $\xi^2$ as in \eqref{eq:zetaxi} and the fact that $(1+2\tilde{\gamma}_4) \le (1+\tilde{\gamma}_4)^2$ and $(1+\tilde{\gamma}_2)(1+\tilde{\gamma}_4)\le (1+\tilde{\gamma}_6)$.

It follows from that \eqref{eq:con-multipleerr} that
\BE\label{ineq:under-a}
\frac{3}{4}\Big(d_r(A_0)-2\sum_{t=0}^{N-1}\varphi_t\tilde{\gamma}_4\Big)-2\sqrt{m} \sum_{t=0}^{N-1}\varphi_t \tilde{\gamma}_4\ge \underline{a}>0.
\EE
From \eqref{ieq:bounddr}, \eqref{ieq:boundda}, \eqref{Ak:par}, \eqref{ieq:offBi}, \eqref{eq:offbound}, and \eqref{ineq:under-a} we have
 \begin{eqnarray*}
 & & \sum_{q=2}^{m}|\hat{a}_{1,q}^{(n-1)}|^2 \le \frac{1}{2}\|\off(B_1^{(n-1)})\|_F^2 \\
 & \le & m d_a^2(A_{n-1}) + \frac{\big(2\|\off(A_{n-1})\|_F^2+2m d_a^2(A_{n-1})\big)^2}{\big(d_r(A_{n-1}) - \|\off(A_{n-1})\|_F-\sqrt{m} d_a(A_{n-1})\big)^2} \\
 & \le & m d_a^2(A_{n-1}) + \frac{\big(2\|\off(A_0)\|_F^2+2m d_a^2(A_{n-1})\big)^2}{\big(\frac{3}{4}d_r(A_{n-1})-\sqrt{m} d_a(A_{n-1})\big)^2}\\
 & \le & 4m \Big(\sum_{t=0}^{N-1}\varphi_t\Big)^2 \tilde{\gamma}_4^2+\frac{\big(2\|\off(A_0)\|_F^2+8m (\sum_{t=0}^{N-1}\varphi_t\big)^2 \tilde{\gamma}_4^2\big)^2}{\big(\frac{3}{4}(d_r(A_0)-2\sum_{t=0}^{N-1}\varphi_t\tilde{\gamma}_4)-2\sqrt{m} \sum_{t=0}^{N-1}\varphi_t \tilde{\gamma}_4\big)^2}\\
 & \le & 4m \Big(\sum_{t=0}^{N-1}\varphi_t\Big)^2 \tilde{\gamma}_4^2 + \underline{a}^{-2}\Big(2\|\off(A_0)\|_F^2+8m \big(\sum_{t=0}^{N-1}\varphi_t\big)^2 \tilde{\gamma}_4^2\Big)^2.
 \end{eqnarray*}
 This, together with \eqref{ineq:a1qn} and \eqref{eq:zetaxi}, yields
 \begin{eqnarray*}
&&\sum_{q=2}^{n} |\hat{a}_{1,q}^{(N)}|^2 \le 2\sum_{q=2}^n \|G_{q-2}\|_F^2\tilde{\gamma}_4^2(1+\tilde{\gamma}_6)^{2(n-2)}(1+2\zeta\tilde{\gamma}_4) + \Big(\sum_{k=m}^{n-2}|s_k|^2\Big)(1+\tilde{\gamma}_6)^{2(n-1)}(1+2\zeta\tilde{\gamma}_4) \xi^2 \\
 & &\quad + \Big(4m \big(\sum_{t=0}^{N-1}\varphi_t\big)^2 \tilde{\gamma}_4^2 + \underline{a}^{-2}\big(2\|\off(A_0)\|_F^2+8m \big(\sum_{t=0}^{N-1}\varphi_t\big)^2 \tilde{\gamma}_4^2\big)^2\Big)(1+2\tilde{\gamma}_4)^{n-2}(1+2\zeta\tilde{\gamma}_4).
 \end{eqnarray*}

 To summarize, we have, for $r=1,\ldots,m-1$,
 \begin{eqnarray}\label{ineq:a1qN-rn}
&& \sum_{q=r+1}^{n} |\hat{a}_{r,q}^{(N)}|^2 = \sum_{q=2}^{n-r+1} |\hat{a}_{r-1+1,r-1+q}^{(\iota(r-1)+(n-r+1)(n-r)/2)}|^2 \nonumber\\
&&\quad \le  \Big(4m (\sum_{t=0}^{N-1}\varphi_t)^2 \tilde{\gamma}_4^2 + \underline{a}^{-2}\big(2\|\off(A_0)\|_F^2+8m (\sum_{t=0}^{N-1}\varphi_t)^2 \tilde{\gamma}_4^2\big)^2\Big)\cdot\nonumber\\
&&\quad (1+\tilde{\gamma}_4)^{2(n-2)}(1+2\zeta\tilde{\gamma}_4)
+ 2 \sum_{q=2}^{n-r} \|G_{\iota(r-1)+q-2}\|_F^2\tilde{\gamma}_4^2(1+\tilde{\gamma}_6)^{2(n-2)}(1+2\zeta\tilde{\gamma}_4) \nonumber\\
& & \quad + \sum_{k=m+1-r}^{n-r-1}|s_{\iota(r-1)+k}|^2 (1+\tilde{\gamma}_6)^{2(n-1)}(1+2\zeta\tilde{\gamma}_4) \xi^2,
\end{eqnarray}
where $A^{\iota(r-1)}_{k} = A_{\iota(r-1)+k}(r:n,r:n)$. Using \eqref{inq:arqn}, \eqref{eq:zetaxi} and the fact that $(1+2\tilde{\gamma}_4) \le (1+\tilde{\gamma}_4)^2$ and $(1+\tilde{\gamma}_2)(1+\tilde{\gamma}_4)\le (1+\tilde{\gamma}_6)$, we obtain, for $r=m,\ldots,n-1$,
\begin{eqnarray} \label{ineq:arqN-2}
\sum_{q=r+1}^{n} |\hat{a}_{r,q}^{(N)}|^2 & \le & 2 \sum_{q=2}^{n-r} \|G_{\iota(r-1)+q-2}\|_F^2\tilde{\gamma}_4^2(1+\tilde{\gamma}_6)^{2(n-2)}(1+2\zeta\tilde{\gamma}_4) \nonumber \\
& &+ \sum_{k=0}^{n-r-1}|s_{\iota(r-1)+k}|^2 (1+\tilde{\gamma}_6)^{2(n-1)}(1+2\zeta\tilde{\gamma}_4) \xi^2.
\end{eqnarray}

Let $\Omega = \cup_{1 \le p<q \le m}\{(p,q)\}$. By Lemma \ref{lem:diagdist-revised} we have
\begin{eqnarray}\label{ineq:sk-notin-O}
\sum_{(p_k,q_k) \notin \Omega} |s_k|^2 & \le & \frac{4\sum_{(p_k,q_k) \notin \Omega} |\hat{a}_{p_k,q_k}^{(k)}|^2}{\min_{(p_k,q_k)\notin \Omega}d_r^2(A_k)} \le \frac{4\sum_{k=0}^{N-1} |\hat{a}_{p_k,q_k}^{(k)}|^2}{\min_{(p_k,q_k) \notin \Omega}d_r^2(A_k)}  \nonumber\\
& \le & 2\frac{\|\off(A_0)\|_F^2 + 2\sum_{k=0}^{N-1}\|H_k\|_F^2\tilde{\gamma}_4 + 2\sum_{k=0}^{N-1}\|G_k\|_F^2\tilde{\gamma}_4^2}{{\min_{(p_k,q_k) \notin \Omega}d_r^2(A_k)}},
\end{eqnarray}
where the last inequality follows from \eqref{ieq:sumz}.

From \eqref{ineq:a1qN-rn}, \eqref{ineq:arqN-2}, and \eqref{ineq:sk-notin-O} we have
\begin{eqnarray*}
& & \|\off(A_N)\|_F^2 = 2\sum_{r=1}^{n-1}\sum_{q=r+1}^{n} |\hat{a}_{r,q}^{(N)}|^2=2\sum_{r=1}^{m-1}\sum_{q=r+1}^{n} |\hat{a}_{r,q}^{(N)}|^2 +2\sum_{r=m}^{n-1}\sum_{q=r+1}^{n} |\hat{a}_{r,q}^{(N)}|^2 \\
		& \le & 4\sum_{r=1}^{n-1}\sum_{q=2}^{n-r} \|G_{\iota(r-1)+q-2}\|_F^2\tilde{\gamma}_4^2(1+2\zeta\tilde{\gamma}_4)(1+\tilde{\gamma}_6)^{2(n-2)} \\
		& & + 2(\sum_{r=1}^{m-1} \sum_{k=m+1-r}^{n-r-1} + \sum_{r=m}^{n-1}\sum_{k=0}^{n-r-1})|s_{\iota(r-1)+k}|^2 (1+2\zeta\tilde{\gamma}_4)(1+\tilde{\gamma}_6)^{2(n-1)} \xi^2 \\
		& & + 2(m-1)\Big(4m (\sum_{t=0}^{N-1}\varphi_t)^2 \tilde{\gamma}_4^2 + \underline{a}^{-2}\big(2\|\off(A_0)\|_F^2+8m (\sum_{t=0}^{N-1}\varphi_t)^2 \tilde{\gamma}_4^2\big)^2\Big)(1+\tilde{\gamma}_4)^{2(n-2)}(1+2\zeta\tilde{\gamma}_4)\\
		& = & 4\sum_{k=0}^{N-1}\|G_k\|_F^2\tilde{\gamma}_4^2(1+2\zeta\tilde{\gamma}_4)(1+\tilde{\gamma}_6)^{2(n-2)} + 2\sum_{(p_k,q_k) \notin \Omega}|s_k|^2 (1+2\zeta\tilde{\gamma}_4)(1+\tilde{\gamma}_6)^{2(n-1)} \xi^2 \\
		& & + 2(m-1)\Big(4m (\sum_{t=0}^{N-1}\varphi_t)^2 \tilde{\gamma}_4^2 + \underline{a}^{-2} \big(2\|\off(A_0)\|_F^2+8m (\sum_{t=0}^{N-1}\varphi_t)^2 \tilde{\gamma}_4^2\big)^2\Big)(1+\tilde{\gamma}_4)^{2(n-2)}(1+2\zeta\tilde{\gamma}_4)\\
        & \le & 4\sum_{k=0}^{N-1}\|G_k\|_F^2\tilde{\gamma}_4^2(1+2\zeta\tilde{\gamma}_4)(1+\tilde{\gamma}_6)^{2(n-2)}
         + 4\Big(\min_{(p_k,q_k) \notin \Omega}d_r^2(A_k)\|\off(A_0)\|_F^2\Big)^{-1}\|\off(A_0)\|_F^2 \cdot\\
        & & \Big(\|\off(A_0)\|_F^2 + 2\sum_{k=0}^{N-1}\|H_k\|_F^2\tilde{\gamma}_4 + 2\sum_{k=0}^{N-1}\|G_k\|_F^2\tilde{\gamma}_4^2\Big) (1+2\zeta\tilde{\gamma}_4)(1+\tilde{\gamma}_6)^{2(n-1)} \xi^2 \\
		& & + 2(m-1)\Big(4m (\sum_{t=0}^{N-1}\varphi_t)^2 \tilde{\gamma}_4^2 + \underline{a}^{-2}\big(2\|\off(A_0)\|_F^2+8m (\sum_{t=0}^{N-1}\varphi_t)^2 \tilde{\gamma}_4^2\big)^2\Big)(1+\tilde{\gamma}_4)^{2(n-2)}(1+2\zeta\tilde{\gamma}_4)\\
& \le & 4\sum_{k=0}^{N-1}\|G_k\|_F^2\tilde{\gamma}_4^2(1+2\zeta\tilde{\gamma}_4)(1+\tilde{\gamma}_6)^{2(n-2)}
         + 4\Big(\min_{(p_k,q_k) \notin \Omega}d_r^2(A_k)\|\off(A_0)\|_F^2\Big)^{-1}\cdot\\
        & & \Big(\|\off(A_0)\|_F^2 + \sum_{k=0}^{N-1}\|H_k\|_F^2\tilde{\gamma}_4 + \sum_{k=0}^{N-1}\|G_k\|_F^2\tilde{\gamma}_4^2\Big)^2 (1+2\zeta\tilde{\gamma}_4)(1+\tilde{\gamma}_6)^{2(n-1)} \xi^2 \\
		& & + 2(m-1)\Big(4m (\sum_{t=0}^{N-1}\varphi_t)^2 \tilde{\gamma}_4^2 + \underline{a}^{-2}\big(2\|\off(A_0)\|_F^2+8m (\sum_{t=0}^{N-1}\varphi_t)^2 \tilde{\gamma}_4^2\big)^2\Big)(1+\tilde{\gamma}_4)^{2(n-2)}(1+2\zeta\tilde{\gamma}_4)\\
		& \le & (\alpha_2' \|\off(A_0)\|_F^2 + \beta_2' \tilde{\gamma}_4)^2 \\
		& & + 2(m-1)\big(4m (\sum_{t=0}^{N-1}\varphi_t)^2 \tilde{\gamma}_4^2 + \underline{a}^{-2}\big(2\|\off(A_0)\|_F^2+8m (\sum_{t=0}^{N-1}\varphi_t)^2 \tilde{\gamma}_4^2\big)^2\Big)(1+\tilde{\gamma}_4)^{2(n-2)}(1+2\zeta\tilde{\gamma}_4)\\
		& \le & \Big(\alpha_2' \|\off(A_0)\|_F^2 + \beta_2' \tilde{\gamma}_4+ (1+\tilde{\gamma}_4)^{n-2}(1+2\zeta\tilde{\gamma}_4)^{1/2}\sqrt{8(m-1)m}(\sum_{t=0}^{N-1}\varphi_t)\tilde{\gamma}_4 \\
		& & + \sqrt{2(m-1)}(1+\tilde{\gamma}_4)^{n-2}(1+2\zeta\tilde{\gamma}_4)^{1/2} \underline{a}^{-1} \big(2\|\off(A_0)\|_F^2+8m (\sum_{t=0}^{N-1}\varphi_t)^2 \tilde{\gamma}_4^2\big)\Big)^2 \\
&=& \big(\alpha_3 \|\off(A_0)\|_F^2 + \beta_3 \tilde{\gamma}_4\big)^2,
\end{eqnarray*}
where the  third inequality uses the fact that $a_1^2(a_1^2+a_2) \le (a_1^2+a_2/2)^2$ for all $a_1,a_2\in\R$, the fourth inequality  uses the fact that $a_1+a_2 \le (\sqrt{a_1}+\sqrt{a_2})^2$ for all $a_1,a_2\ge 0$, the last inequality  uses the fact that $a_1^2+a_2^2+a_3^2 \le (a_1+a_2+a_3)^2$ for all $a_1,a_2,a_3\ge 0$, and
\[
\alpha_2' = 2(1+\tilde{\gamma}_6)^{n-1}(1+2\zeta\tilde{\gamma}_4)^{1/2} \xi\big(\min_{(p_k,q_k) \notin \Omega} d_r(A_k) \|\off(A_0)\|_F\big)^{-1},
\]
\[
\beta_2' = \alpha_2'\Big(\sum_{k=0}^{N-1}\|H_k\|_F^2 + \sum_{k=0}^{N-1}\|G_k\|_F^2\tilde{\gamma}_4\Big) + 2\Big(\sum_{k=0}^{N-1} \|G_k\|_F^2\Big)^{1/2}(1+\tilde{\gamma}_6)^{n-2}(1+2\zeta\tilde{\gamma}_4)^{1/2},
\]
\BE \label{eq:rate-alp3}
\alpha_3 =  \alpha_2' + 2\underline{a}^{-1}\sqrt{2(m-1)}(1+\tilde{\gamma}_4)^{n-2}(1+2\zeta\tilde{\gamma}_4)^{1/2},
\EE
and
\begin{eqnarray}\label{eq:rate-beta3}
\beta_3 &=& \beta_2' + \sqrt{8(m-1)m} (1+\tilde{\gamma}_4)^{n-2}(1+2\zeta\tilde{\gamma}_4)^{1/2}\sum_{i=0}^{N-1} \varphi_i \nonumber\\
&& + \underline{a}^{-1}\Big(8m\sqrt{2(m-1)}(1+\tilde{\gamma}_4)^{n-2}(1+2\zeta\tilde{\gamma}_4)^{1/2} (\sum_{i=0}^{N-1} \varphi_i )^2\Big)\tilde{\gamma}_4.
\end{eqnarray}

Finally, we show that $\xi(\min_{0 \le k \le N-1} d(A_k) \|\off(A_0)\|_F)^{-1} <\infty$. Under the conditions \eqref{eq:con-sdescent}, \eqref{eq:con-multipleerr}, and \eqref{eq:con-accuerr-r},  it follows from   Lemmas \ref{lea:Jacobi-dist-revised}  and  \ref{lem:saij} that $d(A_{k})\ge \underline{d} >0$ for all $0 \le k \le N-1$ and $\|\off(A_k^{\iota(r-1)})\|_F \le \|\off(A_{\iota(r-1)+k})\|_F\le\|\off(A_{0})\|_F$ for all $1\le r\le n-1$ and $k\ge 0$. Then, interchanging the sum of $\sum_{q=3}^{n-r+1} \sum_{k=0}^{q-3} |\hat{a}_{r,r-1+q}^{(\iota(r-1)+k)}|^2$ yields
\begin{eqnarray*}
\xi^2 & = & \max_{1\le r\le n-1} \Big(\|\off(A^{\iota(r-1)}_0)\|_F^2 + 4\sum_{k=0}^{n-r-2} \sum_{q=3+k}^{n-r+1} |\hat{a}_{r,r-1+q}^{(\iota(r-1)+k)}|^2 \tilde{\gamma}_4 \Big) \\
& \le &  \max_{1\le r\le n-1} \Big( \|\off(A^{\iota(r-1)}_0)\|_F^2 + 4\sum_{k=0}^{n-r-2} \|\off(A_k^{\iota(r-1)})\|_F^2 \tilde{\gamma}_4 \Big)\\
& \le & (1+4(n-2)\tilde{\gamma}_4)\|\off(A_0)\|_F^2.
\end{eqnarray*}
Thus,
$
\xi\big(\min_{0 \le k \le N-1} d(A_k) \|\off(A_0)\|_F\big)^{-1} \le \sqrt{1+4n\tilde{\gamma}_4}/\underline{d} \le (1+2n\tilde{\gamma}_4)/\underline{d}<\infty.
$

Based on the above analysis, we have the following error bound for one sweep of the row  cyclic Jacobi method in floating point arithmetic.
\begin{theorem} \label{thm:qc-rc-om}
Let $A_k$ be the matrix $A_0=A$ after $k$ Jacobi updates, which is generated by Algorithm {\rm\ref{alg:rcJacobi}} in floating point arithmetic, where the index pairs $\{(p_k,q_k)\}$ are chosen in the row cyclic order.  Let $ \varphi_k= ((6+4\sqrt{2})\|G_k\|_F^2 + 2\|H_k\|_F^2)^{1/2}$ with $G_k$ and $H_k$ being defined as in Lemma {\rm \ref{lem:pqrk-err}} with $(p,q)=(p_k,q_k)$ for all $k\ge 0$. Suppose $A_0$ has only one  multiple eigenvalue $\la_1(A)$ with multiplicity $m$ and the remaining $n-m$ eigenvalues $\{\la_t(A)\}_{t=m+1}^n$ are distinct, and $\|\off(A_0) \|_F < d_r(A_0)/4$. If $\tilde{\gamma}_4$ satisfies the conditions \eqref{eq:con-sdescent}, \eqref{eq:con-multipleerr}, and \eqref{eq:con-accuerr-r}, then we have
$
 \|\off(A_N)\|_F \le \alpha_3 \|\off(A_0)\|_F^2 + \beta_3\tilde{\gamma}_4,
$
where the  constants $\alpha_3>0$ and $\beta_3>0$ are defined by \eqref{eq:rate-alp3} and \eqref{eq:rate-beta3}, respectively.
\end{theorem}

\section{A mixed precision preconditioned Jacobi method for the symmetric
eigenvalue problem} \label{Sec:Alg}
In this section, we propose  a mixed precision preconditioned Jacobi algorithm for computing the eigenvalue decomposition  of an $n$-by-$n$ real symmetric matrix $A$. We first compute an approximate eigenvalue decomposition $A\approx ZDZ^T$ in low precision $\upsilon$, where $Z\in\Rnn$ and $D=\diag(d_1,\ldots,d_n)$. To improve the orthogonality of $Z$, we use the MGS method  to $Z$ in high precision $\omega \ll \upsilon$ and we obtain a high accuracy orthogonal matrix $Q\in\Rnn$. One may also employ Householder orthogonalization to orthogonalize  $Z$, which is less computationally efficient than MGS \cite[\S 5.2.9]{Golub.VanLoan-2013}. Then we use the matrix $Q$ as an initial guess for the Jacobi method for computing the eigenvalue decomposition of $A$.

Based on the above analysis, we present our mixed precision preconditioned Jacobi algorithm  in Algorithm {\rm \ref{alg:evd-main}}, where ``In precision $\omega$" means computed and stored in precision $\omega$. 
\begin{algorithm}[!htb]
	\caption{A mixed precision preconditioned cyclic Jacobi method for the symmetric
eigenvalue decomposition.} \label{alg:evd-main}
	\begin{algorithmic}[1]	
	\REQUIRE A symmetric matrix $A\in\Rnn$ and a tolerance $\epsilon>0$.
	\STATE $[Z,\Phi] = \text{eig}\left(A\right)$ \hfill $\triangleright$  Symmetric QR factorization in  precision $\upsilon$ and store $Z$ at precision $\omega$
	\STATE $Q = \text{MGS}\left(Z\right)$ \hfill $\triangleright$ MGS orthogonalization in  precision $\omega$
	\STATE Set $T = Q^T A Q$ and $P=Q$ \hfill $\triangleright$ In precision $\omega$
	\WHILE {$\| \off(T) \|_F > \epsilon \| A\|_F$}
	  \FOR {$i=1,\ldots n-1$}
     \FOR{$j=i+1,\ldots, n$}
	\STATE  \qquad\qquad Compute a cosine-sine group $(c,s)$ as in Lemma  \ref{lem:offa-jr} with $A=T$. $\triangleright$ In precision $\omega$
    \STATE \qquad\qquad Set $T=J(i,j,\theta)^TTJ(i,j,\theta)$ and $P=PJ(i,j,\theta)$.
    $\triangleright$ In precision $\omega$
    \STATE \quad\qquad\hspace*{0.02in}{\bf end if} \hspace*{0.02in}
     \ENDFOR
       \ENDFOR
	\ENDWHILE
	\end{algorithmic}
\end{algorithm}

We observe from Theorems \ref{thm:qc-rc} and  \ref{thm:qc-rc-om} that, to take advantage of the Jacobi method as much as possible, it is desirable that  the Frobenius norm of the off-diagonal entries of the matrix $Q^T A Q$ in Step 3 of Algorithm {\rm\ref{alg:evd-main}}  is small enough.

In the following, we give the error analysis of Algorithm \ref{alg:evd-main}. We first give an estimate of the distance between $Z$ and $Q$ generated by Algorithm  \ref{alg:evd-main}.

\begin{lemma} \label{lea:main}
Suppose that  $Z \in\Rnn$ in Algorithm {\rm\ref{alg:evd-main}} is computed by any eigensolver in LAPACK or EISPACK in precision $\upsilon$ and $Q \in\Rnn$ in Algorithm {\rm\ref{alg:evd-main}}  is computed by using the MGS method to $Z$ in precision $\omega$ $(\omega \ll \upsilon)$.
Then there exist constants $h_t\equiv h_t(n)$ for $t=1,2$  such that
$
\|Z - Q \|_F \le h_1 \upsilon + h_2\omega.
$
\end{lemma}

\begin{proof}
By Lemma \ref{lea:symeig}, there exists a matrix $\delta Z\in\Rnn$ such that $Z + \delta Z$ is orthogonal, where $\| \delta Z \| \le p(n)\upsilon$. Thus,
\BE \label{v:ub}
\|Z\|\le \|Z + \delta Z\| + \|\delta Z\|=1+\|\delta Z\|\le 1+p(n)\upsilon\equiv c_1.
\EE
This, together with the orthogonality of $Z + \delta Z$, yields
\[
\left\{
\begin{array}{l}
\la_{\max}(Z^T Z)\le 1 + (2\|\delta Z \| \|Z\| + \| \delta Z\|^2) \le 1 + \chi_1, \\
\la_{\min}(Z^T Z)\ge 1- (2 \|\delta Z \| \|Z\| + \| \delta Z\|^2) \ge 1 - \chi_1,
\end{array}
\right.
\]
provided that $\chi_1=(2c_1+ p(n)\upsilon)p(n) \upsilon <1$. Hence, $Z$ is nonsingular and $\rank(Z) = n$. In addition,
\BE \label{eq:kv}
\kappa(Z)=\frac{\sigma_{\max}(Z)}{\sigma_{\min}(Z)}
=\sqrt{\frac{\la_{\max}(Z^T Z)}{\la_{\min}(Z^T Z)}}
\le \sqrt{\frac{1 + \chi_1}{1 - \chi_1}} .
\EE
It follows from Lemma \ref{lea:MGS}  and \eqref{v:ub} that there exists an upper triangular matrix $R\in\Rnn$ such that
\BE \label{eq:qrfac}
Z = Q R + F_1, \quad \left\| F_1 \right\| \le \eta_1 \| Z\| \omega \le  \eta_1c_1\omega\equiv c_2\omega,
\EE
where $Q$ is such that
\BE \label{eq:qrorth}
Q^T Q = I + F_2, \quad \left\| F_2 \right\| \le \eta_2\kappa(Z) \omega \le \sqrt{\frac{1 + \chi_1}{1 - \chi_1}}\cdot \eta_2\omega\equiv c_3\omega.
\EE

We now give an upper bound and a lower bound to each diagonal entry of $R$.
From \eqref{eq:qrorth} we have
\BE \label{qnm:ub}
\| Q \| \le \sqrt{1+\| F_2 \|} \le 1 + \| F_2 \|.
\EE
Using  \eqref{eq:qrfac} and \eqref{eq:qrorth}  we have
$
Q^T Z = R + F_2 R + Q^T F_1,
$
which implies that
\[
\| R \|  \le \| Q\| \| Z \| + \| F_2 \| \| R \| + \| Q \| \| F_1 \| .
\]
Thus,
\begin{eqnarray} \label{rnm:ub}
\| R \| &\le & \frac{\| Q \|}{1 - \| F_2 \|} \left(\| Z \| + \| F_1 \| \right) \le  \frac{1 + \| F_2 \|}{1- \| F_2 \|}\left(\| Z \| + \| F_1 \| \right) \nonumber\\
& \le & \Big(1+ \frac{2\| F_2 \|}{1- \| F_2 \|} \Big)\left(\| Z \| + \| F_1 \| \right) \nonumber \\
			& \le & \Big(1+ \frac{2\chi_2}{1-\chi_2}\Big)(1 + \eta_1\omega)(1+p(n)\upsilon)
			\equiv 1 + g_1\upsilon,
\end{eqnarray}
provided that $\chi_2=c_3\omega<1$, where
\[
g_1= \eta_1\omega\upsilon^{-1} + (1+\eta_1\omega)p(n)+\frac{2c_3\omega\upsilon^{-1} + 2\eta_1\chi_2\omega\upsilon^{-1} + 2\chi_2p(n) + 2\eta_1\chi_2 p(n)\omega} {1-\chi_2}.
\]
Similarly, we have by \eqref{eq:qrfac},
\[
(Z+\delta Z) R^{-1} = Q + (F_1 + \delta Z) R^{-1}.
\]
This, together with the orthogonality of $Z+ \delta Z$, yields
\[
\| R^{-1} \| =  \| (Z+\delta Z) R^{-1} \|   \le  \| Q \| + \left(\| F_1 \| + \| \delta Z \| \right) \| R^{-1}\|.
\]
Using  \eqref{eq:qrfac}--\eqref{qnm:ub}  we have
\BE\label{rinv:ub}
\| R^{-1} \| \le \frac{\| Q \| }{1-(\| F_1  \| + \| \delta Z \|)  }  \le   \frac{1 + \| F_2 \|}{1-(\| F_1  \| + \| \delta Z \|) }
				   \le  \frac{1+c_3\omega}{1-\chi_3} \equiv 1 + g_2\upsilon,
\EE
provided that $\chi_3=c_2\omega +p(n)\upsilon<1$,
where $g_2 = ((c_2+c_3)\omega\upsilon^{-1} + p(n))/(1-\chi_3)$.

If $g_2\upsilon<1$, then
\BE \label{eq:boundr1}
1-g_2\upsilon \le \frac{1}{1+g_2\upsilon} \le \frac{1}{\| R^{-1}  \|}\le r_{ss} \le \|R\| \le 1+g_1\upsilon, \; s=1,2,\ldots,n.
\EE

Next, we bound the spectral norm of the strictly upper triangular part of $R$. Substituting \eqref{eq:qrfac} and \eqref{eq:qrorth} into $(Z+\delta Z)^T(Z+\delta Z) = I$ yields
\[
I - R^T R = R^T F_2 R + (F_1+\delta Z)^T Q R + R^T Q^T (F_1 + \delta Z) + (F_1+\delta Z)^T (F_1 + \delta Z).
\]
Then, by \eqref{eq:qrfac}, \eqref{eq:qrorth},  \eqref{qnm:ub}, and  \eqref{rnm:ub}   we have
\begin{eqnarray*}
&&\| I- R^T R \|  \le  \| R \|^2 \| F_2 \| + 2 \left(\| F_1 \| + \| \delta Z\| \right) \| Q \| \| R \| + \left(\| F_1 \| + \| \delta Z\| \right)^2 \\
& \le & (1+ g_1\upsilon)^2  c_3\omega + 2 ( c_2\omega +p(n)\upsilon)\sqrt{1+\chi_2}(1+g_1v) + ( c_2\omega +p(n)\upsilon)^2
\equiv  c_4\upsilon,
\end{eqnarray*}
where
$
c_4=c_3(1+ g_1\upsilon)^2\omega\upsilon^{-1} + 2 ( c_2\omega\upsilon^{-1} +p(n))\sqrt{1+\chi_2}(1+g_1v) + \chi_3( c_2\omega\upsilon^{-1} +p(n)).
$
This, together with \eqref{rinv:ub}, yileds
\[
\| R^{-1} - R^T \| \le \| I - R^T R \| \| R^{-1}\| \le  c_4(1+g_2\upsilon)\upsilon\equiv c_5\upsilon,
\]
which implies that
\BE \label{eq:boundr2}
|\hat{r}_{st}| \le \| R^{-1} - R^T \| \le c_5\upsilon\quad \forall s<t.
\EE
It follows from  \eqref{eq:boundr1} and \eqref{eq:boundr2} that
\[
\| R - I \|_F  =  \sqrt{\sum_{i=1}^n(\hat{r}_{ii}-1)^2+\sum_{i<j}\hat{r}_{ij}^2}
\le \upsilon \sqrt{n(\max\{g_1,g_2\})^2+\frac{n(n-1)}{2}c_5^2} \equiv  c_6\upsilon.
\]
Therefore, by \eqref{eq:qrfac},  \eqref{eq:qrorth}, and \eqref{qnm:ub}  we have
\begin{eqnarray*}
\| Z - Q \|_F & \le &  \| Q \| \| R - I \|_F + \| F_1 \|_F
				  \le  \sqrt{1+\| F_2 \|} \| R - I \|_F + \sqrt{n}\| F_1 \| \\
				  & \le & \sqrt{1+c_3\omega}\cdot c_6 \upsilon + \sqrt{n}\cdot c_2\omega
				  \equiv h_1 \upsilon + h_2\omega.
\end{eqnarray*}
\end{proof}



On the Frobenius norm of the off-diagonal entries of  the matrix $Q^T A Q$ generated in Step 3 of Algorithm {\rm\ref{alg:evd-main}}, we have the following result.
\begin{theorem} \label{thm:evd-ubd}
Suppose that  $Z \in \Rnn$ in  Algorithm {\rm\ref{alg:evd-main}} is computed by any eigensolver in LAPACK or EISPACK  in precision  $\upsilon$  and $Q\in\Rnn$  in  Algorithm {\rm\ref{alg:evd-main}}  is computed by using the MGS method to $Z$ in precision $\omega$ $(\omega \ll \upsilon)$. Then there exists a constant $\tilde{\psi}_1\equiv \tilde{\psi}_1(n)$ such that
$
\| \off(Q^T A Q)  \|_F  \le \tilde{\psi}_1 \| A \| \upsilon.
$
\end{theorem}
\begin{proof}
By Lemma \ref{lea:symeig} we know there exists a symmetric matrix $E\in\Rnn$ such that $A+E$ admits the following exact eigenvalue decomposition:
\BE \label{eq:eigfac}
A + E = (Z+\delta Z) D (Z+\delta Z)^T,\quad D=\diag(d_1,\ldots,d_n)\in\Rnn,
\EE
where $\|E\| \le p(n) \| A \| \upsilon$ and $Z+\delta Z$ is orthogonal with $\| \delta Z \| \le p(n)\upsilon$.
We note that $Q$ is computed by using the MGS method to $Z$ with precision $\omega$. It follows that
\begin{eqnarray} \label{eq:qaq-d}
Q^T A Q - D   &=& (Q - Z + Z)^T A (Q - Z + Z) - D \nonumber \\
				&=& (Q - Z )^T A (Q - Z ) + (Q - Z )^T A Z + Z^T A (Q - Z ) + (Z^T A Z - D).
\end{eqnarray}

On the other hand, using \eqref{eq:eigfac} we have
$
D=(Z+\delta Z)^T(A + E)(Z+\delta Z),
$
i.e.,
\[
D-Z^T A Z=Z^TA\delta Z+\delta Z^TAZ+\delta Z^TA\delta Z+(Z+\delta Z)^TE(Z+\delta Z).
\]
Thus,
\begin{eqnarray*}
\| Z^T A Z - D \|_F &\le &  2\| A \|_F \| Z \| \| \delta Z \| +  \| A \|_F \| \delta Z \|^2 + \| E \|_F\\
&\le & \| A \|_F (2 \| Z \|+ \| \delta Z \| )  \| \delta Z \| + \| E \|_F.
\end{eqnarray*}

By Lemma \ref{lea:main} and using \eqref{v:ub} and \eqref{eq:qaq-d}, we have
\begin{eqnarray}\label{t0:ubv}
&&  \| \off(Q^T A Q)  \|_F  \le  \| Q^T A Q - D \|_F \nonumber \\
&& \quad \le \| (Q - Z )^T A (Q - Z )\|_F  + 2\| (Q - Z )^T A Z \|_F + \| Z^T A Z - D \|_F   \nonumber \\
&& \quad \le  \| A \| \| Q - Z \|_F^2 + 2 \| A \| \| Z \| \| Q - Z \|_F + \| A \|_F( 2 \| Z \| +  \| \delta Z \|)\| \delta Z \|  +  \| E \|_F  \nonumber \\
&& \quad \le  \| A \| \| Q - Z \|_F(\| Q - Z \|_F + 2\| Z \|)  + \sqrt{n}\| A \| (2\| Z \| +  \| \delta Z \|) \| \delta Z \|  + \sqrt{n}\| E \|  \nonumber \\
&& \quad \le (2c_1+h_1 \upsilon + h_2\omega) \| A \| (h_1 \upsilon + h_2\omega)
+ \sqrt{n} \| A \| (2c_1+p(n)\upsilon)p(n)\upsilon + \sqrt{n}  \| A \| p(n)\upsilon  \nonumber \\
&& \quad \qquad \equiv  \tilde{\psi}_1 \| A \| \upsilon,
\end{eqnarray}
where $\tilde{\psi}_1=(2c_1+ h_1 \upsilon + h_2\omega)(h_1  + h_2\omega\upsilon^{-1})+\sqrt{n}p(n)(1+2c_1+p(n)\upsilon)$.
\end{proof}

The following theorem gives an error bound for one sweep of Jacobi's method in Algorithm \ref{alg:evd-main}. Here, $\fl(\cdot)$ is floating-point operation at precision $\omega$.
\begin{theorem} \label{thm:evd-main}
Let $T_k=(\hat{t}_{ij}^{(k)})$ be the matrix $T_0 =\fl(Q^T A Q)$ after $k$ Jacobi updates in Algorithm {\rm\ref{alg:evd-main}}, where the $k$th computed Jacobi rotation is $\hat{J}_k = J(p_k,q_k,q;\hat{c}_k,\hat{s}_k)$. Let $\gamma_n := (1-\omega)^{-n}-1$, $\tilde{\gamma}_j := (1-\omega)^{-wj}-1$ for a small integer constant $w>0$, and  $ \varphi_k= ((6+4\sqrt{2})\|G_k\|_F^2 + 2\|H_k\|_F^2)^{1/2}$, where $G_k$ and $H_k$ are defined as in Lemma {\rm \ref{lem:pqrk-err}} with $A_k=T_k$ and $(p,q)=(p_k,q_k)$ for all $k\ge 0$. Suppose $T_0$ has $n$ distinct eigenvalues with $d(T_0) > 0$.
If the precisions $\omega$ and $\upsilon$ satisfy $\omega \ll \upsilon$,
$4 \psi_1 \| A \| \upsilon < d(T_0)$ for some constant $\psi_1= \tilde{\psi}_1(n) + \upsilon^{-1}\gamma_n(2+\gamma_n)\||Q|^T|A||Q|\|_F/\|A\|$,  $d(T_0)\ge \underline{d} + \sum_{k=0}^{N-1} \varphi_k \tilde{\gamma}_4$ for some $\underline{d} >0$, and
\[
	\frac{1}{8}\varphi_{k}^2 \tilde{\gamma}_4^2 + \frac{1}{2}\|\off(T_{k})\|_F \varphi_{k}\tilde{\gamma}_4 + \|H_{k}\|_F^2\tilde{\gamma}_4 + \|G_{k}\|_F^2\tilde{\gamma}_4^2 \le |\hat{t}_{p_k q_k}^{(k)}|^2,
\]
for all $k = 0,1,\ldots,N-1$, then there exist two constants $\alpha, \beta >0 $ such that
$
\| \off(T_{N}) \|_F$ $\le \alpha \| \off(T_0) \|_F^2 + \beta \tilde{\gamma}_4,
$
where $N=n(n-1)/2$ and $\tilde{\psi}_1(n)$ is defined as Theorem {\rm\ref{thm:evd-ubd}}.
\end{theorem}

\begin{proof}
The computed symmetric matrix congruence $T_0 = \fl(Q^TAQ)$ satisfies
$
|T_0 - Q^TAQ| \le \gamma_n(2+\gamma_n)|Q|^T|A||Q|
$
\cite[Sect. 3]{Higham-2002},
and thus
$
\|T_0 - Q^TAQ\|_F \le \gamma_n(2+\gamma_n)\||Q|^T|A||Q|\|_F.
$
By Theorem \ref{thm:evd-ubd} we have
\[
\| \off(T_0)\|_F \le \|T_0 -Q^TAQ\|_F + \| \off(Q^T A Q)  \|_F  \le \psi_1 \| A \| \upsilon
\]
for some constant  $\psi_1 \equiv \tilde{\psi}_1(n) + \upsilon^{-1}\gamma_n(2+\gamma_n)\||Q|^T|A||Q|\|_F/\|A\|$. Therefore, we have
$4\| \off(T_0)\|_F< d(T_0)$ provided that $ 4\psi_1 \| A \|\upsilon < d(T_0)$.
The theorem is established by using Theorem \ref{thm:qc-rc}.
\end{proof}

In Theorem \ref{thm:evd-main}, it is worth pointing out that
\[
 (\|Q|^T|A||Q|\|_F)/\|A\| \le \|Q\|_F^2 \|A\|_F/ \|A\| \le n^{3/2}(1+c_3\omega).
\]
by using \eqref{eq:qrorth}--\eqref{qnm:ub}.
Hence, $\psi_1$ is not too large. We also note that the error bound of Algorithm {\rm\ref{alg:evd-main}} is established under the assumption that $ \psi_1 \| A \| \upsilon < d(T_0)/4$ for some constant $\psi_1=\psi_1(n)$ and there exists a clear gap between the eigenvalues of $T_0=\fl(Q^TAQ)$. In the following theorem, we establish the error analysis of Algorithm {\rm\ref{alg:evd-main}} under the assumption that there exists a clear gap between the eigenvalues of the original matrix $A$.

\begin{theorem}\label{thm:conda}
Let $T_k=(\hat{t}_{ij}^{(k)})$ be the matrix $T_0 =\fl(Q^T A Q)$ after $k$ Jacobi updates in Algorithm {\rm\ref{alg:evd-main}}, where the $k$th computed Jacobi rotation is $\hat{J}_k = J(p_k,q_k,q;\hat{c}_k,\hat{s}_k)$. Let $\gamma_n := (1-\omega)^{-n}-1$, $\tilde{\gamma}_j := (1-\omega)^{-wj}-1$ for a small integer constant $w>0$, and  $ \varphi_k= ((6+4\sqrt{2})\|G_k\|_F^2 + 2\|H_k\|_F^2)^{1/2}$, where $G_k$ and $H_k$ are defined as in Lemma {\rm \ref{lem:pqrk-err}} with $A_k=T_k$ and $(p,q)=(p_k,q_k)$ for all $k\ge 0$. Suppose $A$ has $n$ distinct eigenvalues with $d(A) > 0$.
If the precisions $\omega$ and $\upsilon$ satisfy $\omega \ll \upsilon$,
$4 \psi_1 \| A \| \upsilon < (1-\rho_1)d(A)$ for two constants $\psi_1= \tilde{\psi}_1(n) + \upsilon^{-1}\gamma_n(2+\gamma_n)\||Q|^T|A||Q|\|_F/\|A\|$ and $0<\rho_1<1$, $
2\psi_2  \|A\| \omega \le \rho_1d(A)
$
for some constant  $\psi_2=\psi_2(n)$,  $(1-\rho_1)d(A)\ge \underline{d} + \sum_{k=0}^{N-1} \varphi_k \tilde{\gamma}_4$ for some $\underline{d} >0$,
and
\[
	\frac{1}{8}\varphi_{k}^2 \tilde{\gamma}_4^2 + \frac{1}{2}\|\off(T_{k})\|_F \varphi_{k}\tilde{\gamma}_4 + \|H_{k}\|_F^2\tilde{\gamma}_4 + \|G_{k}\|_F^2\tilde{\gamma}_4^2 \le |\hat{t}_{p_k q_k}^{(k)}|^2,
\]
for all $k = 0,1,\ldots,N-1$, then there exist two constants $\alpha, \beta >0 $ such that
$
\| \off(T_{N}) \|_F  \le \alpha \| \off(T_0) \|_F^2 + \beta \tilde{\gamma}_4,
$
where $\tilde{\psi}_1(n)$ is defined as Theorem {\rm\ref{thm:evd-ubd}}
\end{theorem}
\begin{proof}
By Lemma \ref{lea:MGS}, there exists a matrix $\delta Q\in\Rnn$ such that  $Q + \delta Q$ is orthogonal with
\BE\label{dq:ub}
\| \delta Q \| \le \eta_3\kappa(Z)\omega\le  \sqrt{\frac{1 + \chi_1}{1 - \chi_1}} \cdot\eta_3\omega\equiv \eta_4\omega,
\EE
where the second inequality follows from \eqref{eq:kv}.
We note that
\[
(Q + \delta Q)^TA(Q + \delta Q)=Q^TAQ+Q^TA\delta Q+\delta Q^TAQ +\delta Q^TA\delta Q.
\]
By hypothesis, $T_0=\fl(Q^T A Q) = Q^T A Q + \Delta_1$ with $\|\Delta_1\| \le \||\Delta_1|\| \le \gamma_n(2+\gamma_n)\||Q|^T|A||Q|\|$.
It follows from Lemma \ref{lea:perturbation-eig},  \eqref{eq:qrorth} and \eqref{qnm:ub} that, for any $1\le k\le n$,
\begin{eqnarray*}
|\la_k(A)-\la_k(T_0)| &\le & 2\|A\|\|Q\|\|\delta Q\|+\|A\|\|\delta Q\|^2 + \gamma_n(2+\gamma_n)\||Q|^T|A||Q|\| \\
 &\le & \|A\| (2 \sqrt{1+\| F_2 \|}+\|\delta Q\|)\|\delta Q\| + \gamma_n(2+\gamma_n)\||Q|^T|A||Q|\| \\
&\le & \|A\|(2\sqrt{1+c_3\omega}+\eta_4\omega)\eta_4\omega + \gamma_n(2+\gamma_n)\||Q|^T|A||Q|\| \\
& \equiv & \psi_2 \|A\| \omega,
\end{eqnarray*}
where $\psi_2 = (2\sqrt{1+c_3\omega}+\eta_4\omega)\eta_4 + \omega^{-1}\gamma_n(2+\gamma_n)\||Q|^T|A||Q|\|/\|A\|$.

For any $i\neq j$, we have
\begin{eqnarray*}
|\la_i(A)-\la_j(A)| &\le& |\la_i(A)-\la_i(T_0)| + |\la_i(T_0)-\la_j(T_0)| +|\la_j(T_0)-\la_j(A)| \\
&\le& |\la_i(T_0)-\la_j(T_0)|  + 2\psi_2  \|A\| \omega,\\
|\la_i(A)-\la_j(A)| &\ge& | |\la_i(T^{(0)})-\la_j(T_0)| -|\la_i(A)-\la_i(T_0)| - |\la_j(T_0)-\la_j(A)| \\
&\ge&  |\la_i(T_0)-\la_j(T_0)| - 2\psi_2  \|A\| \omega.
\end{eqnarray*}
If $2\psi_2 \|A\| \omega \le \rho_1d(A)$ for some $0<\rho_1<1$, then we have
\[
 \frac{1}{1+\rho_1}  |\la_i(T_0)-\la_j(T_0)| \le |\la_i(A)-\la_j(A)| \le \frac{1}{1-\rho_1}  |\la_i(T_0)-\la_j(T_0)|.
\]
By hypothesis, $d(A)>0$. Thus,
$
d(T_0) \ge (1-\rho_1)d(A).
$
The theorem follows from Theorem  \ref{thm:evd-main} provided that $ \psi_1 \| A \|\upsilon < (1-\rho_1)d(A)/4$.
\end{proof}

In Theorem \ref{thm:conda}, we need an additional condition that  $2\psi_2  \|A\| \omega \le \rho_1d(A)$ for some   constant  $\psi_2=\psi_2(n)$ and $0<\rho_1<1$. In fact, it is much weaker than the condition that $ \psi_1 \| A \| \upsilon< (1-\rho_1)d(A)/4$ for some constant $\psi_1=\psi_1(n)$ since $\omega \ll \upsilon$. Therefore, the later is an essential condition for the error analysis of Algorithm {\rm\ref{alg:evd-main}} under the assumption that  there exists a clear gap between the eigenvalues of the original matrix $A$.

\section{A mixed precision preconditioned one-sided Jacobi method for the SVD} \label{Sec:SVD}
In this section, we propose a mixed precision preconditioned one-sided Jacobi method for computing the SVD of a real matrix. As we know, the one-sided Jacobi method for the singular value problem was originally mentioned in \cite{Hestenes-1958}.

We first  describe the one-sided Jacobi algorithm. Let $A$ be an  $m$-by-$n$ real matrix ($m \ge n$). The one-sided Jacobi algorithm aims to o construct a sequence of orthogonal updates $A^{(k+1)} = A^{(k)}J_k$ such that the columns of $A^{(k+1)}$ are mutually orthogonal sufficiently, where  $A^{(0)}=A$ and $J_k$ is a Jacobi rotation defined by \eqref{eq:JacMat}. Then, the computed SVD of $A$ is available by columns scaling of the updated $A^{(k+1)}$.

On  how to orthogonalize two columns of $A$, we have the following result \cite{deRijk-1989}.
\begin{lemma}\label{lem:ortho-2c}
Let $A$ be an $m\times n$ real matrix $(m \ge n)$. For any index pair $(i,j)$ with $1\le i<j\le n$, if $\ba_{i}^T\ba_{j}\neq 0$, then there exists a Jacobi rotation  $J=J(i,j;\theta)$
defined by \eqref{eq:JacMat} such that, for the updated matrix $B=AJ$,
\[
\bb_i^T\bb_j=0,\quad b_{ii}^2+b_{jj}^2=a_{ii}^2+a_{jj}^2+2a_{ij}^2,
\]
where  $c=\cos \theta =(1+t^2)^{-1/2}$ and $s=tc$ with $|\theta|\le \pi/4$. Here,
$t =1/(\mu + \sqrt{1+\mu^2})$ if $\mu\ge 0$ and $t=1/(\mu - \sqrt{1+\mu^2})$ if $\mu< 0$, where $\mu =
(\ba_j^T\ba_j - \ba_{i}^T\ba_i)/(2\ba_{i}^T\ba_{j})$.
\end{lemma}

In fact, the  Jacobi rotation  $J=J(i,j;\theta)$ defined by Lemma \ref{lem:ortho-2c} is such that the off-diagonal  entries $(i,j)$ and $(j,i)$ in the symmetric matrix $B^TB=J^TA^TAJ$ are zeros. A  row-cyclic one-sided  Jacobi algorithm  for the SVD is sated as Algorithm {\rm\ref{alg:rcJacobi-svd}}.
\begin{algorithm}[!htb]
	\caption{Cyclic one-sided  Jacobi's method for the SVD.} \label{alg:rcJacobi-svd}
   \begin{algorithmic}[1]
   \REQUIRE A matrix $A\in\Rmn$ ($m\ge n$) with $\rank(A)=n$ and a  tolerance $\epsilon>0$. Let $V=I_n$.
	\WHILE{$\| \off(A^TA) \|_F > \epsilon \| A^TA\|_F$}
     \FOR{$i=1,\ldots, n-1$}
     \FOR{$j=i+1,\ldots, n$}
	\STATE \qquad\qquad Compute a cosine-sine group $(c,s)$ as in Lemma  \ref{lem:ortho-2c}.
    \STATE \qquad\qquad Set $A=AJ(i,j,\theta)$ and $V=VJ(i,j,\theta)$.
    \ENDFOR
    \ENDFOR
	\ENDWHILE
	\end{algorithmic}
\end{algorithm}

For a reliable implementation of Algorithm {\rm\ref{alg:rcJacobi-svd}}, one may refer to \cite{Drmac-1997}.

As in Section \ref{Sec:Alg}, we propose a mixed precision preconditioned one-sided Jacobi algorithm for computing the SVD of a real matrix with  full column rank, which is stated as Algorithm {\rm\ref{alg:svd-main}}. Here, $\omega\le u\le \upsilon$.

\begin{algorithm}[!htb]
	\caption{A mixed precision preconditioned one-sided Jacobi method for the SVD.} \label{alg:svd-main}
	\begin{algorithmic}[1]
	\REQUIRE A matrix  $A\in\Rmn$ $(m\ge n)$ with $\rank(A)=n$ and a tolerance $\epsilon>0$.
	\STATE $[\sim,\sim,Z] = \text{svd}\left(A\right)$ \hfill $\triangleright$  SVD in precision $\upsilon$ and store $Z$ in precision $\omega$
	\STATE $Q = \text{MGS}\left(Z\right)$ \hfill $\triangleright$  Modified Gram-Schmidt orthogonalization in precision $\omega$
	\STATE Set $C = AQ$ and $V=Q$\hfill $\triangleright$ In precision $\omega$
    \WHILE{$\| \off(C^TC) \|_F > \epsilon \| C^TC\|_F$}
   \FOR{$i=1,\ldots, n-1$}
    \FOR{$j=i+1,\ldots, n$}
	\STATE \qquad\qquad Compute a cosine-sine group $(c,s)$ as in Lemma  \ref{lem:ortho-2c} with $A=C$. $\triangleright$ In precision $\omega$
    \STATE \qquad\qquad Set $C=CJ(i,j,\theta)$ and $V=VJ(i,j,\theta)$. $\triangleright$ In precision $\omega$
     \ENDFOR
     \ENDFOR	
     \ENDWHILE
     	\STATE Reorder the columns of $[C^T,V^T]^T$ with $\|\bc_1\|\ge\cdots\ge \|\bc_n\|> 0$.
	\STATE Compute $\sigma_j=\|\bc_j\|$ and $\bu_j=\bc_j/\sigma_j$ for $j=1,\ldots,n$. $\triangleright$ In precision $\omega$
	\STATE Set $U=[\bu_1,\ldots,\bu_n]$,$V=[\bv_1,\ldots,\bv_n]$, and  $\Sigma = \diag(\sigma_1,\ldots,\sigma_n)$. $\triangleright$ In precision $\omega$
	\end{algorithmic}
\end{algorithm}

\begin{remark}
In Algorithms {\rm \ref{alg:rcJacobi-svd}--\ref{alg:svd-main}}, the given matrix $A$ is assumed to be of full column rank. In fact, these algorithms can be used to compute the SVD of a general matrix \cite{Hestenes-1958}. When $m\gg n$, one may  use the rank-revealing QR factorization as a  preconditioner for these algorithms \cite{Drmac.Veselic-2008}.
\end{remark}

In the following, we give the error analysis of Algorithm  \ref{alg:svd-main}. We first give an estimate of the distance between $Z$ and $Q$ generated by \ref{alg:svd-main}.
\begin{lemma} \label{lea:svd-main}
Suppose that  $Z\in\Rnn$ in Algorithm {\rm\ref{alg:svd-main}} is computed by any svd solver in LAPACK, LINPACK or  EISPACK  in precision  $\upsilon$  and $Q\in\Rnn$  in Algorithm {\rm\ref{alg:svd-main}}  is computed by using the MGS method to $Z$ in precision $\omega$ $(\omega \ll \upsilon)$.
Then there exist constants  $f_i\equiv f_i(m,n)$ for $i=1,2$ such that
$
\| Z - Q \|_F \le f_1 \upsilon + f_2\omega.
$
\end{lemma}

\begin{proof}
The lemma follows from the arguments similar to that of Lemma \ref{lea:main}.
\end{proof}

On the orthogonalization of the columns of $AQ$ generated in Step 3 of Algorithm {\rm\ref{alg:svd-main}},  we have the following result.
\begin{theorem} \label{thm:svd-ubd}
Suppose that  $Z\in\Rnn$ in Algorithm {\rm\ref{alg:svd-main}} is computed by any svd solver in LAPACK, LINPACK or  EISPACK  in precision  $\upsilon$  and $Q\in\Rnn$  in Algorithm {\rm\ref{alg:svd-main}}  is computed by using the MGS method to $Z$ in precision $\omega$ $(\omega \ll \upsilon)$. Then there exists a constant  $\tilde{\zeta}_1 \equiv \tilde{\zeta}_1 (m,n)$ such that
$
\|\off((AQ)^T(AQ))\|_F  \le \tilde{\zeta}_1 \|A\|^2\upsilon.
$
\end{theorem}
\begin{proof}
By Lemma \ref{lea:svd}, there exists a matrix $E\in\Rmn$ such that $A+E$ admits the following exact SVD:
\BE \label{eq:svdfac}
A + E = (Y + \delta Y) S (Z + \delta Z)^T,\quad S=\diag(s_1,\ldots,s_n)\in\Rmn,
\EE
where $\| E\| \le p(m,n) \| A \| \upsilon$ and $Y + \delta Y$ and $Z + \delta Z$ are both orthogonal with $\| \delta Y \| \le p(m,n) \upsilon$ and $\| \delta Z \| \le p(m,n) \upsilon$. Using the orthogonality of  $Z + \delta Z$ we have
\BE \label{zsvd:ub}
\|Z\|\le \|Z + \delta Z\| + \|\delta Z\|=1+\|\delta Z\|\le 1+p(m,n)\upsilon\equiv \tau_1.
\EE
We note that $Q$ is computed by using the MGS method to $Z$ in  precision $\omega$. Then we have
\begin{eqnarray} \label{eq:svdexpand}
 Q^TA^TAQ -S^TS &=& (Q-Z+Z)^TA^TA(Q-Z+Z)-S^TS \nonumber \\
 &=& (Q-Z)^TA^TA(Q-Z)+(Q-Z)^TA^TAZ \nonumber \\
 && +Z^TA^TA(Q-Z)+(Z^TA^TAZ-S^TS).
\end{eqnarray}

On the other hand, it follows from  \eqref{eq:svdfac} that
$
S^TS =(Z+\delta Z)^T (A+E)^T (A+E) (Z+\delta Z),
$
i.e.,
\begin{eqnarray*}
S^TS - Z^TA^TAZ &=& Z^TA^TA \delta Z + \delta Z^TA^TA Z+ \delta Z^TA^TA\delta Z \\
&& + (Z+\delta Z)^T(A^TE+E^TA+E^TE)(Z+\delta Z).
\end{eqnarray*}
Thus,
\[
\| Z^TA^TAZ - S^TS \| \le 2 \|A\|^2 \|Z\| \|\delta Z\| + \|A\|^2\|\delta Z\|^2 + 2\|A\|\|E\|+\|E\|^2.
\]

By Lemma \ref{lea:svd-main} and using \eqref{zsvd:ub} and \eqref{eq:svdexpand} we obtain
\begin{eqnarray*}
&& \|\off(Q^TA^TAQ)\|_F  \le  \|Q^TA^TAQ - S^TS\|_F \\
&&\quad \le \|A\|^2\|Q-Z\|_F^2 + 2\|A\|^2\|Z\|\|Q-Z\|_F + \| Z^TA^TAZ - S^TS \|_F \\
&&\quad \le \|A\|^2\|Q-Z\|_F(\|Q-Z\|_F + 2\|Z\|) + \sqrt{n}\| Z^TA^TAZ - S^TS \|\\
&&\quad \le \|A\|^2\|Q-Z\|_F(\|Q-Z\|_F + 2\|Z\|) + \sqrt{n} \|A\|^2( 2\|Z\| + \|\delta Z\|)\|\delta Z\| \\
&& \qquad + \sqrt{n} (2\|A\|+\|E\|)\|E\|\\
&&\quad \le  \|A\|^2 (2\tau_1 + f_1\upsilon + f_2\omega)(f_1\upsilon + f_2 \omega)
+ \sqrt{n}\|A\|^2 (2\tau_1+2+ 2p(m,n)\upsilon)p(m,n)\upsilon \\
&&\quad \equiv \tilde{\zeta}_1 \|A\|^2\upsilon,
\end{eqnarray*}
where $\tilde{\zeta}_1 = (2\tau_1 + f_1\upsilon + f_2\omega)(f_1+f_2\omega\upsilon^{-1})
+2\sqrt{n}(\tau_1+1+p(m,n)\upsilon)p(m,n)$.
\end{proof}

In the following theorem, we give an error bound for one sweep of Algorithm {\rm\ref{alg:svd-main}}. Here, $\fl(\cdot)$ is floating-point operation at precision $\omega$.
\begin{theorem} \label{thm:svd-main}
Let $C_k$ be the matrix $C_0=\fl(A Q)$ after $k$ Jacobi updates in Algorithm {\rm\ref{alg:svd-main}}, where the $k$th computed Jacobi rotation is $\hat{J}_k = J(p_k,q_k,q;\hat{c}_k,\hat{s}_k)$.
Let $\gamma_n := (1-\omega)^{-n}-1$, $\tilde{\gamma}_j := (1-\omega)^{-wj}-1$ for a small integer constant $w>0$ and  $ \varphi_k= ((6+4\sqrt{2})\|G_k\|_F^2 + 2\|H_k\|_F^2)^{1/2}$, where $G_k$ and $H_k$ are defined as in Lemma {\rm \ref{lem:pqrk-err}} with $A_k=C_k^TC_k$ and $(p,q)=(p_k,q_k)$ for all $k\ge 0$.
Suppose $C_0$ has $n$ distinct singular values with $d(C_0^TC_0) =\min_{s\neq t}|\sigma_s^2(C_0) - \sigma_t^2(C_0) |> 0$.
If the precisions $\omega$ and $\upsilon$ satisfy $\omega \ll \upsilon$, 
$4 \zeta_1 \| A \|^2 \upsilon < d(C_0^TC_0)$
for some constant
\[
\zeta_1 = \tilde{\zeta}_1 + 2\sqrt{n}\gamma_n\upsilon^{-1} \|AQ\|\||A||Q|\|/\|A\|^2 + \sqrt{n}\gamma_n^2\upsilon^{-1}\||A||Q|\|/\|A||^2,
\]
$ d(C_0^TC_0)\ge \underline{d} + \sum_{k=0}^{N-1} \varphi_k \tilde{\gamma}_4$ for some $\underline{d} >0$, and
\[
	\frac{1}{8}\varphi_{k}^2 \tilde{\gamma}_4^2 + \frac{1}{2}\|\off(C_k^TC_k)\|_F \varphi_{k}\tilde{\gamma}_4 + \|H_{k}\|_F^2\tilde{\gamma}_4 + \|G_{k}\|_F^2\tilde{\gamma}_4^2 \le |(C_k^TC_k)_{p_k q_k}|^2,
\]
for all $k = 0,1,\ldots,N-1$, then there exist two constants $\alpha, \beta >0 $ such that
$
\| \off(C_N^TC_N) \|_F$  $\le \alpha \| \off(C_0^TC_0) \|_F^2 + \beta \tilde{\gamma}_4,
$
where $N=n(n-1)/2$ and  $\tilde{\zeta}_1$ is defined as Theorem {\rm\ref{thm:svd-ubd}}.
\end{theorem}

\begin{proof}
The computed matrix multiplication $C_0 =\fl(AQ)$ satisfies
$
|C_0 - AQ| = |\Delta_2| \le \gamma_n |A||Q|
$
 \cite[Sect. 3]{Higham-2002}.
Thus,
\begin{eqnarray*}
\|\off(C_0^TC_0)\|_F & = & \|\off((AQ+\Delta_2)^T(AQ+\Delta_2))\|_F \\
				 & \le & \|Q^TA^TAQ\|_F + 2 \|AQ\|\|\Delta_2\|_F + \|\Delta_2^T\Delta_2\|_F \\
				 & = & \|Q^TA^TAQ\|_F + 2\sqrt{n} \|AQ\|\|\Delta_2\| + \sqrt{n}\|\Delta_2\|^2 \\
				 & \le & \|Q^TA^TAQ\|_F + 2\sqrt{n}\gamma_n \|AQ\|\||A||Q|\| + \sqrt{n}\gamma_n^2\||A||Q|\|^2.
\end{eqnarray*}
By Theorem \ref{thm:svd-ubd} we have
$
\|\off(C_0^TC_0)\|_F \le \zeta_1\|A\|^2\upsilon
$
for some constant
\[
\zeta_1 \equiv \tilde{\zeta}_1 + 2\sqrt{n}\gamma_n\upsilon^{-1} \|AQ\|\||A||Q|\|/\|A\|^2 + \sqrt{n}\gamma_n^2\upsilon^{-1}\||A||Q|\|/\|A||^2.
\]
Therefore, we have
$ \|\off(C_0^TC_0)\|_F < d(C_0^TC_0)/4$ provided that
$
\zeta_1 \|A\|^2 \upsilon <d(C_0^TC_0)/4.
$
Then the theorem follows from Theorem \ref{thm:qc-rc}.
\end{proof}

In Theorem \ref{thm:svd-main}, the error analysis of Algorithm {\rm\ref{alg:svd-main}} is established under the assumption that  $\zeta_1 \|A\|^2 \upsilon <d(C_0^TC_0)/4$ for some constant $\zeta_1 \equiv \zeta_1 (m,n)$  and there exists a clear gap between the singular values  of $C_0=\fl(AQ)$, which is not easy to estimate in practice.
In the following theorem, we establish the error analysis of Algorithm {\rm\ref{alg:svd-main}} under the assumption that   there exists a clear gap between the singular values  of the original matrix $A$.

\begin{theorem}
Let $C_k$ be the matrix  $C_0=\fl(A Q)$ after $k$ Jacobi updates in Algorithm {\rm\ref{alg:svd-main}}, where the $k$th computed Jacobi rotation is $\hat{J}_k = J(p_k,q_k,q;\hat{c}_k,\hat{s}_k)$.
Let $\gamma_n := (1-\omega)^{-n}-1$, $\tilde{\gamma}_j := (1-\omega)^{-wj}-1$ for a small integer constant $w>0$ and  $ \varphi_k= ((6+4\sqrt{2})\|G_k\|_F^2 + 2\|H_k\|_F^2)^{1/2}$, where $G_k$ and $H_k$ are defined as in Lemma {\rm \ref{lem:pqrk-err}} with $A_k=C_k^TC_k$ and $(p,q)=(p_k,q_k)$ for all $k\ge 0$.
Suppose $A$ has $n$ distinct singular values with $d(A^TA)> 0$.
If the precisions $\omega$ and $\upsilon$ satisfy $\omega \ll \upsilon$ and
$4 \zeta_1 \| A \|^2 \upsilon < (1-\rho_2)d(A^TA)$ for two constants
\[
\zeta_1 = \tilde{\zeta}_1 + 2\sqrt{n}\gamma_n\upsilon^{-1} \|AQ\|\||A||Q|\|/\|A\|^2 + \sqrt{n}\gamma_n^2\upsilon^{-1}\||A||Q|\|/\|A||^2
\]
and $0<\rho_2<1$, $2\zeta_2  \|A\|^2 \omega \le \rho_2 d(A^TA)$ for some  constant  $\zeta_2=\zeta_2(m,n)$, 
$(1-\rho_2)d(A^TA)\ge \underline{d} + \sum_{k=0}^{N-1} \varphi_k \tilde{\gamma}_4$,
and
\[
	\frac{1}{8}\varphi_{k}^2 \tilde{\gamma}_4^2 + \frac{1}{2}\|\off(C_k^TC_k)\|_F \varphi_{k}\tilde{\gamma}_4 + \|H_{k}\|_F^2\tilde{\gamma}_4 + \|G_{k}\|_F^2\tilde{\gamma}_4^2 \le |(C_k^TC_k)_{p_k q_k}|^2,
\]
for all $k = 0,1,\ldots,N-1$, then there exist two constants $\alpha, \beta >0 $ such that
$
\| \off(C_N^TC_N) \|_F$  $\le \alpha \| \off(C_0^TC_0) \|_F^2 + \beta \tilde{\gamma}_4,
$
where  $\tilde{\zeta}_1$ is defined as Theorem {\rm\ref{thm:svd-ubd}}.
\end{theorem}

\begin{proof}
By Lemma \ref{lea:MGS}, there exists a matrix  $\delta Q\in\Rnn$ such that  $Q + \delta Q$ is orthogonal, where
$\| \delta Q \| \le \eta_4\omega$ with $\eta_4=\eta_4(n)$ being defined by \eqref{dq:ub}.
Then
\BE \label{q:ub}
\|Q\|\le \|Q + \delta Q\| + \|\delta Q\|=1+\|\delta Q\|\le 1+\eta_4\omega \equiv \eta_5.
\EE
We note that
\[
(Q + \delta Q)^TA^TA(Q + \delta Q)=Q^TA^TAQ+Q^TA^TA\delta Q+\delta Q^TA^TAQ +\delta Q^TA^TA\delta Q.
\]
By hypothesis, $C_0=\fl(A Q) = AQ + \Delta_2$ with $\|\Delta_2\| \le \||\Delta_2|\|\le \gamma_n\||A||Q|\|$.
From Lemma \ref{lea:perturbation-eig} and \eqref{q:ub} we obtain, for any $1\le k\le n$,
\begin{eqnarray*}
&&|\sigma_k^2(A)-\sigma_k^2(C_0)|=|\la_k(A^TA)-\la_k(C_0^TC_0)|  \\
&&\quad \le 2\|A\|^2\|Q\|\|\delta Q\|+\|A\|^2\|\delta Q\|^2  + 2\|\Delta_2\|\|AQ\| + \|\Delta_2\|^2\\
&&\quad \le \|A\|^2 (2\|Q\|+\|\delta Q\|)\|\delta Q\| + 2\gamma_n\||A||Q|\|\|AQ\| + \gamma_n^2\||A||Q|\|^2 \\
&&\quad \le \|A\|^2 (2\eta_5+\eta_4\omega)\eta_4\omega + 2\gamma_n\||A||Q|\|\|AQ\| + \gamma_n^2\||A||Q|\|^2 \equiv \zeta_2  \|A\|^2 \omega,
\end{eqnarray*}
where
$
\zeta_2 = (2\eta_5+\eta_4\omega)\eta_4 + 2\gamma_n \omega^{-1} (\||A||Q|\|\|AQ\|)/\|A\|^2 + \gamma_n^2\omega^{-1} (\||A||Q|\|^2)/\|A\|^2.$

For any $i\neq j$, we have
\begin{eqnarray*}
&& |\sigma_i^2(A)-\sigma_j^2(A)|=|\la_i(A^TA)-\la_j(A^TA)| \\
&&\quad \le |\la_i(A^TA)-\la_i(C_0^TC_0)| + |\la_i(C_0^TC_0)-\la_j(C_0^TC_0)| +|\la_j(C_0^TC_0)-\la_j(A^TA)| \\
&&\quad \le |\la_i(C_0^TC_0)-\la_j(C_0^TC_0)|  + 2\zeta_2  \|A\|^2 \omega
= |\sigma_i^2(C_0)-\sigma_j^2(C_0)|  + 2\zeta_2  \|A\|^2 \omega
\end{eqnarray*}
and
\begin{eqnarray*}
&&  |\sigma_i^2(A)-\sigma_j^2(A)|=|\la_i(A^TA)-\la_j(A^TA)| \\
&& \quad \ge | |\la_i(C_0^TC_0)-\la_j(C_0^TC_0)| -|\la_i(A^TA)-\la_i(C_0^TC_0)| - |\la_j(C_0^TC_0)-\la_j(A^TA)| \\
&& \quad \ge   |\la_i(C_0^TC_0)-\la_j(C_0^TC_0)| - 2\zeta_2  \|A\|^2 \omega
=|\sigma_i^2(C_0)-\sigma_j^2(C_0)|  - 2\zeta_2  \|A\|^2 \omega.
\end{eqnarray*}
If $2\zeta_2  \|A\|^2 \omega \le \rho_2d(A^TA)$ for some $0<\rho_2<1$, then we have
\[
 \frac{1}{1+\rho_2} |\sigma_i^2(C_0)-\sigma_j^2(C_0)|  \le |\sigma_i^2(A)-\sigma_j^2(A)| \le \frac{1}{1-\rho_2}  |\sigma_i^2(C_0)-\sigma_j^2(C_0)| .
\]
By hypothesis, $d(A^TA)>0$. Hence,
$
d(C_0^TC_0)  \ge (1-\rho_2) d(A^TA) >0.
$
The theorem follows from Theorem \ref{thm:qc-rc} if $\zeta_1 \|A\|^2 \upsilon < (1-\rho_2)d(A^TA)/4$.
\end{proof}

\section{Numerical Experiments} \label{Sec:Exp}
In this section, we present some numerical experiments to illustrate the effectiveness of Algorithms {\rm\ref{alg:evd-main}} and {\rm\ref{alg:svd-main}}  for computing the symmetric eigenvalue decomposition and the SVD. Numerical experiments were implemented in C++ and linked with Intel oneAPI Math Kernel Library (oneMKL) running on a workstation of CentOS equipped with an Intel(R) Xeon(R) Gold 6348 CPU at 2.60 GHz, 250GB of RAM and NVIDIA A30 Tensor Core GPU.
In our numerical tests, we set $\upsilon = 2^{-24}$ (single precision) and $\omega =u= 2^{-53}$ (double precision) for Algorithms {\rm\ref{alg:evd-main}} and {\rm\ref{alg:svd-main}}. 

In Algorithm {\rm\ref{alg:evd-main}}, we employ oneMKL routine `LAPACKE\_ssyev' to compute the eigenvalues and associated eigenvectors of a given matrix in precision $\upsilon$. In Algorithm {\rm\ref{alg:svd-main}}, we use `LAPACKE\_sgesvd' and the associated left and right singular vectors of a given matrix as an approximate SVD in precision $\upsilon$. Besides, various precisions were simulated by software {\it Advanpix} \cite{software-MCT} in {\tt MATLAB R2022a}.

In our numerical tests, ``{\tt Res.}",  ``{\tt OR-P.}",  ``{\tt CT.}",   ``{\tt JU.}",  and  ``{\tt SP.}" denote the computed relative residual $\|AP-P\diag(t_{11},\ldots,t_{nn})\|_F/\|A\|_F$ (or $\|AV-U\Sigma\|_F/\|A\|_F$),  the measure  of orthonormality $\|P^TP-I_n\|_F/\sqrt{n}$ of $P$, the running time in seconds, the Jacobi updates, and the number of sweeps at the final iterates of the corresponding algorithms, respectively. Also, for Algorithms {\rm\ref{alg:rcJacobi}} and {\rm\ref{alg:evd-main}}, we set $\epsilon=20.0\times \omega$. For Algorithms {\rm\ref{alg:rcJacobi-svd}}--\ref{alg:svd-main}, we set  $\epsilon = \omega$  and the algorithms are stopped if the ratio of Jacobi updates to $N := n(n-1)/2$ is less than $2\times 10^{-4}$.   For comparison, Algorithms \ref{alg:rcJacobi} and 
\ref{alg:rcJacobi-svd} are implemented at the machine precision $u$.



\subsection{The symmetric eigenvalue problem} In this subsection, we consider following two examples.

\begin{example} \label{ex1}
Let $A$ be an $n\times n$ random symmetric and positive definite matrix with pre-assigned singular value generated by  {\tt MATLAB 2022a}'s {\tt gallery ('randsvd', n, -kappa, mode)} with $\kappa(A)= {\tt kappa}$. We report our numerical results for {\rm(a)} ${\tt mode}=1$: the large eigenvalue is $1$ and the rest  of the eigenvalues are
$1/{\tt kappa}$, {\rm(b)} ${\tt mode}=2$: the small eigenvalue is $1/{\tt kappa}$ and the rest  of the eigenvalues are   $1$,  {\rm(c)} ${\tt mode}=3$: geometrically distributed eigenvalues, {\rm(d)} ${\tt mode}=4$: arithmetically distributed eigenvalues, and {\rm(e)} ${\tt mode}=5$: random eigenvalues with uniformly distributed logarithm.
\end{example}

\begin{example} \label{ex2}
We consider the case of multiple eigenvalues. Let $A=P_*\diag({\tt lam}\otimes\bfo_r) P_*^T$ be an $n\times n$ random symmetric and positive definite matrix with $n=rs$, where the orthogonal matrix $P_*\in\Rnn$ is randomly generated by the built-in functions {\tt randn} and {\tt orth} in {\tt MATLAB R2022a} and the vector ${\tt lam}\in\R^s$ of exact eigenvalues is generated as follows: {\rm(a)}  ${\tt mode}=1$:  {\tt lam=[1; ones(s-1,1)*1/kappa]}, {\rm(b)} ${\tt mode}=2$:   {\tt lam=[ones(s-1,1); 1/kappa]}, {\rm(c)} ${\tt mode}=3$:   {\tt lam=kappa.*linspace(-1,0,s)}, {\rm(d)} ${\tt mode}=4$:   {\tt lam=1-(1-1/kappa).*linspace(0,1,s)}, {\rm(e)} ${\tt mode}=5$:   {\tt lam=kappa.*} {\tt (-rand(s,1))}.
\end{example}

In Table \ref{tab:JacobiEVD}, we report the numerical results for Example \ref{ex1}. We observe from Table \ref{tab:JacobiEVD} that Algorithm {\rm\ref{alg:evd-main}} preserves the high accuracy of Algorithm {\rm\ref{alg:rcJacobi}} for all test matrices and even produces a sightly better orthonormality of $P$  than Algorithm {\rm\ref{alg:rcJacobi}} (especially for the cases that ${\tt mode}=3, 4, 5$). Moreover, Algorithm {\rm\ref{alg:evd-main}} works much better than  v for the cases that ${\tt mode}=3, 4, 5$ in terms of the computing time, the number of rotations, and sweeps. 

\begin{sidewaystable}
	\caption{Numerical results for Example \ref{ex1} with $n=512$.}\label{tab:JacobiEVD} \vskip 0.2mm
	\centering
	\small{
		\begin{tabular*}{\textwidth}{@{\extracolsep{\fill}}cccccccccccccc}
			
				&  \multicolumn{3}{c} {{\tt LAPACKE\_dgesvd}} & \multicolumn{5}{c}{ Alg. \ref{alg:rcJacobi} } & \multicolumn{5}{c}{Alg. \ref{alg:evd-main} } \\
					\cline{2-4} \cline{5-9} \cline{10-14}
			$({\tt mode},{\tt kappa})$ & {\tt Res.} & {\tt OR-P.} & {\tt CT.} & {\tt Res.}  &  {\tt OR-P.}  &  {\tt JU.}  &  {\tt CT.} & {\tt SP.} & {\tt Res.}  &  {\tt OR-P.}  & {\tt JU.} &  {\tt CT.} & {\tt SP.} \\ \midrule
			$(1,10^3)$ & 1.62e-15 & 3.11e-15 & 0.05 & 1.60e-15 & 8.11e-16 & 0.178N & 0.08 & 1 & 1.76e-15 & 1.13e-15 & 0.249N & 0.16 & 1 \\
			$(2,10^3)$ & 1.25e-15 & 3.33e-15 & 0.02 & 1.27e-15 & 1.06e-16 & 0.004N & $<$0.01 & 1 & 1.44e-15 & 7.53e-16 & 0.004N & 0.04 & 1 \\
			$(3,10^3)$ & 2.54e-15 & 4.78e-15 & 0.03 & 4.44e-15 & 4.63e-15 & 9.183N & 3.41 & 12 & 2.50e-15 & 2.31e-15 & 1.988N & 0.79 & 2 \\
			$(4,10^3)$ & 3.30e-15 & 5.46e-15 & 0.03 & 4.65e-15 & 4.46e-15 & 8.555N & 3.18 & 10 & 1.96e-15 & 2.30e-15 & 1.987N & 0.79 & 2 \\
			$(5,10^3)$ & 2.56e-15 & 5.00e-15 & 0.03 & 4.71e-15 & 4.65e-15 & 9.318N & 3.44 & 12 & 2.48e-15 & 2.28e-15 & 1.989N & 0.80 & 2 \\
			\hline
			$(1,10^4)$ & 1.70e-15 & 3.44e-15 & 0.03 & 1.61e-15 & 1.44e-15 & 0.729N & 0.28 & 1 & 1.83e-15 & 1.62e-15 & 0.815N & 0.36 & 1 \\
			$(2,10^4)$ & 1.25e-15 & 3.18e-15 & 0.02 & 1.26e-15 & 1.00e-16 & 0.004N & $<$0.01 & 1 & 1.42e-15 & 7.43e-16 & 0.004N & 0.04 & 1 \\
			$(3,10^4)$ & 2.38e-15 & 4.62e-15 & 0.03 & 4.68e-15 & 4.69e-15 & 9.366N & 3.46 & 13 & 2.76e-15 & 2.30e-15 & 2.009N & 0.81 & 3 \\
			$(4,10^4)$ & 3.43e-15 & 5.48e-15 & 0.03 & 4.93e-15 & 4.45e-15 & 8.506N & 3.17 & 10 & 2.33e-15 & 2.28e-15 & 1.980N & 0.79 & 2 \\
			$(5,10^4)$ & 2.34e-15 & 4.79e-15 & 0.03 & 4.66e-15 & 4.74e-15 & 9.553N & 3.56 & 13 & 2.64e-15 & 2.29e-15 & 1.990N & 0.81 & 2 \\
			\hline
			$(1,10^5)$ & 1.58e-15 & 3.78e-15 & 0.03 & 1.88e-15 & 1.65e-15 & 0.963N & 0.36 & 1 & 2.07e-15 & 1.73e-15 & 0.981N & 0.43 & 1 \\
			$(2,10^5)$ & 1.25e-15 & 3.29e-15 & 0.02 & 1.28e-15 & 9.90e-17 & 0.004N & $<$0.01 & 1 & 1.43e-15 & 7.46e-16 & 0.004N & 0.04 & 1 \\
			$(3,10^5)$ & 2.26e-15 & 4.73e-15 & 0.03 & 4.90e-15 & 4.79e-15 & 9.842N & 3.63 & 14 & 2.72e-15 & 2.31e-15 & 2.036N & 0.84 & 3 \\
			$(4,10^5)$ & 3.30e-15 & 5.49e-15 & 0.03 & 4.71e-15 & 4.49e-15 & 8.572N & 3.19 & 10 & 2.05e-15 & 2.28e-15 & 1.986N & 0.80 & 2 \\
			$(5,10^5)$ & 2.12e-15 & 4.45e-15 & 0.03 & 4.90e-15 & 4.76e-15 & 9.748N & 3.62 & 15 & 2.77e-15 & 2.30e-15 & 2.020N & 0.82 & 3 \\
			\hline
			$(1,10^6)$ & 1.52e-15 & 4.11e-15 & 0.03 & 1.59e-15 & 1.62e-15 & 0.991N & 0.37 & 1 & 1.89e-15 & 1.74e-15 & 0.998N & 0.43 & 1 \\
			$(2,10^6)$ & 1.23e-15 & 3.22e-15 & 0.02 & 1.27e-15 & 9.88e-17 & 0.004N & $<$0.01 & 1 & 1.43e-15 & 7.58e-16 & 0.004N & 0.04 & 1 \\
			$(3,10^6)$ & 2.20e-15 & 4.78e-15 & 0.03 & 4.85e-15 & 4.86e-15 & 10.043N & 3.74 & 16 & 2.65e-15 & 2.34e-15 & 2.088N & 0.85 & 3 \\
			$(4,10^6)$ & 3.34e-15 & 5.40e-15 & 0.03 & 4.62e-15 & 4.46e-15 & 8.507N & 3.19 & 10 & 2.04e-15 & 2.28e-15 & 1.986N & 0.80 & 2 \\
			$(5,10^6)$ & 2.09e-15 & 4.26e-15 & 0.03 & 4.90e-15 & 4.86e-15 & 10.211N & 3.77 & 16 & 2.63e-15 & 2.34e-15 & 2.104N & 0.85 & 4 \\
	\end{tabular*}
	}
\end{sidewaystable}


To  show the effectiveness of the initial guess $Q$ generated by Algorithm \ref{alg:evd-main}, in Figure \ref{fig:modeL4R5}, we plot the quantities $\|\off(A)\|_F$, $\|\off(T_0)\|_F$, ${\tt bd}/10$, ${\tt bd}$, $10{\tt bd}$, and $d(A)/4$ versus the dimension $n$ for Example \ref{ex1} with ${\tt mode}=4$ (left) and ${\tt mode}=5$ (right), respectively.  Here, $T_0=\fl(Q^TAQ)$,
$d(A)=\min_{\la_i(A) \ne \la_j(A)} | \la_i(A) - \la_j(A) |$ and ${\tt bd}=n\| A \| \upsilon$ is an approximation estimate of the theoretical bound $\psi_1 \| A \| \upsilon$, which is obtained in Theorem \ref{thm:conda}.  We also see from Table \ref{tab:JacobiEVD}  and Figure \ref{fig:modeL4R5} that   Algorithm {\rm\ref{alg:evd-main}} is much efficient  over Algorithm {\rm\ref{alg:rcJacobi}}, where the preprocessed matrix $Q$ is such that $\|\off(T_0)\|_F$ is much less than $\|\off(A)\|_F$, despite the quantity $\psi_1\|A\|\upsilon$ is not necessarily less than $d(A)/4$.
Moreover, in Table \ref{tab:JacobiEVD-comp}, we compare the performance of Algorithm \ref{alg:evd-main}, where the starting guess $Q$ was computed in two ways: both eig and MGS in double precision and eig in single precision and MGS in double precision. Here, ``{\tt init-Res.}",  ``{\tt init-OR-Q.}" mean the computed relative residual $\|AZ-Z\Phi\|_F/\|A\|_F$ and the measure  of orthonormality $\|Q^TQ-I_n\|_F/\sqrt{n}$ of $Q$, respectively. We see from Tables \ref{tab:JacobiEVD}--\ref{tab:JacobiEVD-comp} that the starting guess via the mixed precision may accelerate the Jacobi iteration. Even  both eig and MGS are implemented in double precision, the Jacobi iteration can further improve the accuracy. 

To investigate the sensitivity of the lower precision in the proposed algorithm, we utilize {\it Advanpix} to simulate the performance of  Algorithm \ref{alg:evd-main} for Example \ref{ex2} with different precisions. Figures \ref{fig:qc-mode35}--\ref{fig:qc-mode4} describe $\| \off(T_{jN}) \|_F$ versus $j$ (the number of sweeps) for various choices of $\upsilon$.  We observe from Figures \ref{fig:qc-mode35}--\ref{fig:qc-mode4} that Algorithm \ref{alg:evd-main} is much more efficient than  Algorithm {\rm\ref{alg:rcJacobi}} for  different precisions, especially for the original matrices with uniform distribution of eigenvalues and large condition numbers.

\begin{figure}[!htb]
\centering
\includegraphics[width=\textwidth,height=0.25\textheight]{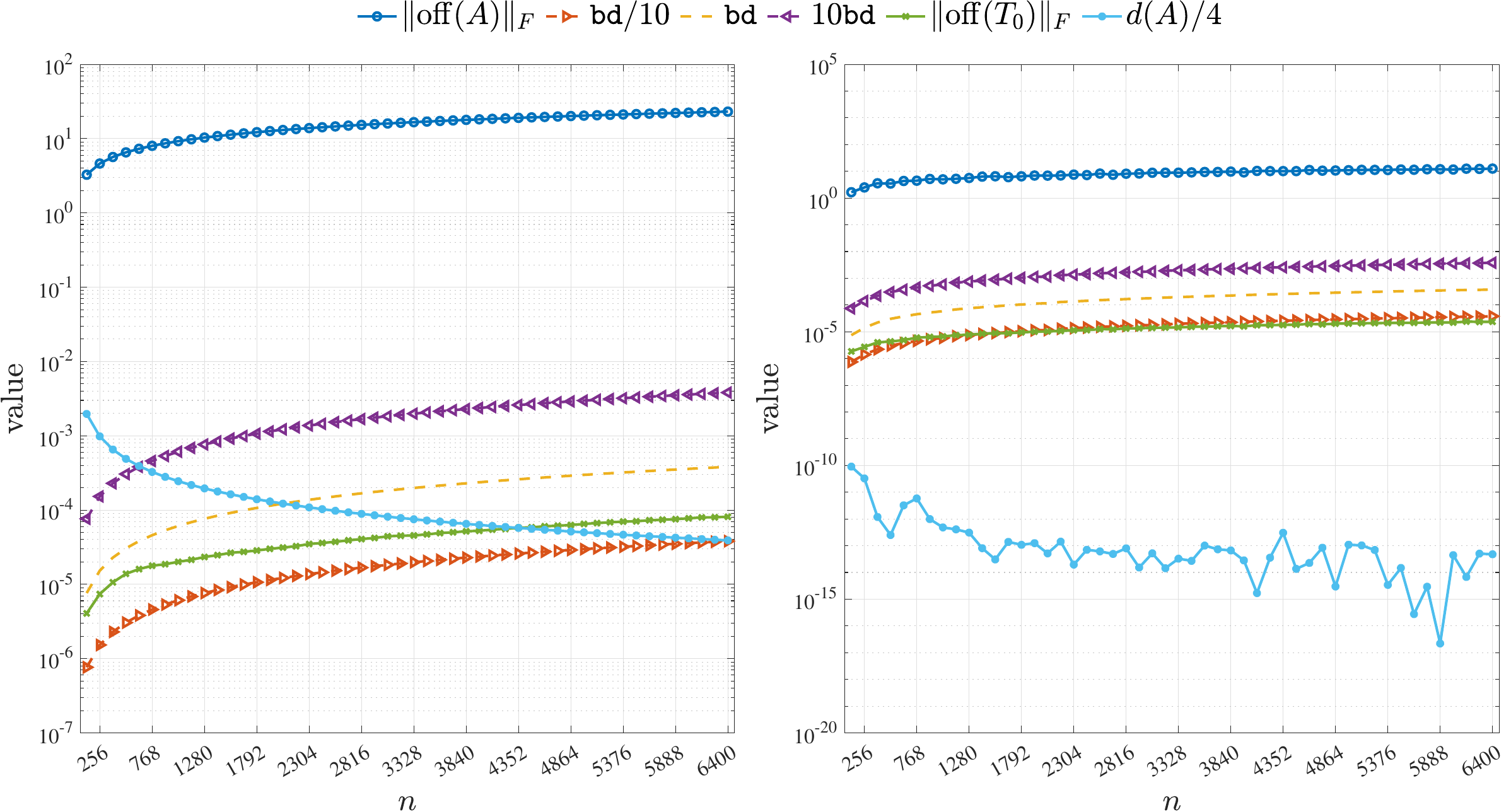}
\caption{Different quantities versus $n$ for Example \ref{ex1}.} \label{fig:modeL4R5}
\end{figure}

\begin{sidewaystable}
	\caption{Numerical results for Example \ref{ex1} with $n=2048$.}\label{tab:JacobiEVD-comp} 
	\centering
	\scriptsize{
		\begin{tabular*}{\textwidth}{@{\extracolsep{\fill}}ccccccccccccccc}
  
          &  \multicolumn{7}{c}{double-double solver} & \multicolumn{7}{c}{single-double solver} \\
            \cline{2-8} \cline{9-15}
			$({\tt mode},{\tt kappa})$ & {\tt init-Res.} & {\tt init-OR-Q.} & {\tt Res.} & {\tt OR-P.} & {\tt JU.} & {\tt CT.} & {\tt SP.} & {\tt init-Res.} & {\tt init-OR-Q.} & {\tt Res.} & {\tt OR-P.} & {\tt JU.} & {\tt CT.} & {\tt SP.}  \\ \midrule
	   $(1,10^3)$ & 2.58e-15 & 1.09e-15 & 1.58e-15 & 1.14e-15 & 0.013N & 2.94 & 1 & 1.11e-06 & 1.39e-15 & 2.99e-15 & 1.48e-15 & 0.025N & 3.36 & 1 \\
            $(2,10^3)$ & 1.79e-15 & 1.19e-15 & 1.42e-15 & 1.19e-15 & 0.001N & 2.44 & 1 & 1.46e-06 & 1.27e-15 & 2.49e-15 & 1.28e-15 & 0.001N & 2.11 & 1 \\
            $(3,10^3)$ & 4.51e-15 & 1.23e-15 & 3.70e-15 & 2.00e-15 & 0.414N & 20.88 & 1 & 2.73e-06 & 1.42e-15 & 4.14e-15 & 4.49e-15 & 2.002N & 76.21 & 3 \\
            $(4,10^3)$ & 6.17e-15 & 1.28e-15 & 4.03e-15 & 1.47e-15 & 0.075N & 6.44 & 1 & 3.18e-06 & 1.42e-15 & 3.24e-15 & 4.49e-15 & 1.994N & 76.14 & 2 \\
            $(5,10^3)$ & 4.71e-15 & 1.24e-15 & 3.82e-15 & 2.01e-15 & 0.421N & 21.40 & 1 & 2.60e-06 & 1.43e-15 & 4.17e-15 & 4.50e-15 & 2.002N & 77.63 & 3 \\
            $(1,10^4)$ & 3.13e-15 & 1.07e-15 & 2.52e-15 & 2.39e-15 & 0.491N & 22.41 & 1 & 1.40e-06 & 1.37e-15 & 3.34e-15 & 2.71e-15 & 0.587N & 27.52 & 1 \\
            $(2,10^4)$ & 1.78e-15 & 1.19e-15 & 1.41e-15 & 1.19e-15 & 0.001N & 2.89 & 1 & 1.45e-06 & 1.28e-15 & 2.50e-15 & 1.28e-15 & 0.001N & 2.75 & 1 \\
            $(3,10^4)$ & 4.48e-15 & 1.22e-15 & 3.47e-15 & 2.32e-15 & 0.640N & 32.79 & 1 & 2.56e-06 & 1.42e-15 & 4.33e-15 & 4.52e-15 & 2.035N & 86.76 & 3 \\
            $(4,10^4)$ & 5.49e-15 & 1.32e-15 & 5.07e-15 & 1.51e-15 & 0.075N & 6.67 & 1 & 2.89e-06 & 1.42e-15 & 2.98e-15 & 4.49e-15 & 1.996N & 84.18 & 2 \\
            $(5,10^4)$ & 4.26e-15 & 1.22e-15 & 3.51e-15 & 2.33e-15 & 0.638N & 30.82 & 1 & 2.56e-06 & 1.43e-15 & 4.31e-15 & 4.51e-15 & 2.023N & 87.41 & 3 \\
            $(1,10^5)$ & 3.60e-15 & 1.08e-15 & 2.78e-15 & 3.12e-15 & 0.919N & 37.86 & 1 & 1.50e-06 & 1.37e-15 & 3.44e-15 & 3.25e-15 & 0.937N & 47.80 & 1 \\
            $(2,10^5)$ & 1.76e-15 & 1.19e-15 & 1.41e-15 & 1.19e-15 & 0.001N & 2.67 & 1 & 1.39e-06 & 1.27e-15 & 2.50e-15 & 1.26e-15 & 0.001N & 2.35 & 1 \\
            $(3,10^5)$ & 4.63e-15 & 1.21e-15 & 3.29e-15 & 2.51e-15 & 0.761N & 40.33 & 1 & 2.36e-06 & 1.42e-15 & 4.45e-15 & 4.59e-15 & 2.097N & 91.02 & 4 \\
            $(4,10^5)$ & 5.50e-15 & 1.32e-15 & 5.09e-15 & 1.51e-15 & 0.075N & 6.86 & 1 & 2.86e-06 & 1.42e-15 & 3.02e-15 & 4.49e-15 & 1.995N & 76.60 & 2 \\
            $(5,10^5)$ & 3.95e-15 & 1.21e-15 & 3.34e-15 & 2.51e-15 & 0.762N & 35.58 & 1 & 2.08e-06 & 1.42e-15 & 4.43e-15 & 4.59e-15 & 2.099N & 89.77 & 4 \\
            $(1,10^6)$ & 3.47e-15 & 1.08e-15 & 2.87e-15 & 3.22e-15 & 0.992N & 40.33 & 1 & 2.07e-06 & 1.36e-15 & 3.50e-15 & 3.33e-15 & 0.993N & 41.66 & 1 \\
            $(2,10^6)$ & 1.78e-15 & 1.18e-15 & 1.39e-15 & 1.18e-15 & 0.001N & 2.60 & 1 & 1.68e-06 & 1.28e-15 & 2.50e-15 & 1.28e-15 & 0.001N & 2.29 & 1 \\
            $(3,10^6)$ & 4.45e-15 & 1.20e-15 & 3.17e-15 & 2.61e-15 & 0.827N & 42.89 & 1 & 2.07e-06 & 1.42e-15 & 4.54e-15 & 4.65e-15 & 2.179N & 85.75 & 4 \\
            $(4,10^6)$ & 5.47e-15 & 1.32e-15 & 5.05e-15 & 1.51e-15 & 0.076N & 6.64 & 1 & 2.88e-06 & 1.42e-15 & 3.00e-15 & 4.48e-15 & 1.995N & 89.05 & 2 \\
            $(5,10^6)$ & 3.67e-15 & 1.20e-15 & 3.23e-15 & 2.64e-15 & 0.840N & 36.07 & 1 & 1.80e-06 & 1.42e-15 & 4.55e-15 & 4.70e-15 & 2.228N & 87.31 & 4 \\
            \cline{1-15}
		\end{tabular*}
	}
\end{sidewaystable}

\begin{figure}[!htb]
\centering
\includegraphics[width=\textwidth,height=0.25\textheight]{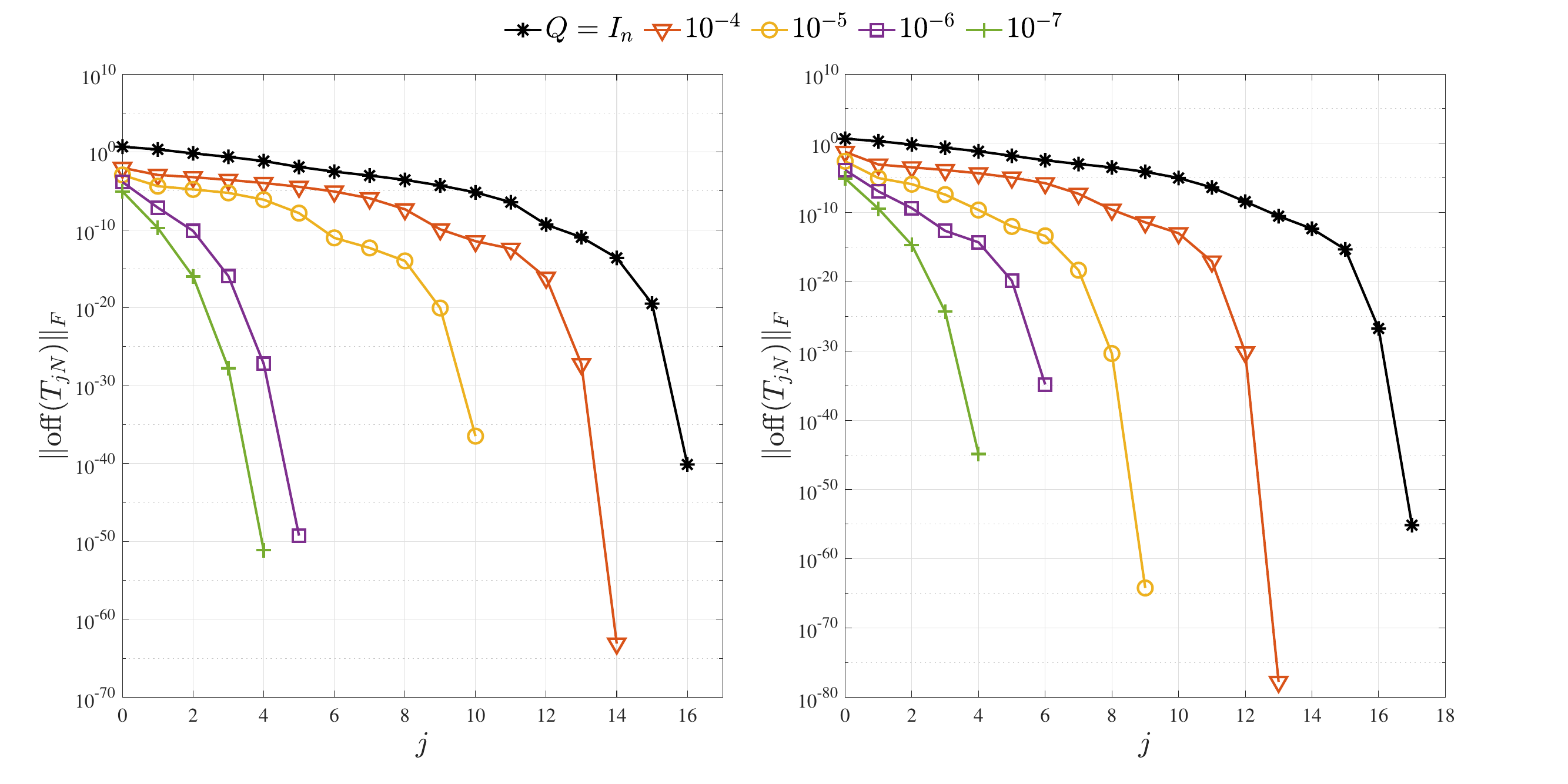}
\caption{ $\| \off(T_{jN}) \|_F$ versus $j$ (sweeps)  for Example \ref{ex2} with $(r,s)=(2,256)$. Left: $(\tt{mode},\tt{kappa}) =(3,10^4)$ and right: $(\tt{mode}, \tt{kappa}) = (5, 10^4)$.} \label{fig:qc-mode35}
\end{figure}

\begin{figure}[!htb]
\centering
\includegraphics[width=\textwidth,height=0.25\textheight]{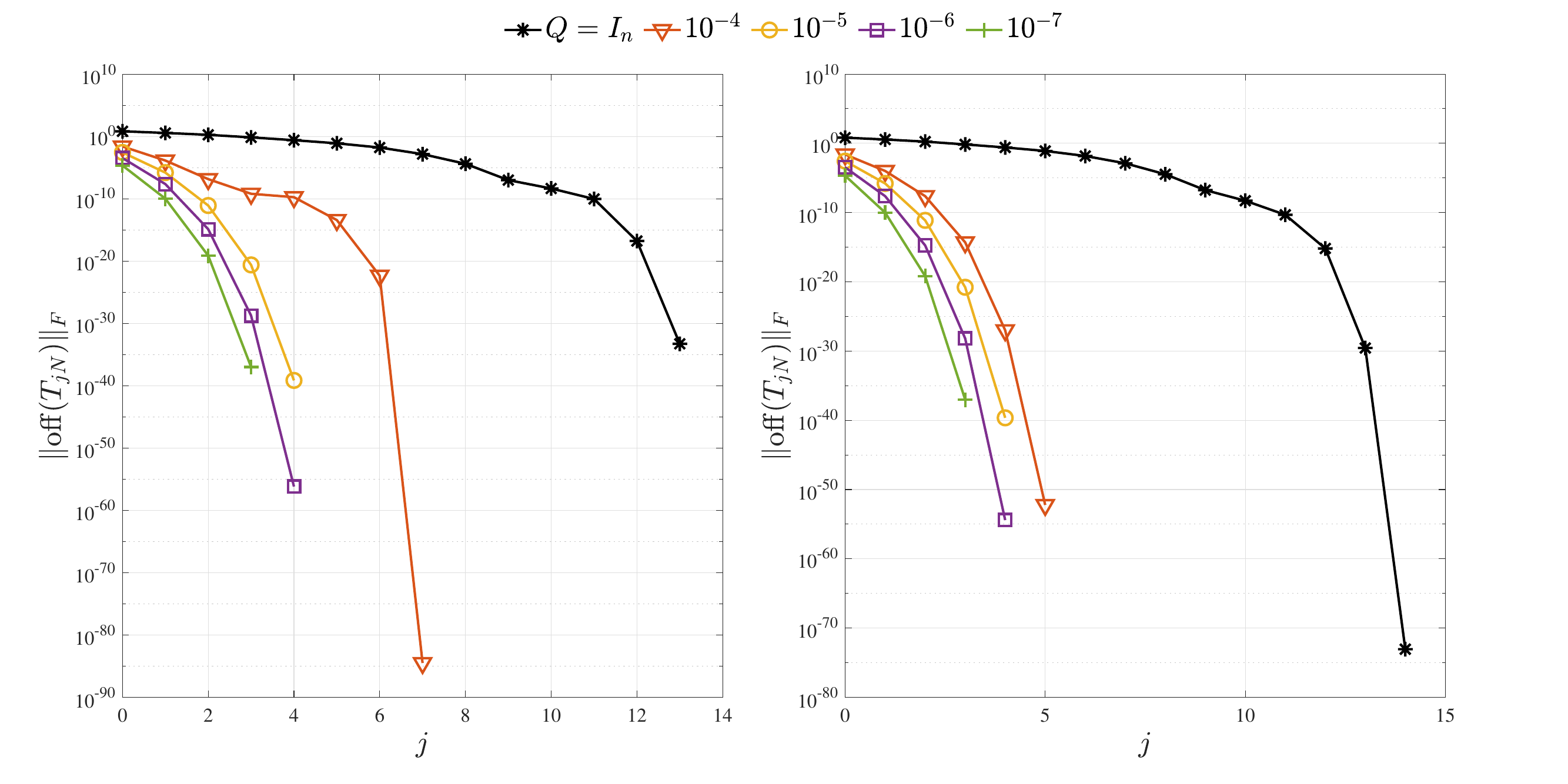}
\caption{$\| \off(T_{jN}) \|_F$ versus $j$ (sweeps)  for Example \ref{ex2} with $(r,s)=(4,128)$. Left: $(\tt{mode},\tt{kappa}) =(4,10^8)$ and right: $(\tt{mode}, \tt{kappa}) = (4, 10^{12})$.} \label{fig:qc-mode4}
\end{figure}


To further illustrate the effectiveness of Algorithm \ref{alg:evd-main}, we also implement  Algorithm  \ref{alg:rcJacobi} and  Algorithm \ref{alg:evd-main} in parallel. Here, we utilize NVIDIA cuSOLVER library \footnote{https://docs.nvidia.com/cuda/cusolver} and NVIDIA cuBLAS library\footnote{https://docs.nvidia.com/cuda/cublas} and the parallel ordering of the rotation set $\{(i,j)\;|\; 1\le i<j\le n\}$ can be taken as some non-overlapping order, e.g., the merry-go-round ordering $(i,j) = (1,2), (3,4),\ldots,(n-1,n),(1,4),\ldots,(n-3,n-1),(1,6),\ldots,(n,n-2)$.
The pre-processing stage in Steps 1--3 of Algorithm \ref{alg:evd-main} was implemented on the NVIDIA CUDA routine, where the function `{\tt cusolverDnSsyevd}` was employed as the eigensolver in precision $\upsilon$ and the MGS method  was replaced by the Householder QR factorization, which can theoretically guarantee higher orthogonality \cite[\S 5.2]{Golub.VanLoan-2013}.

The numerical results for Example \ref{ex1} with different $n$ are displayed in Tables \ref{tab:JacobiEVDcuda-GPU2048}--\ref{tab:JacobiEVDcuda-GPU4096}. Here,   ``{\tt CT-Pre.}"   means  the running time for the pre-processing  stage in Steps 1--3 of (parallel) Algorithm \ref{alg:evd-main} and ``{\tt CT-J.}" means the running time for the stage of  the Jacobi procedure of (parallel) Algorithm \ref{alg:evd-main} or (parallel)  Algorithm  \ref{alg:rcJacobi}. We see from Tables \ref{tab:JacobiEVDcuda-GPU2048}--\ref{tab:JacobiEVDcuda-GPU4096} that, as  parallel Algorithm  \ref{alg:rcJacobi},  parallel Algorithm \ref{alg:evd-main}  can significantly  improve the efficiency of Algorithm \ref{alg:evd-main}. As expected,  Algorithm \ref{alg:evd-main} (parallel version, respectively) works much better than  Algorithm  \ref{alg:rcJacobi} (parallel version, respectively) in terms of the total computing time.


\begin{table}[!htb]
	\caption{Numerical results for Example \ref{ex1} with $n=2048$.} \label{tab:JacobiEVDcuda-GPU2048} \vskip 0.2mm
	\centering
	\scriptsize{
		\begin{tabular*}{\textwidth}{@{\extracolsep{\fill}}cccccccccc}
			& \multicolumn{4}{c}{parallel Alg. \ref{alg:rcJacobi} } & \multicolumn{4}{c}{parallel Alg. \ref{alg:evd-main} } \\
			\cline{2-5} \cline{6-10}
			$({\tt mode},{\tt kappa})$ & {\tt Res.} & {\tt OR-P.} & {\tt CT-J.} & {\tt SP.} & {\tt Res.} & {\tt OR-P.} & {\tt CT-Pre.} & {\tt CT-J.} & {\tt SP.} \\ \midrule
			$(3,10^3)$ & 1.61e-12 & 1.52e-12 & 0.88 & 15 & 2.20e-13 & 2.39e-13 & 0.14 & 0.12 & 3 \\
			$(4,10^3)$ & 1.36e-12 & 1.37e-12 & 0.66 & 12 & 2.24e-13 & 2.24e-13 & 0.14 & 0.12 & 2 \\
			$(5,10^3)$ & 1.65e-12 & 1.58e-12 & 0.85 & 15 & 2.25e-13 & 2.51e-13 & 0.14 & 0.13 & 3 \\
			\hline
			$(3,10^4)$ & 1.75e-12 & 1.63e-12 & 0.94 & 17 & 2.29e-13 & 2.88e-13 & 0.14 & 0.14 & 3 \\
			$(4,10^4)$ & 1.37e-12 & 1.37e-12 & 0.66 & 12 & 2.24e-13 & 2.24e-13 & 0.14 & 0.12 & 2 \\
			$(5,10^4)$ & 1.76e-12 & 1.67e-12 & 0.95 & 17 & 2.33e-13 & 2.96e-13 & 0.14 & 0.15 & 4 \\
			\hline
			$(3,10^5)$ & 1.83e-12 & 1.72e-12 & 1.03 & 18 & 2.45e-13 & 3.70e-13 & 0.14 & 0.17 & 4 \\
			$(4,10^5)$ & 1.38e-12 & 1.38e-12 & 0.66 & 12 & 2.24e-13 & 2.24e-13 & 0.14 & 0.12 & 2 \\
			$(5,10^5)$ & 1.80e-12 & 1.72e-12 & 1.04 & 19 & 2.50e-13 & 3.87e-13 & 0.14 & 0.19 & 5 \\
			\hline
			$(3,10^6)$ & 1.97e-12 & 1.88e-12 & 1.16 & 21 & 2.63e-13 & 4.70e-13 & 0.14 & 0.21 & 5 \\
			$(4,10^6)$ & 1.37e-12 & 1.38e-12 & 0.67 & 12 & 2.24e-13 & 2.24e-13 & 0.14 & 0.12 & 2 \\
			$(5,10^6)$ & 2.04e-12 & 1.94e-12 & 1.17 & 21 & 2.63e-13 & 4.79e-13 & 0.14 & 0.24 & 6 \\
			\bottomrule
	\end{tabular*}
	\begin{tabular*}{\textwidth}{@{\extracolsep{\fill}}cccccccccc}
			 & \multicolumn{4}{c}{Alg. \ref{alg:rcJacobi} } & \multicolumn{4}{c}{Alg. \ref{alg:evd-main} } \\
			\cline{2-5} \cline{6-10}
			$({\tt mode},{\tt kappa})$ & {\tt Res.} & {\tt OR-P.} & {\tt CT-J.} & {\tt SP.}  &  {\tt Res.} & {\tt OR-P.} & {\tt CT-Pre.} & {\tt CT-J.} & {\tt SP.} \\ \midrule
			$(3,10^3)$ & 1.04e-14 & 9.65e-15 & 379.91 & 14 & 4.50e-15 & 4.50e-15 & 2.22 & 76.42 & 3 \\
			$(4,10^3)$ & 1.08e-14 & 9.34e-15 & 358.10 & 11 & 3.70e-15 & 4.48e-15 & 2.17 & 76.09 & 2 \\
			$(5,10^3)$ & 1.03e-14 & 9.75e-15 & 388.71 & 15 & 4.39e-15 & 4.50e-15 & 2.36 & 76.22 & 3 \\
			\hline
			$(3,10^4)$ & 1.03e-14 & 9.81e-15 & 391.40 & 15 & 4.55e-15 & 4.52e-15 & 2.52 & 76.78 & 3 \\
			$(4,10^4)$ & 1.07e-14 & 9.33e-15 & 355.92 & 12 & 3.76e-15 & 4.49e-15 & 2.29 & 74.76 & 2 \\
			$(5,10^4)$ & 1.02e-14 & 9.86e-15 & 390.06 & 15 & 4.61e-15 & 4.52e-15 & 2.43 & 76.44 & 3 \\
			\hline
			$(3,10^5)$ & 9.79e-15 & 9.90e-15 & 394.39 & 16 & 4.67e-15 & 4.58e-15 & 2.43 & 78.35 & 4 \\
			$(4,10^5)$ & 1.04e-14 & 9.28e-15 & 348.81 & 11 & 3.72e-15 & 4.49e-15 & 2.26 & 74.79 & 2 \\
			$(5,10^5)$ & 1.00e-14 & 9.97e-15 & 402.72 & 17 & 4.60e-15 & 4.61e-15 & 2.30 & 79.21 & 4 \\
			\hline
			$(3,10^6)$ & 9.66e-15 & 1.01e-14 & 409.20 & 17 & 4.78e-15 & 4.66e-15 & 2.40 & 81.36 & 3 \\
			$(4,10^6)$ & 1.05e-14 & 9.30e-15 & 351.52 & 11 & 3.40e-15 & 4.49e-15 & 2.37 & 75.61 & 2 \\
			$(5,10^6)$ & 1.00e-14 & 1.01e-14 & 417.28 & 18 & 4.84e-15 & 4.66e-15 & 2.20 & 82.55 & 4 \\
			\bottomrule
	\end{tabular*}
	}
\end{table}

\begin{table}[!htb]
	\caption{Numerical results for Example \ref{ex1} with $n=4096$.} \label{tab:JacobiEVDcuda-GPU4096} \vskip 0.2mm
	\centering
	\scriptsize{
		\begin{tabular*}{\textwidth}{@{\extracolsep{\fill}}cccccccccc}
			& \multicolumn{4}{c}{parallel Alg. \ref{alg:rcJacobi} } & \multicolumn{4}{c}{parallel Alg. \ref{alg:evd-main} } \\
			\cline{2-5} \cline{6-10}
			$({\tt mode},{\tt kappa})$ & {\tt Res.} & {\tt OR-P.} & {\tt CT-J.} & {\tt SP.} & {\tt Res.} & {\tt OR-P.} & {\tt CT-Pre.} & {\tt CT-J.} & {\tt SP.} \\ \midrule
			$(3,10^3)$ & 5.74e-12 & 5.03e-12 & 4.86 & 7 & 4.59e-13 & 5.29e-13 & 0.76 & 0.70 & 1 \\
			$(4,10^3)$ & 4.44e-12 & 4.43e-12 & 4.02 & 6 & 4.60e-13 & 4.61e-13 & 0.75 & 0.69 & 1 \\
			$(5,10^3)$ & 5.41e-12 & 5.08e-12 & 5.10 & 7 & 4.70e-13 & 5.47e-13 & 0.76 & 0.80 & 2 \\
			\hline
			$(3,10^4)$ & 5.77e-12 & 5.65e-12 & 5.36 & 7 & 4.87e-13 & 6.25e-13 & 0.77 & 0.77 & 2 \\
			$(4,10^4)$ & 4.32e-12 & 4.32e-12 & 4.00 & 6 & 4.60e-13 & 4.61e-13 & 0.75 & 0.69 & 1 \\
			$(5,10^4)$ & 6.62e-12 & 5.66e-12 & 5.39 & 7 & 4.97e-13 & 6.33e-13 & 0.77 & 0.91 & 2 \\
			\hline
			$(3,10^5)$ & 6.49e-12 & 5.99e-12 & 5.86 & 8 & 5.18e-13 & 7.90e-13 & 0.78 & 0.86 & 2 \\
			$(4,10^5)$ & 4.35e-12 & 4.35e-12 & 4.00 & 6 & 4.60e-13 & 4.61e-13 & 0.76 & 0.69 & 1 \\
			$(5,10^5)$ & 5.39e-12 & 5.92e-12 & 6.16 & 8 & 5.24e-13 & 7.98e-13 & 0.77 & 0.97 & 2 \\
			\hline
			$(3,10^6)$ & 6.62e-12 & 6.45e-12 & 6.35 & 9 & 5.45e-13 & 1.00e-12 & 0.78 & 1.06 & 2 \\
			$(4,10^6)$ & 4.40e-12 & 4.40e-12 & 4.06 & 6 & 4.60e-13 & 4.61e-13 & 0.76 & 0.70 & 1 \\
			$(5,10^6)$ & 7.13e-12 & 6.55e-12 & 6.56 & 9 & 5.51e-13 & 1.04e-12 & 0.77 & 1.15 & 2 \\
			\bottomrule
	\end{tabular*}
	\begin{tabular*}{\textwidth}{@{\extracolsep{\fill}}cccccccccc}
			 & \multicolumn{4}{c}{Alg. \ref{alg:rcJacobi} } & \multicolumn{4}{c}{Alg. \ref{alg:evd-main} } \\
			\cline{2-5} \cline{6-10}
			$({\tt mode},{\tt kappa})$ & {\tt Res.} & {\tt OR-P.} & {\tt CT-J.} & {\tt SP.}  &  {\tt Res.} & {\tt OR-P.} & {\tt CT-Pre.} & {\tt CT-J.} & {\tt SP.} \\ \midrule
			$(3,10^3)$ & 1.44e-14 & 1.40e-14 & 4977.08 & 15 & 6.02e-15 & 6.35e-15 & 35.95 & 902.30 & 3 \\
			$(4,10^3)$ & 1.50e-14 & 1.36e-14 & 4851.85 & 12 & 4.90e-15 & 6.33e-15 & 34.34 & 928.37 & 2 \\
			$(5,10^3)$ & 1.46e-14 & 1.41e-14 & 5285.85 & 15 & 5.95e-15 & 6.35e-15 & 35.88 & 1027.24 & 3 \\
			\hline
			$(3,10^4)$ & 1.48e-14 & 1.41e-14 & 5695.27 & 16 & 6.13e-15 & 6.40e-15 & 47.86 & 995.31 & 4 \\
			$(4,10^4)$ & 1.50e-14 & 1.36e-14 & 4958.71 & 12 & 4.81e-15 & 6.33e-15 & 33.99 & 966.36 & 3 \\
			$(5,10^4)$ & 1.47e-14 & 1.43e-14 & 5365.20 & 17 & 6.31e-15 & 6.40e-15 & 35.22 & 988.31 & 4 \\
			\hline
			$(3,10^5)$ & 1.45e-14 & 1.43e-14 & 5472.13 & 17 & 6.39e-15 & 6.48e-15 & 34.87 & 1039.77 & 4 \\
			$(4,10^5)$ & 1.53e-14 & 1.36e-14 & 4893.77 & 12 & 4.49e-15 & 6.32e-15 & 34.99 & 956.58 & 2 \\
			$(5,10^5)$ & 1.46e-14 & 1.44e-14 & 5614.83 & 18 & 6.45e-15 & 6.49e-15 & 34.46 & 1086.99 & 4 \\
			\hline
			$(3,10^6)$ & 1.45e-14 & 1.46e-14 & 5907.33 & 18 & 6.59e-15 & 6.62e-15 & 34.87 & 1086.11 & 5 \\
			$(4,10^6)$ & 1.50e-14 & 1.36e-14 & 5048.47 & 12 & 4.42e-15 & 6.33e-15 & 35.06 & 941.44 & 2 \\
			$(5,10^6)$ & 1.45e-14 & 1.46e-14 & 5886.09 & 19 & 6.62e-15 & 6.61e-15 & 34.09 & 1159.67 & 5 \\
			\bottomrule
	\end{tabular*}
	}
\end{table}

The numerical results for Example \ref{ex2} with different $n$ are reported in Tables \ref{tab:JacobiEVDmultiple-GPU2048}--\ref{tab:JacobiEVDmultiple-GPU4096}, where the values of $\|\off(Q^T A Q)\|_F$ and $d(A)/4$ were calculated under CPU environment.
We observe from Tables \ref{tab:JacobiEVDmultiple-GPU2048}--\ref{tab:JacobiEVDmultiple-GPU4096} that parallel Algorithm \ref{alg:evd-main} is much more efficient than Algorithm \ref{alg:evd-main} in terms of the total computing time.

\begin{sidewaystable}
	\caption{Numerical results for Example \ref{ex2} with $n=2048$.} \label{tab:JacobiEVDmultiple-GPU2048} \vskip 0.2mm
	\centering
	\footnotesize{
		\begin{tabular*}{\textwidth}{@{\extracolsep{\fill}}ccccccccccccc}
			& \multicolumn{2}{c}{} & \multicolumn{5}{c}{parallel Alg. \ref{alg:evd-main} } & \multicolumn{5}{c}{Alg. \ref{alg:evd-main} } \\
			\cline{2-3} \cline{4-8} \cline{9-13}
			$({\tt mode},{\tt kappa})$ & $\| \off(Q^T A Q)\|_F$ &$d(A)/4$ &  {\tt Res.} & {\tt OR-P.}  &  {\tt CT-Pre.} & {\tt CT-J.} & {\tt SP.} &  {\tt Res.} & {\tt OR-P.}  &  {\tt CT-Pre.} & {\tt CT-J.} & {\tt SP.} \\ \midrule
			$(1,10^3)$ & 1.97e-06 & 2.50e-01 & 2.52e-14 & 4.82e-15 & 0.13 & 0.07 & 3 & 3.07e-15 & 1.93e-15 & 1.93 & 8.08 & 1 \\
			$(2,10^3)$ & 7.14e-07 & 2.50e-01 & 1.50e-14 & 1.50e-14 & 0.17 & 0.14 & 6 & 4.51e-15 & 1.33e-15 & 1.84 & 0.27 & 1 \\
			$(3,10^3)$ & 2.75e-05 & 3.40e-06 & 3.35e-13 & 3.12e-13 & 0.18 & 0.30 & 6 & 5.02e-15 & 4.48e-15 & 2.46 & 71.36 & 3 \\
			$(4,10^3)$ & 7.19e-05 & 4.89e-04 & 3.11e-13 & 3.04e-13 & 0.19 & 0.24 & 6 & 3.78e-15 & 4.48e-15 & 2.40 & 70.46 & 2 \\
			$(5,10^3)$ & 2.39e-05 & 1.02e-07 & 3.45e-13 & 3.21e-13 & 0.18 & 0.28 & 6 & 5.02e-15 & 4.48e-15 & 2.31 & 70.45 & 3 \\
			\hline
			$(1,10^4)$ & 2.87e-06 & 2.50e-01 & 1.81e-14 & 5.11e-15 & 0.17 & 0.08 & 3 & 3.47e-15 & 3.17e-15 & 2.24 & 31.55 & 1 \\
			$(2,10^4)$ & 6.15e-07 & 2.50e-01 & 7.95e-15 & 7.64e-15 & 0.18 & 0.13 & 5 & 4.56e-15 & 1.32e-15 & 1.82 & 0.27 & 1 \\
			$(3,10^4)$ & 2.23e-05 & 4.55e-07 & 3.39e-13 & 3.39e-13 & 0.19 & 0.28 & 6 & 5.04e-15 & 4.50e-15 & 2.68 & 71.75 & 3 \\
			$(4,10^4)$ & 7.21e-05 & 4.89e-04 & 3.11e-13 & 3.05e-13 & 0.18 & 0.25 & 6 & 3.70e-15 & 4.48e-15 & 2.51 & 70.15 & 2 \\
			$(5,10^4)$ & 2.16e-05 & 1.64e-08 & 3.36e-13 & 3.45e-13 & 0.19 & 0.28 & 6 & 5.14e-15 & 4.51e-15 & 2.33 & 72.04 & 3 \\
			\hline
			$(1,10^5)$ & 3.40e-06 & 2.50e-01 & 2.68e-14 & 5.67e-15 & 0.18 & 0.07 & 3 & 3.62e-15 & 3.33e-15 & 2.43 & 37.14 & 1 \\
			$(2,10^5)$ & 5.19e-07 & 2.50e-01 & 1.27e-14 & 1.23e-14 & 0.17 & 0.14 & 6 & 4.61e-15 & 1.30e-15 & 1.96 & 0.22 & 1 \\
			$(3,10^5)$ & 1.87e-05 & 5.70e-08 & 3.48e-13 & 4.11e-13 & 0.19 & 0.30 & 7 & 5.10e-15 & 4.56e-15 & 2.42 & 73.50 & 3 \\
			$(4,10^5)$ & 7.36e-05 & 4.89e-04 & 3.12e-13 & 3.04e-13 & 0.19 & 0.24 & 6 & 4.48e-15 & 4.48e-15 & 2.40 & 70.82 & 2 \\
			$(5,10^5)$ & 1.77e-05 & 2.00e-09 & 3.44e-13 & 4.15e-13 & 0.18 & 0.30 & 7 & 5.09e-15 & 4.55e-15 & 2.19 & 73.11 & 3 \\
			\hline
			$(1,10^6)$ & 3.48e-06 & 2.50e-01 & 2.34e-14 & 5.06e-15 & 0.18 & 0.05 & 2 & 3.55e-15 & 3.35e-15 & 2.35 & 35.74 & 1 \\
			$(2,10^6)$ & 3.96e-07 & 2.50e-01 & 1.56e-14 & 1.56e-14 & 0.18 & 0.15 & 6 & 4.59e-15 & 1.30e-15 & 1.80 & 0.27 & 1 \\
			$(3,10^6)$ & 1.60e-05 & 6.85e-09 & 3.62e-13 & 5.13e-13 & 0.19 & 0.35 & 8 & 5.24e-15 & 4.63e-15 & 2.26 & 76.06 & 3 \\
			$(4,10^6)$ & 7.33e-05 & 4.89e-04 & 3.12e-13 & 3.05e-13 & 0.19 & 0.24 & 6 & 4.46e-15 & 4.48e-15 & 2.44 & 70.76 & 2 \\
			$(5,10^6)$ & 1.40e-05 & 1.51e-10 & 3.58e-13 & 5.13e-13 & 0.19 & 0.34 & 8 & 5.20e-15 & 4.63e-15 & 2.09 & 76.08 & 5 \\
			\bottomrule
	\end{tabular*}
	}
\end{sidewaystable}

\begin{sidewaystable}
	\caption{Numerical results for Example \ref{ex2} with  $n=4096$.} \label{tab:JacobiEVDmultiple-GPU4096} \vskip 0.2mm
	\centering
	\footnotesize{
		\begin{tabular*}{\textwidth}{@{\extracolsep{\fill}}ccccccccccccc}
			& \multicolumn{2}{c}{} & \multicolumn{5}{c}{parallel Alg. \ref{alg:evd-main} } & \multicolumn{5}{c}{Alg. \ref{alg:evd-main} } \\
			\cline{2-3} \cline{4-8} \cline{9-13}
			$({\tt mode},{\tt kappa})$ & $\| \off(Q^T A Q)\|_F$ &$d(A)/4$ &  {\tt Res.} & {\tt OR-P.}  &  {\tt CT-Pre.} & {\tt CT-J.} & {\tt SP.} &  {\tt Res.} & {\tt OR-P.}  &  {\tt CT-Pre.} & {\tt CT-J.} & {\tt SP.} \\ \midrule
			$(1,10^3)$ & 2.31e-06 & 2.50e-01 & 5.41e-14 & 5.94e-15 & 0.79 & 0.33 & 1 & 4.04e-15 & 1.99e-15 & 35.72 & 23.04 & 1 \\
			$(2,10^3)$ & 5.85e-07 & 2.50e-01 & 1.68e-14 & 1.72e-14 & 0.80 & 0.56 & 2 & 5.89e-15 & 1.71e-15 & 43.40 & 1.97 & 1 \\
			$(3,10^3)$ & 5.65e-05 & 1.69e-06 & 6.93e-13 & 6.91e-13 & 0.81 & 1.82 & 4 & 6.50e-15 & 6.33e-15 & 35.89 & 1040.80 & 3 \\
			$(4,10^3)$ & 1.44e-04 & 2.44e-04 & 6.20e-13 & 6.15e-13 & 0.82 & 1.74 & 3 & 4.83e-15 & 6.32e-15 & 36.15 & 1039.27 & 2 \\
			$(5,10^3)$ & 5.56e-05 & 2.32e-08 & 6.95e-13 & 7.11e-13 & 0.79 & 1.78 & 3 & 6.52e-15 & 6.33e-15 & 42.69 & 1033.98 & 3 \\
			\hline
			$(1,10^4)$ & 3.30e-06 & 2.50e-01 & 5.70e-14 & 5.99e-15 & 0.80 & 0.28 & 1 & 4.67e-15 & 4.17e-15 & 34.12 & 403.45 & 1 \\
			$(2,10^4)$ & 7.41e-07 & 2.50e-01 & 1.95e-14 & 2.02e-14 & 0.81 & 0.71 & 2 & 5.96e-15 & 1.71e-15 & 32.30 & 1.95 & 1 \\
			$(3,10^4)$ & 4.53e-05 & 2.26e-07 & 7.14e-13 & 7.51e-13 & 0.79 & 1.68 & 3 & 6.74e-15 & 6.37e-15 & 35.62 & 1021.47 & 3 \\
			$(4,10^4)$ & 1.44e-04 & 2.44e-04 & 6.15e-13 & 6.10e-13 & 0.80 & 1.75 & 3 & 4.86e-15 & 6.33e-15 & 36.14 & 1012.79 & 2 \\
			$(5,10^4)$ & 4.14e-05 & 1.93e-09 & 7.22e-13 & 7.72e-13 & 0.81 & 1.78 & 3 & 6.77e-15 & 6.37e-15 & 34.37 & 1037.00 & 3 \\
			\hline
			$(1,10^5)$ & 4.30e-06 & 2.50e-01 & 5.64e-14 & 5.88e-15 & 0.79 & 0.31 & 1 & 4.91e-15 & 4.62e-15 & 43.51 & 561.86 & 1 \\
			$(2,10^5)$ & 5.11e-07 & 2.50e-01 & 1.58e-14 & 1.50e-14 & 0.80 & 0.57 & 2 & 5.92e-15 & 1.71e-15 & 58.46 & 1.96 & 1 \\
			$(3,10^5)$ & 3.65e-05 & 2.83e-08 & 7.13e-13 & 8.68e-13 & 0.81 & 1.74 & 3 & 6.87e-15 & 6.45e-15 & 38.55 & 1025.13 & 5 \\
			$(4,10^5)$ & 1.47e-04 & 2.44e-04 & 6.15e-13 & 6.12e-13 & 0.81 & 1.71 & 3 & 5.66e-15 & 6.32e-15 & 61.18 & 950.80 & 2 \\
			$(5,10^5)$ & 3.42e-05 & 1.06e-10 & 7.04e-13 & 8.63e-13 & 0.82 & 1.77 & 3 & 6.92e-15 & 6.46e-15 & 59.68 & 1024.97 & 6 \\
			\hline
			$(1,10^6)$ & 5.16e-06 & 2.50e-01 & 4.78e-14 & 5.97e-15 & 0.79 & 0.31 & 1 & 4.95e-15 & 4.67e-15 & 59.52 & 480.34 & 1 \\
			$(2,10^6)$ & 6.22e-07 & 2.50e-01 & 1.90e-14 & 1.99e-14 & 0.77 & 0.68 & 2 & 5.89e-15 & 1.71e-15 & 55.73 & 1.93 & 1 \\
			$(3,10^6)$ & 3.11e-05 & 3.40e-09 & 7.36e-13 & 1.08e-12 & 0.80 & 1.77 & 3 & 7.01e-15 & 6.59e-15 & 35.89 & 1112.71 & 5 \\
			$(4,10^6)$ & 1.46e-04 & 2.44e-04 & 6.22e-13 & 6.17e-13 & 0.82 & 1.73 & 3 & 5.74e-15 & 6.32e-15 & 62.06 & 1014.28 & 2 \\
			$(5,10^6)$ & 2.83e-05 & 3.18e-11 & 7.53e-13 & 1.17e-12 & 0.79 & 1.88 & 3 & 6.98e-15 & 6.62e-15 & 50.38 & 1185.05 & 7 \\
			\bottomrule
	\end{tabular*}
	}
\end{sidewaystable}



\subsection{The singular value problem}
In this subsection, we compare Algorithms \ref{alg:rcJacobi-svd}--\ref{alg:svd-main} for computing the SVD of a real matrix. We consider the following example.
\begin{example} \label{ex3}
Let $A$ be an $m\times n$ random matrix with pre-assigned singular value generated by  {\tt MATLAB 2022a}'s {\tt gallery ('randsvd', [m,n], kappa, mode)} with $\kappa(A)$ $={\tt kappa}$. We report our numerical results for {\rm(a)} ${\tt mode}=1$: the large singular value is equal to $1$ and the rest  of the singular values are equal to
$1/{\tt kappa}$, {\rm(b)} ${\tt mode}=2$: the small singular value is equal to $1/{\tt kappa}$ and the rest  of the singular values are equal to  $1$,  {\rm(c)} ${\tt mode}=3$: geometrically distributed singular values, {\rm(d)} ${\tt mode}=4$: arithmetically distributed singular values, and {\rm(e)} ${\tt mode}=5$: random singular values with uniformly distributed logarithm.
\end{example}

The numerical results for Example \ref{ex3} are reported in Table \ref{tab:JacobiSVD}. We see from Table \ref{tab:JacobiSVD} that   Algorithm \ref{alg:svd-main} works more efficient than Algorithm \ref{alg:rcJacobi-svd} for the cases ${\tt mode}=3,4,5$.

\begin{sidewaystable}
	\caption{Numerical results for Example \ref{ex3} with $m\times n=2048\times 1024$.}\label{tab:JacobiSVD} \vskip 0.2mm
	\centering
	\small{
		\begin{tabular*}{\textwidth}{@{\extracolsep{\fill}}ccccccccccccc}
			Rule
					& \multicolumn{6}{c}{ Alg. \ref{alg:rcJacobi-svd} } & \multicolumn{6}{c}{Alg. \ref{alg:svd-main}} \\
					\cline{2-7} \cline{8-13}
			$({\tt mode},{\tt kappa})$ &  {\tt Res.} & {\tt OR-U.} &  {\tt OR-V.}  & {\tt JU.} &  {\tt CT.}  &  {\tt SP.}  &  {\tt Res.} & {\tt OR-U.}  &  {\tt OR-V.}  &  {\tt JU.} &  {\tt CT.} &  {\tt SP.}  \\ \midrule
	        $(1,10^3)$ & 1.26e-15 & 1.12e-14 & 4.71e-16 & 0.011N & 4.25 & 4 & 1.82e-15 & 1.78e-14 & 1.19e-15 & 0.006N & 8.98 & 3 \\
			$(2,10^3)$ & 1.34e-16 & 2.35e-15 & 1.06e-16 & 0.004N & 2.13 & 2 & 8.96e-16 & 2.56e-15 & 1.15e-15 & 0.002N & 6.81 & 1 \\
			$(3,10^3)$ & 6.48e-15 & 2.36e-14 & 7.01e-15 & 10.501N & 71.07 & 18 & 2.81e-15 & 2.66e-14 & 3.04e-15 & 1.737N & 19.72 & 5 \\
			$(4,10^3)$ & 6.47e-15 & 2.08e-14 & 6.58e-15 & 9.262N & 61.67 & 15 & 3.08e-15 & 1.87e-14 & 3.19e-15 & 1.905N & 20.50 & 5 \\
			$(5,10^3)$ & 6.51e-15 & 2.33e-14 & 7.03e-15 & 10.642N & 72.66 & 19 & 2.81e-15 & 2.68e-14 & 3.04e-15 & 1.719N & 19.61 & 5 \\
			\hline
			$(1,10^4)$ & 1.45e-15 & 3.58e-14 & 1.03e-15 & 0.140N & 4.91 & 4 & 2.29e-15 & 4.18e-14 & 2.20e-15 & 0.717N & 17.84 & 8 \\
			$(2,10^4)$ & 1.41e-16 & 2.28e-15 & 1.13e-16 & 0.004N & 2.06 & 2 & 8.99e-16 & 2.55e-15 & 1.15e-15 & 0.003N & 6.77 & 1 \\
			$(3,10^4)$ & 6.56e-15 & 2.57e-14 & 7.24e-15 & 11.291N & 78.41 & 21 & 2.75e-15 & 2.57e-14 & 3.02e-15 & 1.729N & 19.76 & 5 \\
			$(4,10^4)$ & 6.49e-15 & 2.08e-14 & 6.59e-15 & 9.316N & 63.23 & 16 & 3.08e-15 & 1.85e-14 & 3.18e-15 & 1.909N & 20.62 & 5 \\
			$(5,10^4)$ & 6.61e-15 & 2.52e-14 & 7.34e-15 & 11.570N & 79.64 & 21 & 2.75e-15 & 2.49e-14 & 3.03e-15 & 1.740N & 19.54 & 5 \\
			\hline
			$(1,10^5)$ & 2.34e-15 & 4.33e-14 & 2.88e-15 & 1.652N & 15.99 & 8 & 3.13e-15 & 4.44e-14 & 3.85e-15 & 2.879N & 28.65 & 9 \\
			$(2,10^5)$ & 1.33e-16 & 2.27e-15 & 1.10e-16 & 0.004N & 1.98 & 2 & 8.98e-16 & 2.54e-15 & 1.15e-15 & 0.003N & 6.56 & 1 \\
			$(3,10^5)$ & 6.65e-15 & 2.77e-14 & 7.53e-15 & 12.376N & 83.99 & 24 & 2.73e-15 & 2.52e-14 & 3.01e-15 & 1.785N & 20.54 & 6 \\
			$(4,10^5)$ & 6.52e-15 & 2.10e-14 & 6.62e-15 & 9.373N & 62.79 & 16 & 3.08e-15 & 1.85e-14 & 3.18e-15 & 1.909N & 20.33 & 5 \\
			$(5,10^5)$ & 6.64e-15 & 2.69e-14 & 7.51e-15 & 12.239N & 84.24 & 24 & 2.73e-15 & 2.47e-14 & 3.02e-15 & 1.787N & 19.22 & 5 \\
			\hline
			$(1,10^6)$ & 3.12e-15 & 4.32e-14 & 4.25e-15 & 3.760N & 26.93 & 9 & 3.66e-15 & 4.35e-14 & 4.71e-15 & 4.471N & 35.79 & 9 \\
			$(2,10^6)$ & 1.35e-16 & 2.29e-15 & 1.07e-16 & 0.004N & 1.96 & 2 & 8.98e-16 & 2.56e-15 & 1.16e-15 & 0.003N & 6.33 & 1 \\
			$(3,10^6)$ & 6.68e-15 & 2.95e-14 & 7.70e-15 & 13.002N & 88.17 & 26 & 2.71e-15 & 2.31e-14 & 3.03e-15 & 1.862N & 19.23 & 5 \\
			$(4,10^6)$ & 6.50e-15 & 2.04e-14 & 6.61e-15 & 9.364N & 60.09 & 16 & 3.08e-15 & 1.85e-14 & 3.18e-15 & 1.908N & 18.24 & 4 \\
			$(5,10^6)$ & 6.77e-15 & 2.94e-14 & 7.83e-15 & 13.426N & 93.02 & 27 & 2.72e-15 & 2.35e-14 & 3.01e-15 & 1.845N & 18.65 & 4 \\
			\bottomrule
	\end{tabular*}
	}
\end{sidewaystable}

\begin{sidewaystable}
	\caption{Numerical results for Example \ref{ex3} with $m\times n=4096\times 2048$.}\label{tab:JacobiSVD-4096} \vskip 0.2mm
	\centering
	\small{
		\begin{tabular*}{\textwidth}{@{\extracolsep{\fill}}ccccccccccccc}
			Rule
					& \multicolumn{6}{c}{ Alg. \ref{alg:rcJacobi-svd} } & \multicolumn{6}{c}{Alg. \ref{alg:svd-main}} \\
					\cline{2-7} \cline{8-13}
			$({\tt mode},{\tt kappa})$ &  {\tt Res.} & {\tt OR-U.} &  {\tt OR-V.}  & {\tt JU.} &  {\tt CT.}  &  {\tt SP.}  &  {\tt Res.} & {\tt OR-U.}  &  {\tt OR-V.}  &  {\tt JU.} &  {\tt CT.} &  {\tt SP.}  \\ \midrule
	        $(1,10^3)$ & 1.66e-15 & 1.40e-14 & 4.66e-16 & 0.004N & 48.28 & 5 & 2.31e-15 & 1.42e-14 & 1.55e-15 & 0.002N & 82.61 & 4 \\
			$(2,10^3)$ & 1.33e-16 & 3.36e-15 & 1.11e-16 & 0.002N & 22.83 & 2 & 9.21e-16 & 3.48e-15 & 1.49e-15 & 0.001N & 59.62 & 1 \\
			$(3,10^3)$ & 9.33e-15 & 4.67e-14 & 1.00e-14 & 10.978N & 944.57 & 19 & 3.89e-15 & 5.57e-14 & 4.26e-15 & 1.742N & 213.95 & 5 \\
			$(4,10^3)$ & 9.39e-15 & 4.26e-14 & 9.52e-15 & 9.831N & 792.14 & 16 & 4.25e-15 & 3.84e-14 & 4.43e-15 & 1.900N & 209.26 & 5 \\
			$(5,10^3)$ & 9.41e-15 & 4.82e-14 & 1.01e-14 & 11.112N & 803.27 & 19 & 3.87e-15 & 5.56e-14 & 4.24e-15 & 1.714N & 188.50 & 5 \\
			\hline
			$(1,10^4)$ & 1.70e-15 & 4.22e-14 & 8.08e-16 & 0.038N & 41.00 & 4 & 2.48e-15 & 7.68e-14 & 1.99e-15 & 0.160N & 144.13 & 9 \\
			$(2,10^4)$ & 1.27e-16 & 3.18e-15 & 1.11e-16 & 0.002N & 19.15 & 2 & 9.21e-16 & 3.49e-15 & 1.50e-15 & 0.001N & 55.17 & 1 \\
			$(3,10^4)$ & 9.41e-15 & 5.14e-14 & 1.04e-14 & 11.794N & 862.50 & 22 & 3.80e-15 & 5.43e-14 & 4.22e-15 & 1.745N & 187.00 & 5 \\
			$(4,10^4)$ & 9.37e-15 & 4.22e-14 & 9.51e-15 & 9.817N & 696.45 & 17 & 4.25e-15 & 3.81e-14 & 4.43e-15 & 1.903N & 214.96 & 5 \\
			$(5,10^4)$ & 9.45e-15 & 5.02e-14 & 1.04e-14 & 11.890N & 877.83 & 23 & 3.81e-15 & 5.34e-14 & 4.23e-15 & 1.754N & 187.01 & 5 \\
			\hline
			$(1,10^5)$ & 2.56e-15 & 8.09e-14 & 2.75e-15 & 0.758N & 136.91 & 9 & 3.79e-15 & 8.64e-14 & 4.54e-15 & 1.976N & 240.94 & 8 \\
			$(2,10^5)$ & 1.32e-16 & 3.16e-15 & 1.10e-16 & 0.002N & 19.27 & 2 & 9.23e-16 & 3.49e-15 & 1.49e-15 & 0.002N & 58.02 & 1 \\
			$(3,10^5)$ & 9.55e-15 & 5.61e-14 & 1.08e-14 & 12.759N & 961.53 & 25 & 3.78e-15 & 4.76e-14 & 4.24e-15 & 1.833N & 182.91 & 4 \\
			$(4,10^5)$ & 9.38e-15 & 4.28e-14 & 9.50e-15 & 9.827N & 690.15 & 16 & 4.25e-15 & 3.81e-14 & 4.43e-15 & 1.904N & 197.22 & 5 \\
			$(5,10^5)$ & 9.51e-15 & 5.42e-14 & 1.07e-14 & 12.522N & 926.41 & 25 & 3.79e-15 & 4.82e-14 & 4.24e-15 & 1.831N & 189.47 & 5 \\
			\hline
			$(1,10^6)$ & 4.06e-15 & 8.67e-14 & 5.31e-15 & 2.979N & 257.62 & 9 & 4.84e-15 & 8.49e-14 & 6.23e-15 & 3.963N & 364.17 & 10 \\
			$(2,10^6)$ & 1.28e-16 & 3.16e-15 & 1.12e-16 & 0.002N & 19.19 & 2 & 9.24e-16 & 3.49e-15 & 1.49e-15 & 0.002N & 55.41 & 1 \\
			$(3,10^6)$ & 9.68e-15 & 6.03e-14 & 1.12e-14 & 13.947N & 1066.63 & 29 & 3.77e-15 & 4.39e-14 & 4.25e-15 & 1.901N & 186.61 & 4 \\
			$(4,10^6)$ & 9.38e-15 & 4.22e-14 & 9.50e-15 & 9.825N & 706.85 & 16 & 4.26e-15 & 3.81e-14 & 4.43e-15 & 1.904N & 202.13 & 5 \\
			$(5,10^6)$ & 9.75e-15 & 5.94e-14 & 1.12e-14 & 13.905N & 1092.29 & 30 & 3.76e-15 & 4.37e-14 & 4.25e-15 & 1.892N & 186.13 & 4 \\
			\bottomrule
	\end{tabular*}
	}
\end{sidewaystable}

\subsection{Numerical verification} \label{sec63}
In this subsection, we will investigate the quantity $\zeta$ in Theorem \ref{thm:qc-rc}.
We note that $(1+2\tilde{\gamma}_4)^n \approx 1$ for $\tilde{\gamma}_4= (1-u)^{-4w}-1$. One may expect that
\[
\zeta_r\approx \tilde{\zeta}_r:= |\tilde{\zeta}_{r1} + \tilde{\zeta}_{r2}| / \mbox{ $\sum_{q=2}^{n-r+1}|\hat{a}_{r,r-1+q}^{(\iota(r-1)+n-r)}|^2$},
\]
where $\tilde{\zeta}_{r1} =\sum_{q=1}^{n-r-1}|\hat{a}_{r,r+1}^{(\iota(r-1)+n-r-1+q)}|^2$ and
\[
\begin{array}{lcl}
\tilde{\zeta}_{r2} &=& \sum_{j=3}^{n-r}\sum_{q=1}^{n-r+1-j}|\hat{a}_{r,r-1+j}^{(\iota(r-1)+(n-r+1)(j-1)-(j-1)j/2+q-1)}|^2\\
& & - \sum_{j=3}^{n-r}\sum_{q=2}^{j-1} |\hat{a}_{r,r-1+q}^{(\iota(r-1)+(n-r+1)(j-1)-(j-1)j/2)}|^2.
\end{array}
\]
Figure {\rm\ref{fig:approx-zeta}} describes the quantity $\tilde{\zeta} = \max_{1\le r\le n-1} \tilde{\zeta}_r$ versus $j$ (the number of sweeps) in double precision for some test matrices of Example \ref{ex1} with different choices of $(\tt{mode}, \tt{kappa})$. We see from Figure {\rm\ref{fig:approx-zeta}} that $\tilde{\zeta}$ and $\zeta$ are of order of $n$. Thus one may expect that the quantities $\tilde{\zeta} \tilde{\gamma}_4$ and $\zeta\tilde{\gamma}_4$ are  not too large since $\tilde{\gamma}_4$ is small enough.

\begin{figure}[!htb]
\centering
\includegraphics[width=0.49\textwidth,height=0.25\textheight]{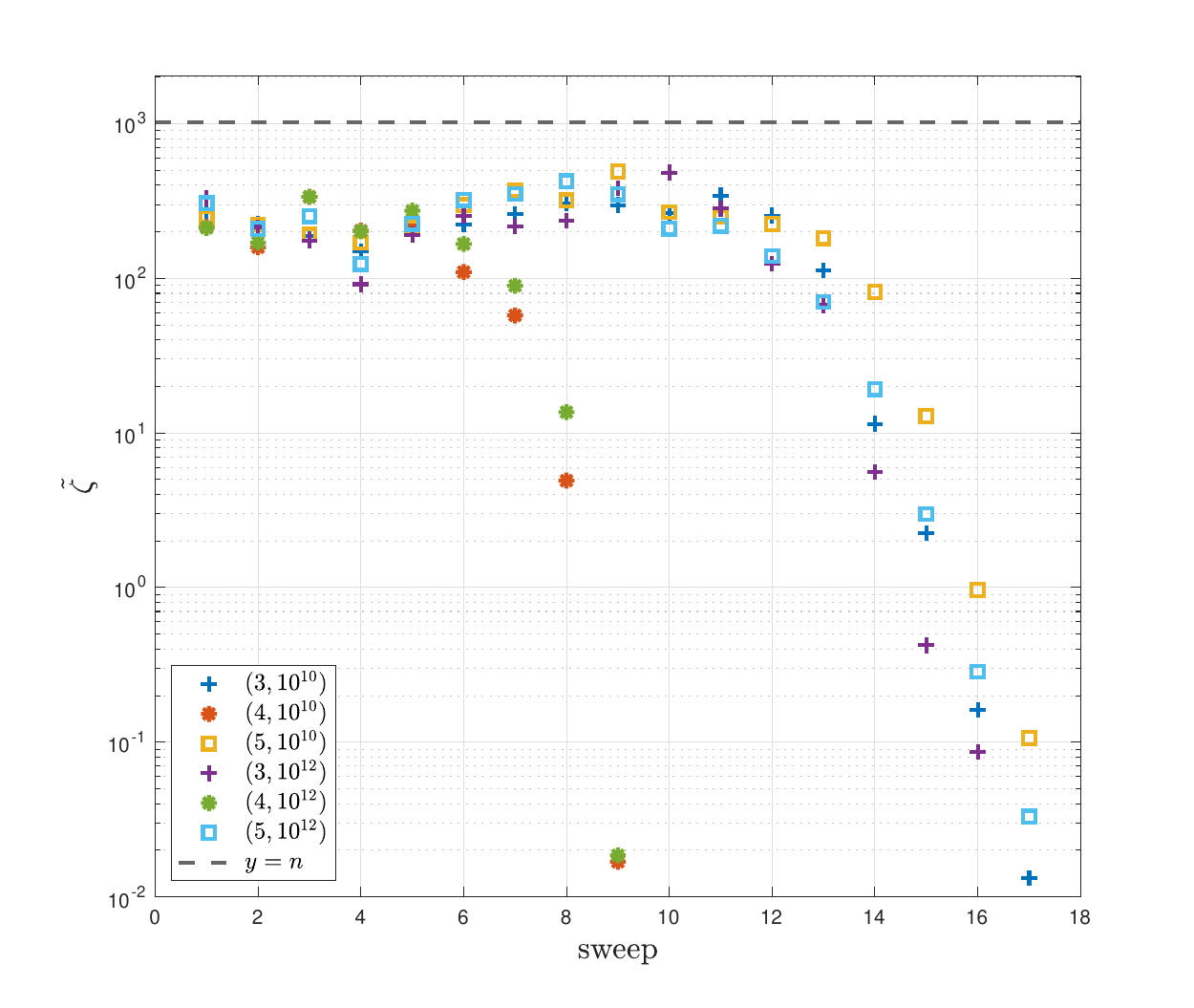}
\includegraphics[width=0.49\textwidth,height=0.25\textheight]{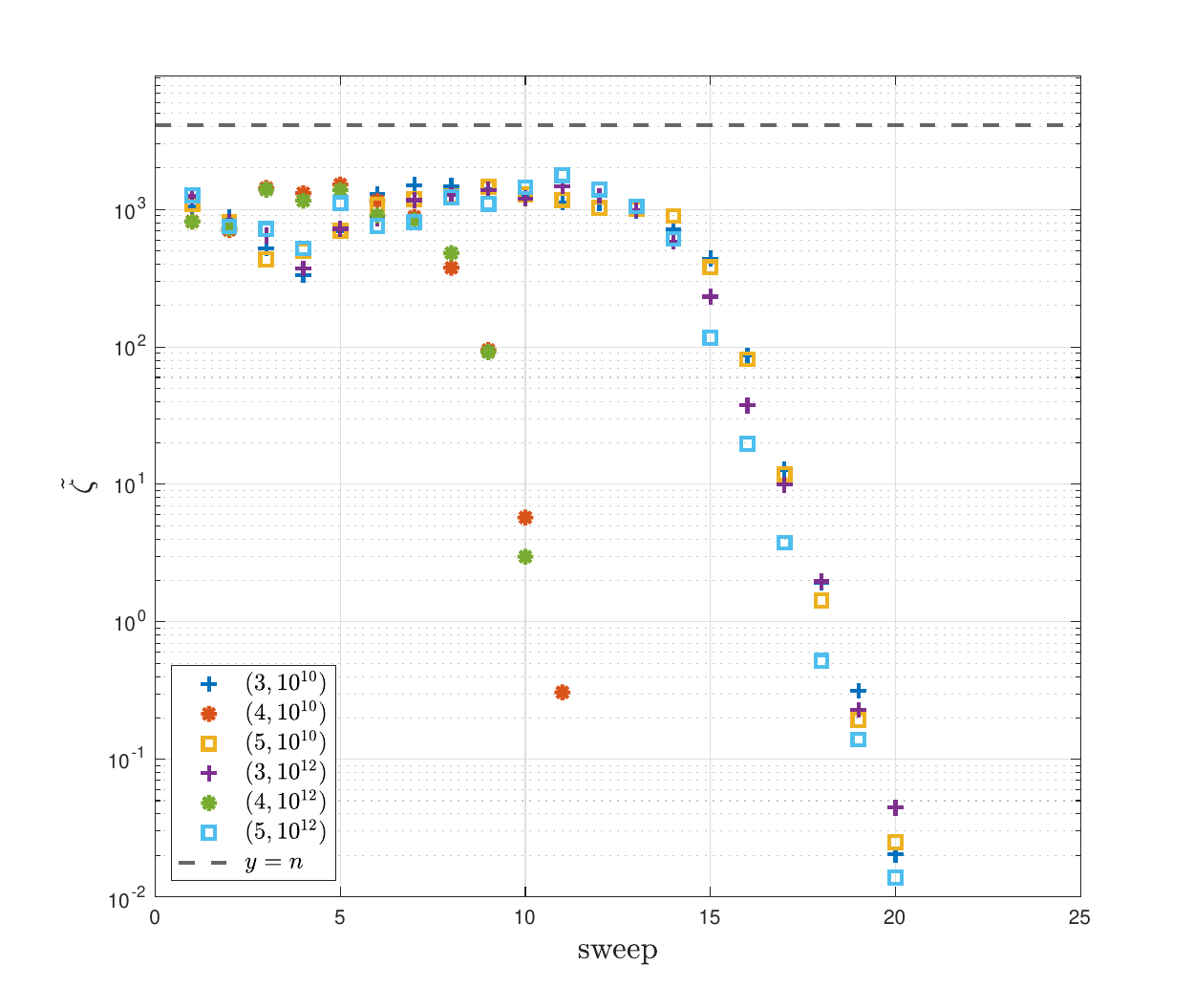}
\caption{$\tilde{\zeta}$ versus $j$ (sweeps) in double precision for  Example \ref{ex1} with different choices of $(\tt{mode}, \tt{kappa})$. Left: $n=1024$ and right: $n=4096$.} \label{fig:approx-zeta}
\end{figure}

%

\section{Conclusions} \label{Sec:Con}
In this paper, we give the error analysis for a single step or sweep of the Jacobi method in floating point arithmetic. Then we propose a mixed precision  preconditioned Jacobi method for computing the eigenvalue decomposition of  a real symmetric matrix and  a mixed precision preconditioned one-sided Jacobi method for the singular value problem. The corresponding rounding error analysis is studied. Our numerical experiments show the efficiency of the proposed mixed precision Jacobi method over the classical Jacobi method. Moreover, our algorithms can achieve higher speedup on GPUs. 
An interesting question is how to develop a mixed precision method for the generalized eigenvalue problem. This needs further study.



\vspace{2mm}
{\bf Conflict of Interest Statement} The authors declare that they have no conflict of interest.

{\bf Data Availability Statement} All data generated or analysed during this study are included in this
manuscript.

\end{document}